\documentclass[12pt,dvipsnames]{article}

\usepackage{authblk} 
\usepackage{tikz,amsmath,amssymb,color,amsthm,eufrak}
\usetikzlibrary{calc,decorations}
\usepackage{pgf}
\usepackage{paralist}
\usepackage[margin=2cm]{geometry}
\usepackage[backend=bibtex,doi=false,isbn=false,url=false,eprint=false]{biblatex}
\usepackage{ytableau}
\usepackage{amsthm}
\usepackage{mathrsfs}

\newcommand{\ttre}{\textcolor{red}}

\theoremstyle{plain}
\newtheorem{theorem}{Theorem}[section]
\newtheorem*{theorem*}{Theorem}

\newtheorem{lemma}[theorem]{Lemma}
\newtheorem{claim}[theorem]{Claim}
\newtheorem{remark}{Remark}
\newtheorem{corollary}[theorem]{Corollary}
\newtheorem{proposition}[theorem]{Proposition}

\newtheorem{condition}[theorem]{Condition}
\theoremstyle{definition} 
\newtheorem{example}[theorem]{Example}
\newtheorem{definition}[theorem]{Definition}

\newcommand{\tp}[1]{{\mathfrak L}{(#1)}} 
\newcommand{\lolli}[2]{L_{#1,#2}} 
\newcommand{\cshuf}[2]{\mathcal{CS}(#1,#2)} 
\newcommand{\shuf}[2]{\cshuf{ #1}{#2}} 
\newcommand{\lc}[1]{\mathcal{LMC}(#1)} 
\newcommand{\smc}[1]{\mathcal{SMC}(#1)} 
\newcommand{\stab}[2]{\mathcal{T}(#1,#2)} 
\newcommand{\redw}[1]{\mathcal{R}(#1)} 
\newcommand{\shredw}[1]{\mathcal{SR}(#1)} 
\newcommand{\balanced}[1]{\text{Bal}(#1)}

\DeclareMathOperator\remmult{mult}
\newcommand{\mult}[3]{\remmult_{#1}(#2,#3)}

\DeclareMathOperator\EG{EG} 
\DeclareMathOperator{\EGQ}{EG_Q} 

\usepackage[colorinlistoftodos]{todonotes}
\newcommand{\sam}[1]{\todo[size=\tiny,inline,color=blue!30]{#1
      \\ \hfill --- Samantha}}
\newcommand{\susanna}[1]{\todo[size=\tiny,inline,color=red!30]{#1
    \\ \hfill --- Susanna}}

\usepackage{tabu}

\definecolor{vividviolet}{rgb}{0.62, 0.0, 1.0}
\newcommand{\tvi}[1]{\todo[size=\tiny,inline,color=vividviolet!30]{#1
      \\ \hfill --- Samantha-tvi}}
\newcommand{\tre}[1]{\todo[size=\tiny,inline,color=red]{#1
      \\ \hfill --- Samantha-tre}}

\def\S{\mathfrak{S}}
\def\Sn{\S_n}
\def\Tam{\mathfrak{T}} 

\DeclareMathOperator\INV{Inv}
\def\R{\mathbb R}

\newcommand{\omitt}[1]{}
\newcommand{\mydef}[1]{\emph{#1}}

\newcommand{\fS}{{\mathfrak S}}
\newcommand{\Inv}{\text{Inv}}
\usepackage{hyperref}
\hypersetup{
	colorlinks=true,
	linkcolor=blue,
        citecolor=blue}
\usetikzlibrary{shapes,arrows,calc}
\usepackage{pgf,pgflibraryshapes}
\usepackage{young}

\newcommand{\cR}{\mathcal{R}}
\newcommand{\cH}{\mathcal{H}}

\newcommand{\cF}{\mathcal{F}}

\newcommand{\LMC}[1]{\mathcal{LMC}({#1})} 
\newcommand{\LMR}[1]{\mathcal{LMR}({#1})} 

\newcommand{\SMC}[1]{\smc{#1}} 
\newcommand{\SMB}[1]{\mathcal{SMB}({#1})} 
\newcommand{\SMR}[1]{\mathcal{SMR}({#1})} 
 
\newcommand{\MB}[1]{\mathcal{MB}({#1})} 
\newcommand{\MR}[1]{\mathcal{MR}({#1})} 
\newcommand{\MH}[1]{\mathcal{MH}({#1})} 

\newcommand{\LMF}[1]{\mathscr{L\hspace{-.2cm} M\hspace{-.15cm} F}({#1})}

\newcommand{\SYT}{\text{SYT}}
\newcommand{\ii}{\text{id}}

\newcommand{\stair}{\mathfrak{s}} 

\DeclareMathOperator{\inv}{inv}

\def\longest{\omega_0}
\newcommand{\longestk}[1]{\longest^{(#1)}}

\newcommand{\wch}[1]{w^{(#1)}}
\newcommand{\vch}[1]{v^{(#1)}}
\newcommand{\uch}[1]{u^{(#1)}}

\newcommand*\circled[1]{\tikz[baseline=(char.base)]{
            \node[shape=circle,draw,inner sep=2pt,thick,red] (char) {#1};}}
\newcommand*\smallcircled[1]{\textcolor{red}{\bf #1}}
\def\vmn{v_{m,n}}

\newcommand{\stacktab}[2]{[#1:#2]}

\DeclareMathOperator{\Des}{Des}
\usepackage{mathrsfs}
\DeclareMathOperator{\set}{set}

\title{Maximal chains in lattices from graph associahedra: Tamari to the weak order }
\author[1]{Samantha~Dahlberg}
\author[2]{Susanna~Fishel\footnote{Partially supported by the Simons collaboration grants for Mathematicians \#359602 and \#709671}}

\affil[1]{Department of Applied Mathematics\\
Illinois Institute of Technology\\
Chicago, IL 60616, U.S.A.\\
\href{}{dahlberg.samantha@gmail.com}}
\affil[2]{School of Mathematical and Statistical Sciences\\
Arizona State University\\
Tempe, AZ 85287, U.S.A.\\
\href{}{sfishel1@asu.edu}}

\begin{document}

\maketitle
\begin{abstract}
In this paper, we study the maximal chains of lattices which
generalizes both the weak order and the Tamari lattice: certain
lattices of maximal tubings. A maximal tubing poset $\mathfrak{L}(G)$
is defined for any graph $G$, but for the graphs we consider in this
paper, the poset is a lattice.  Just as the weak order is an
orientation of the $1$-skeleton of the permutahedron and the Tamari of
the associahedron, each tubing lattice is an orientation of the
$1$-skeleton of a graph associahedron. The partial order on
$\mathfrak{L}(G)$ is given by a projection from $\mathfrak{S}_n$ to
$\mathfrak{L}(G)$. In particular, when the graph is the complete
graph, the graph associahedron is the the permutahedron, and when it
is the path graph, it is the Stasheff associahedron.

Our main results are for lollipop graphs, graphs that ``interpolate''
between the path and the complete graphs. For lollipop graphs, the
lattices consist of permutations which satisfy a generalization of
$312$-avoiding. The maximum length chains correspond to partially shiftable
tableaux under the Edelman-Greene's Coxeter-Knuth bijection. We also
consider functions defined analogously to Stanley's symmetric function for the maximum length chains
and find their expansion in terms of Young quasisymmetric Schur
functions.
  
  \end{abstract}

\tableofcontents

\section{Introduction}

Reduced decompositions of the longest permutation, which can also be seen as maximal chains in the weak order lattice of permutations $\Sn$, are in bijection with standard Young tableaux of staircase shape \cite{stan84, EG}. This result, and the techniques used to prove it, touch on many areas of algebraic combinatorics.  A related result is that the maximum length chains in the Tamari lattice are in bijection with shifted standard Young tableaux of staircase shape \cite{FN,STWW}.

In this paper, we study the maximal chains of a lattice which generalizes both the weak order and the Tamari lattice: certain lattices of maximal tubings. A maximal tubing poset $\tp{G}$ is defined for any graph $G$, but for the graphs we consider in this paper, the poset is a lattice. We will refer to them as tubing lattices, leaving out maximal.  Just as the weak order is an orientation of the $1$-skeleton of the permutahedron and the Tamari of the associahedron, each tubing lattice is an orientation of the $1$-skeleton of a graph associahedron. The partial order on $\tp{G}$ is given by a projection from $\Sn$ to $\tp{G}$. In particular, when the graph is the complete graph, the graph associahedron is the the permutahedron, and when it is the path graph, it is the Stasheff associahedron \cite{postnikov}. Our main results are for graphs that ``interpolate'' between these two graphs--the lollipop graph $\lolli m n$ on vertices $\{1,2,\ldots,m+n\}$. The lollipop is the complete graph on vertices $\{1,\ldots,m\}$ and the path graph on vertices $\{m, m+1,\ldots,m+n\}$, where $m$ and $n$ are integers such that $m\ge2$ and $n\ge0$ . 

Lollipop graphs are a particularly nice type of graph. They are filled and the projection from the weak order on the symmetric group to $\tp{\lolli m n}$ is very well-behaved. The tubing lattice of the lollipop ``interpolates'' between the Tamari ($m=2$) and the weak order ($n=0$) 

Our first main result sets up a bijection between maximum length chains in the tubing lattice of the lollipop graph and certain standard Young tableaux of staircase shape. We denote the set of chains of maximum length in a poset $P$ by $\LMC{P}$ 
 
\begin{theorem*}[\ref{thm:3equiv}]
  The following are equivalent. Let $m$ and $n$ be positive integers, let $\rho\in\redw{\longestk{m+n}}$, and let $C$ be the chain of permutations which corresponds to $\rho$.
  \begin{enumerate}
  \item\label{thmpart:shuf} $\rho\in\shuf m n$
  \item \label{thmpart:shif}$T=\EGQ(\rho)$ is an $n$ row shiftable tableau of shape $(m+n)$-staircase.
  \item \label{thmpart:tp}$C\in\lc {\lolli m n}$
    \end{enumerate}
  
  \end{theorem*}

The theorem generalizes both the result on maximal chains in the weak order and maximum length chains in the Tamari. The subset of reduced words $\shuf m n$ in \eqref{thmpart:shif} generalizes lattices words. A lattice word (or permutation) of shape $\lambda$ is a sequence $a_1,a_2,\ldots,a_N$ in which $i$ occurs $\lambda_i$ times and such that in any left factor $a_1,a_2,\ldots,a_j$, the number of $i$'s is as least as great as the number of $i+1$'s \cite{S99}. The sequence is a reverse lattice word if $a_N,a_{N-1},\ldots,a_2,a_1$ is a lattice word. In \cite{FN}, the first author and Nelson showed that the reduced decompositions of longest permutations are both lattice words and reverse lattice words of shape $(n-1,n-2,\ldots,1)$. The set of commuting shuffles $\shuf m n$ of reduced decompositions is defined in Section~\ref{sec:mlcinlollis} and generalizes the set of reduced decompositions which are both lattice and reverse lattice words.  We analyze the Edelman-Greene \cite{EG} insertion to show that maximal length chains in the tubing lattice of the lollipop graph map to tableaux where some of the rows are shifted, the $n$-row-shiftable tableaux.  The corresponding balanced tableaux are discussed, although not mentioned in the theorem.

A lot of the work to prove Theorem~\ref{thm:3equiv} is done using pattern avoidance. We show that the elements of the tubing lattice for $\tp{\lolli m n}$ can be described by a generalization of $312$-pattern avoidance. 

In Stanley's proof that the number of reduced decompositions of the longest permutations is also the number of standard Young tableaux of staircase shape, he introduced what's come to be known as Stanley's symmetric function. His function is the sum, over reduced decompositions, of the fundamental quasisymmetric function indexed by the descent set of the reduced decomposition. Not only is the function symmetric, but it is Schur positive. \omitt{\susanna{Samantha, is this right? That is, is the Stanley symmetric function Schur positive? Feb. 7, 2024} 
\tre{Yes, Stanley's symmetric function is Schur-positive for all permutations. (Feb 13, 2024)}}
We define a function analogous to Stanley's, $\LMF G$, where $G$ is a connected, filled graph. The sum is over the reduced decompositions which correspond to the maximum length chains in $\tp{G}$. Our second main result is that $\LMF G$ is a positve sum of Young quasisymmetric Schur functions: \omitt{
\susanna{Positive or nonnegative?}

\tvi{A positive sum. However, $\LMF{G}$ is actually a restriction of Stanley's function to just the longest chains in $\tp{G}$ and we are only working with filled $G$. The generalization where we sum over ALL chains in $\tp{G}$, $\cF_G$, for filled $G$, is not Schur positive and not Young quasi Schur positive. (Feb 13, 2024)
}
\susanna{Better? Feb 28, 2024}
\tvi{Yes, better. However, we should specify that the function corresponds to longest maximal length chains. We have another function that corresponds to all maximal length chains that is not Schur positive in any sense. (2-29-24) }
\tvi{(sam to do) In red text add what to say. (4/4)}
\tvi{What's added looks good to me. (4/45)}
}

\begin{theorem*} 
[\ref{thm:Longest chain quasi Schur}]
For $m\geq 1$ and $n\geq 0$,  $\LMF{L_{m,n}}$ is a positive sum of Young quasisymmetric Schur functions. In particular
$$\LMF{L_{m,n}}=\sum_{\alpha\in \text{comp}(m,n)}\hat{\mathscr{S}}_{\alpha}$$
where $\text{comp}(m,n)$ is the set of compositions $(\alpha_1,\alpha_2,\ldots ,\alpha_{m+1},n,\cdots,2,1)$ where $[\alpha_1,\alpha_2,\ldots, \alpha_{m+1}]$ is a permutation of $[n+1,n+m-1]$.

\end{theorem*}

In the weak order, all maximal chains have the same length. When $G$ is not the complete graph, not all maximal chains in $\tp{G}$ have the same length. Our third main result concerns the set $\SMR{\lolli m n}$, which are in bijection with the shortest maximal chains in $\tp{\lolli m n}$.

\begin{theorem*}[Theorem~\ref{thm:R to SMR bijection}]
Let $m\geq 1$ and $n\geq 0$. The map $\psi:\cR(w_{m,n})\rightarrow \SMR{L_{m,n}}$ is a bijection and $\Des(\sigma)=\Des(\psi(\sigma))$ for all $\sigma\in\cR(w_{m,n})$. 

\end{theorem*}

In Section~\ref{sec:SMC} we define the permutation $w_{m,n}$, which depends on  integers $m\geq 1$ and $n\geq 0$, and then a bijection from the reduced decompositions of $w_{m,n}$ to the $\SMR{\lolli m n}$ which respects descents. Since Stanley's symmetric function for any permutation is Schur positive, this bijection 
together with Theorem~\ref{thm:Longest chain quasi Schur} hints at a potentially larger theory. In Section~\ref{sec:further} we define a generalization of Stanley's symmetric function that unifies the Young quasisymmetric positivity in Theorem~\ref{thm:Longest chain quasi Schur} with the implied Schur positivity from Theorem~\ref{thm:R to SMR bijection}. We view $\tp{G}$ for connected filled graphs with  vertices replaced by the permutations which are the maximum elements of their equivalence classes. Cover relations between permutations in this poset are adjacent cycles, which are a generalization of adjacent transpositions.
\omitt{\susanna{What should we add about Theorem 8.3? Feb 28, 2024}
\tvi{The interesting thing about this result is not that we have a bijection, but its implication regarding the Schur positivity of shortest maximal chains if only we could define a descent set for shortest maximal chains (which we do!). (4/4) }
\tvi{(sam to do) add a sentence to explain. (4/4)}
\tvi{A little adjustment in red for accuracy. (4/25)}}

Let us discuss a bit more of the history. The graph associahedron was introduced by Carr and Devadoss in \cite{CD}, who were looking at tilings of analogues of the real moduli space of curves. They defined tubes and tubing on the vertex set of a given graph and then a polytope based on the tubes of the graph. Davis, Januszkiewicz, and Scott in \cite{DJS} studied blow-ups of hyperplane arrangement  associated to a finite
reflection group $W$ where the blow-up locus is $W$-invariant, so that the resulting manifold 
admits a cell decomposition whose maximal cells are all combinatorially isomorphic to a given
convex polytope $P$. The polytopes they found were related to graph associahedra.  Postnikov defined graph permutahedra in \cite{postnikov}; graph associahedra are a special case. Ronco described a partial order on all tubings, which when restricted to the 1-skeleton of the graph associahedron generalizes the Tamari order in \cite{ronco2012}. See also \cite{ss}, by Son Nguyen and Andrew Sack.  Their generalized stack-sortable permutations, developed in the study of poset associahedra, appear to the same as our $m$-312 avoiding permutations.   

Our paper relies heavily on Barnard and McConville~\cite{BM}. We work with their projection from the weak order to the tubing lattice. The projection from the weak order to the poset of maximal tubings for a graph $G$ is a lattice quotient if and only if $G$ is filled \cite[Theorem 4.16]{BM}. They show it is a lattice if $G$ is either left-filled or right-filled.  The lollipop graphs are both, so the tubing posets we consider are indeed  lattices. Finally, we use their $G$-permutations, the minimal elements in the fibers of the projection. 

\omitt{
\susanna{Mention Nguyen Son and Andrew Sack here. Their generalized stack-sortable permutations, developed in the study of poset associahedra, appear to the same as our $m$-312 avoiding permutations.}
}
Background material and notation is in Section~\ref{sec:prelims}. In Section~\ref{sec:L_Gintro}, we define the tubing poset for a general graph $G$ and discuss the projection from $\Sn$ to $\tp{G}$. In Section~\ref{sec:Llollipop}, we specialize to the case $G=\lolli m n$ and introduce pattern avoidance. Our first main result is in Section~\ref{sec:mlcinlollis}, where we study maximum length chains in $\tp{\lolli m n}$ and introduce commuting shuffles, a subset of reduced decompositions, and $n$-row-shifted tableaux. We turn to quasisymmetric functions in Section~\ref{sec:LMF}, and introduce $\LMF{G}$, a restriction of Stanley's symmetric function to maximum length chains in $\tp{\lolli m n}$. In Section~\ref{sec:SMC}, we consider the shortest length maximal chains. Finally, in Section~\ref{sec:further}, we consider a possible generalization of Stanley's symmetric function that unifies some positivity results.

\omitt{
\susanna{2/1/24 Is it ok to remove the itemized list below?}
\begin{itemize}
\item[$\times$] Finish shortest maximal chains
\item[$\times$] Go over preliminaries. 
\begin{itemize}
\item Agree on notation 
\item Agree on what needs to be removed
\end{itemize}
\item[$\times$] Remove stuff
\begin{itemize}
\item Permutahedron notation, turn into comment. 
\item Same for associahedron
\end{itemize} 
\item[$\times$] Make one nice lemma with everything general about $L_{L_{m,n}}$
\item Do further directions that will include diameter stuff and general chain function stuff. 
\end{itemize}
}
\subsection{Acknowledgements} The authors thank Emily Barnard for interesting and useful conversations. 

\section{Preliminaries}
\label{sec:prelims}
Permutations $\fS_n$ of $[n]=\{1,2,\ldots, n\}$ play a large role in this paper. We often write write permutations $w:[n]\rightarrow [n]$ in one-line notation, $w=[w(1), w(2), \ldots, w(n)]$. 
We let $\ii_n=[1,2,\ldots, n]$ be the {\it increasing permutation} or {\it identity} and $\longest^{(n)}=[n,(n-1),\ldots, 1]$ be the {\it decreasing permutation} or {\it longest permutation}. 
\subsection{Posets, the weak order $\fS_n$,  and chains in $\fS_n$}
\label{subsec:weakorder}
A {\it poset}, or partially ordered set, is a set $P$ together with a partial order $\leq$ \omitt{that is reflexive, transitive and anti-symmetric}. \omitt{We use $P$ to denote both the set together with its partial order and the set of vertices.} Two elements $x,y\in P$  are said to be {\it comparable} if either $x\leq y$ or $y\leq x$. In any other case we say the two elements are {\it incomparable}. We say $x$ is covered by $y$ in the poset $P$ if $x<y$ and there does not exist a $z\in P$ such that $x<z<y$. We denote this $x\lessdot y$. We call $x\in P$ a  {\it minimal element}, respectively {\it maximal element}, if there does not exist a $y\in P$ such that $y<x$, respectively $x<y$. 

An \mydef{induced subposet} of $P$ is a subset $Q\subseteq P$ together with the partial order where for $x,y\in Q$ $x<y$ in $Q$ if and only if $x<y$ in $P$. \omitt{Given $Q\subseteq P$, the induced subposet of $P$ with respect to $Q$ is the poset where the elements are the elements of $Q$ and for $x,y\in Q$ $x<y$ in $Q$ if and only if $x<y$ in $P$.}  A {\it chain} is a poset where every element is comparable. We call the subset $C\subseteq P$ a chain if the induced subposet is a chain. We call $C\subseteq P$ a {\it saturated chain} if there does not exist a  $z\in P-C$ such that $x<z<y$ for some $x,y\in C$ with $C\cup \{z\}$ remaining a chain. The {\it length}  of a chain is $\ell(C)=|C|-1$. A chain $C$ is a {\it maximal chain} if there does not exist a $x\in P-C$ such that $C\cup \{x\}$ is a chain. We denote the set of chains of maximum length in a poset $P$ by $\LMC{P}$ A poset is {\it graded} or {\it ranked}  if all maximal chains have the same length. For graded posets, the {\it rank} of an element is defined inductively with $\rho(x)=0$ if $x$ is a minimal element and for non-minimal elements $x$ that cover some $y$, $\rho(x)=\rho(y)+1$. 

Two elements $x,y\in P$ have an {\it upper bound}, respectively {\it lower bound},  if there exists a $z$ such that $x,y<z$, respectively $z<x,y$. We say that $z$ is the {\it least upper bound, join,} of $x,y\in P$, $x\vee y$ if $z$ is an upper bound of $x$ and $y$ and given any other upper bound $u$, $u\geq z$. We say that $z$ is the {\it greatest lower bound, meet,} of $x,y\in P$, $x\wedge y$ if $z$ is a lower bound of $x$  and $y$ and given any other lower bound $u$, $u\leq z$. A {\it lattice} is a poset in which every pair of elements has a unique meet and join. Finite lattices have the nice property that they have a unique minimum, $\hat0$, and maximum, $\hat 1$. We frequently represent posets with a Hasse diagram. For more background on posets and lattices, please see \cite[Chapter 3]{ecI}. \omitt{A picture where elements of $P$ are represented as vertices, connected by edges if there is a cover relation, and larger elements are drawn higher. }


Let $L$ and $L'$ be two lattices. A map $f:L\rightarrow L'$ is a {\it lattice quotient map} if for all $x,y\in L $ we have that 
\begin{enumerate}
\item $f(x\vee y)=f(x)\vee f(y)$, 
\item $f(x\wedge y)=f(x)\wedge f(y)$, 
\item and $f$ is surjective. 
\end{enumerate}
We will then call $L'$ a {\it lattice quotient} of $L$.

For more details on any part of the following discussion of permutations, please see \cite{BB}, for example. The {\it adjacent transposition} $s_i$ is the permutation which interchanges $i$ and $i+1$ and fixes all other values. \omitt{The collection $\{s_1,s_2,\ldots,s_{N-1}\}$ generates $\fS_N$. } Every permutation $w$ can be written as a product of adjacent transpositions $w=s_{i_1}\circ s_{i_2}\circ \cdots \circ s_{i_k}$. If $k$ is minimal among all such expressions for $w$, then $s_{i_1}\circ s_{i_2}\circ \cdots \circ s_{i_k}$ is called a \mydef{reduced decomposition} (or \mydef{expression}) for $w$ and $k$ is the \mydef{length} of $w$, denoted $\ell(w)$.  We often write a reduced expression using its sequence of subscripts and call the sequence $i_1i_2\cdots i_l$ a \mydef{reduced word} for $w$ and let $\redw{w}$ be the collection of all reduced words of $w$.\omitt{An {\it adjacent transposition} is a permutation $s_i\in\fS_n$ for $i<n$ with  $s_i(j)=j$ for all $j\neq i,i+1$, $s_i(j)=j+1$ and $s_i(i+1)=i$. We say that for  $u,v\in\fS_n$ that $v$ covers $u$ in weak order when $u=v\circ s_i$ and $u(i)<u(i+1)$. \susanna{Finish this, transitive closure? 2/13/24}} See Figure~\ref{fig:Bruhat3}. 

The {\it set of inversions} \cite[Section 1.3]{ecI} of a permutation $w\in\fS_n$ is 
$$\Inv(w)=\{(w(i),w(j)):i<j\in[n]\text{ and }w(i)>w(j)\}$$
and the {\it number of inversions} is $\inv(w)=|\Inv(w)|=\ell(w)$. 

The (right) weak order is a partial order on permutations and plays a key role throughout this paper. We use $\fS_n$ to denote both the set of permutations of $[n]$ and the poset, which is in fact a ranked lattice. The rank of a permutation $w$ in the poset $\fS_n$ is $\inv(w)$. The lattice $\fS_n$ is the oriented $1$-skeleton of the permutahedron \cite{armstrong}. The element $u\le w$ in weak order if $w=us_{i_1}s_{i_2}\cdots s_{i_j}$ for adjacent transpositions $s_{i_1},\ldots,s_{i_j}$ and $\ell(us_{i_1}\cdots s_{i_h})=\ell(u)+h$, for $0\le h\le j$. Another description of weak order is $u\le w$ if $\Inv(u)\subseteq\Inv(w)$. We will often use the fact that the reduced words for $w$ are in bijection with maximal chains from  $\ii_n$ to $w$ in $\fS_n$.

We can describe weak order using the standard   geometric realization for type $A$. Consider the hyperplane arrangement in $\mathbb{R}^n$ with hyperplanes $x_a=x_b$  for all $a<b$ in $[n]$. We denote $x_a=x_b$ by $H_{a,b}$ and let $L_1$ be the collection of $H_{a,b}$, the {\it co-dimension one} hyperplanes. 
The hyperplanes in $L_1$ divide $\mathbb{R}^n$ into $n!$ disjoint  chambers, which we label with permutations from $\fS_n$. \omitt{Since this new point of view will be useful we will briefly describe the bijection.} The chamber described by $$\{(v_1,v_2,\ldots,v_n)\in\R^n:v_i<v_j\text{ for }i<j\}$$\omitt{$x_a>x_b$ for all $a>b$}is called the \mydef{base chamber} and has label $\ii_n=[1,2,\ldots, n]$. Suppose a chamber is separated from the base chamber $\ii_n$ by a collection of hyperplanes.  The permutation label  associated to this chamber has $a$ and $b$ in decreasing order  if and only if the hyperplane $H_{a,b}$ separates the chamber from the base chamber. The cover relations of the weak order can now be described geometrically. The permutation $v$ covers the permutation $u$ in the weak order if the chambers labelled by $u$ and $v$ are separated by exactly one hyperplane, say $H_{a,b}$ with $a<b$, and $a$ immediately precedes $b$ in $u$ in one-line notation. In this case, we say that the hyperplane $H_{a,b}$ \mydef{separates} the two permutations. \omitt{We say any permutation is less than or equal to itself and need the transitive closure to obtain the partial order.  
\susanna{Do we have to take the transitive closure to get the order? I think it is right weak, so added (right) 1/2/24} 
\tvi{I think we have to take the transitive closure. I only describe relations between adjacent terms in the poset. I am fine with you adding anything you think we need to add :) (4/4)}}
\omitt{\susanna{I decided to just say we are describing the cover relations only, forget transitive closure. 4/8/2024}}
See Figure~\ref{fig:Bruhat3} and~\ref{fig:PathEquiv}. \omitt{Let $L_1(w)$ be the collection of hyperplanes that separate $\ii_n$ and $w$. This construction also forms the (right) weak order $\fS_n$. }

\omitt{Any permutation $w$ can be written as a composition of adjacent transpositions $w=s_{i_1}\circ s_{i_2}\circ \cdots \circ s_{i_l}$, and there will be a certain collection of compositions that use a minimum number of adjacent transpositions. These shortest compositions are in bijection with maximal chain from  $\ii_n$ to $w$ in $\fS_n$. To simplify notation we focus on the sequence of subscripts and call the sequence $i_1i_2\cdots i_l$ a {\it reduced word} of $w$ and let $\cR(w)$ be the collection of all reduced words of $w$. The {\it length} of the the reduced word, and also the {\it length} of the permutation, is $\ell(w)=\inv(w)$. }
\begin{example}
The permutation $w=4231$ has $\cR(w)=\{32123, 31213, 13213, 31231, 13231, 12321\}$, $L_1(w)=\{H_{1,2},H_{1,3},H_{1,4},H_{2,4},H_{3,4}\}$ and length $\ell(w)=5$. 
\end{example}

\begin{figure}
\begin{center}
\begin{tikzpicture}[scale=1.25]
\draw (.25,.25)--(.75,.75);
\draw (-.25,.25)--(-.75,.75);
\draw (-1,1.25)--(-1,1.75);
\draw (1,1.25)--(1,1.75);
\draw (-.75,2.25)--(-.25,2.75);
\draw (.75,2.25)--(.25,2.75);
\draw (0,0) node {${[1,2,3]}$};
\draw (1,1) node {$[2,1,3]$};
\draw (-1,1) node {$[1,3,2]$};
\draw (1,2) node {$[2,3,1]$};
\draw (-1,2) node {$[3,1,2]$};
\draw (0,3) node {${[3,2,1]}$};
\draw (.75,.5) node {\small{$s_1$}};
\draw (1.25,1.5) node {\small{$s_2$}};
\draw (.75,2.5) node {\small{$s_1$}};
\draw (.25,.5) node {\tiny{$H_{1,2}$}};
\draw (.75,1.5) node {\tiny{$H_{1,3}$}};
\draw (.25,2.5) node {\tiny{$H_{2,3}$}};
\draw (-.25,.5) node {\small{$s_2$}};
\draw (-.75,1.5) node {\small{$s_1$}};
\draw (-.25,2.5) node {\small{$s_2$}};
\draw (-.75,.5) node {\tiny{$H_{2,3}$}};
\draw (-1.25,1.5) node {\tiny{$H_{1,3}$}};
\draw (-.75,2.5) node {\tiny{$H_{1,2}$}};
\begin{scope}[shift={(5,1.5)}]
\draw (1.65,1) node {$H_{1,2}$};
\draw (1.65,-1) node {$H_{1,3}$};
\draw (2.60,0) node {$H_{2,3}$};
\draw (1,.25) node {\small{${[1,2,3]}$}};
{\draw (1,-.25) node {\small{$[1,3,2]$}};}
{\draw (0,.75) node {\small{$[2,1,3]$}};}
{\draw (-1,.25) node {\small{$[2,3,1]$}};}
{\draw (0,-.75) node {\small{$[3,1,2]$}};}
{\draw (-1,-.25) node {\small{$[3,2,1]$}};}
\draw[<->] (-1,1)--(1,-1);
\draw[<->] (-2,0)--(2,0);
\draw[<->] (1,1)--(-1,-1);
\end{scope}
\end{tikzpicture}
\end{center}
\caption{On the right we have the hyperplane arrangement in $\mathbb{R}^3$ projected onto $\mathbb{R}^2$ along the vector $\langle 1,1,1\rangle$ with the chambers labeled by their permutation. On the left we have the weak order $\fS_3$ with both the adjacent transposition and hyperplane labels. }
\label{fig:Bruhat3}
\end{figure}
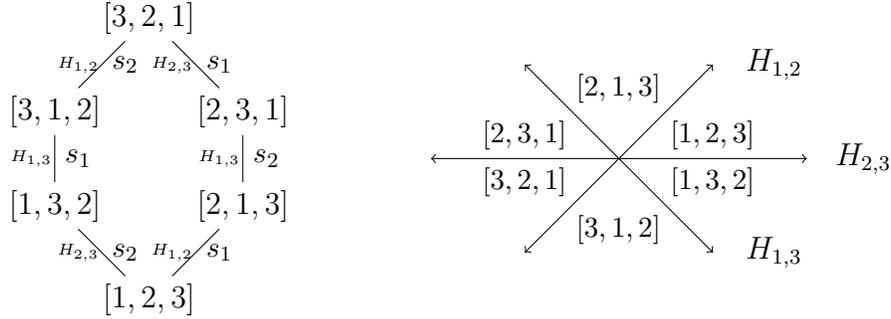

$$L_1(w)=\{H_{w(i),w(j)}:(w(i),w(j))\in\Inv(w)\}.$$

Cover relations of the weak order may be naturally labeled with adjacent transpositions and/or labeled with hyperplanes. See Figure~\ref{fig:Bruhat3}.  By using adjacent transpositions as labels, we view the maximum length chains as reduced words for the longest permutation $\longestk N$ of $\fS_N$.


\subsection{Partitions and tableaux}
\label{sec:prelim tableaux}


An {\it integer partition}, $\lambda=(\lambda_1,\lambda_2,\ldots,\lambda_L)$
of $N$ is a weakly decreasing sequence, $\lambda_1\geq \lambda_2\geq\cdots \geq \lambda_L$, of
positive integers $\lambda_i$ called {\it parts} that sum to $N$. We write $\lambda\vdash N$ and
the {\it length} of an integer partition $\ell(\lambda)=L$ is the number of parts. We can visualize an
integer partition with its {\it Young diagram}, which is an array of $L$ left-justified rows with $\lambda_i$ cells
(or boxes) in row $i$. Row $1$ is at the top, row $L$ at the bottom, and we number the columns left to right.
We will be particularly interested in the {\it staircase shape} $\stair_N=(N-1,N-2,\ldots, 2,1)$, an integer partition of $\binom{N}{2}$.


A tableau $T$ of shape $\lambda$ is an assignment of positive integers
to the cells in the diagram of $\lambda$. We refer to the entry in
row $i$ and column$j$ as either $T_{ij}$ or $T(i,j)$ and to the cell
itself as $(i,j)$.

\subsubsection{Balanced tableaux}

The hook of a cell in a Young diagram is the set of cells in the
diagram to the cell's right, of cells below it, and of the cell
itself. The cell's height is the number of cells below it, including
itself. A cell $(i,j)$ in a tableau $T$ is \mydef{balanced} if the number of
entries in its hook which are less than or equal to $T_{i,j}$ is equal
to the cell's height. In the case where $\lambda$ is of staircase
shape, the balanced condition simplifies to the number of smaller
entries to the left of the cell equals the number of larger entries
below it. Finally, a tableau is {\it balanced} if all its cells are
balanced. For $\lambda\vdash N$ let $\balanced{\lambda}$ be the
collection of Young diagrams with a balanced filling that uses each
number in $[N]$ once. We will call these {\it balanced tableaux}. We
call $\lambda$ the underlying {\it shape} or just the shape of the
tableau.

A \mydef{hyperplane walk} from $u\in\fS_N$ to $v\in\fS_N$ is a minimal length sequence of hyperplanes crossed, in order, when traveling from the chamber associated with $u$ to the chamber associated to $v$. If $u$ and $v$ are not mentioned, then $u$ is assumed to be the identity $\ii_N$ and $v$ to be the longest permutation $\longestk N$. The set of hyperplane walks from $\ii_N$ to $\longestk N$ is denoted $\cH(\longestk N).$ Each such hyperplane walk is matched with a balanced tableau of shape
$\stair_N$ as follows \cite{EG}.  Let $H\in \cH(\longestk N)$ and
suppose $H_{a,b}$ is the $i$th hyperplane in the hyperplane walk. We
then place $i$ in row $a$ and column $N-b+1$. The starting point of much of this paper is the well-known correspondences between reduced words, hyperplane walks, and tableaux \cite{EG,stan84}. 

Figure~\ref{fig:fourassoc}
has an example of a reduced word and its associated tableaux and hyperplane
walk. 

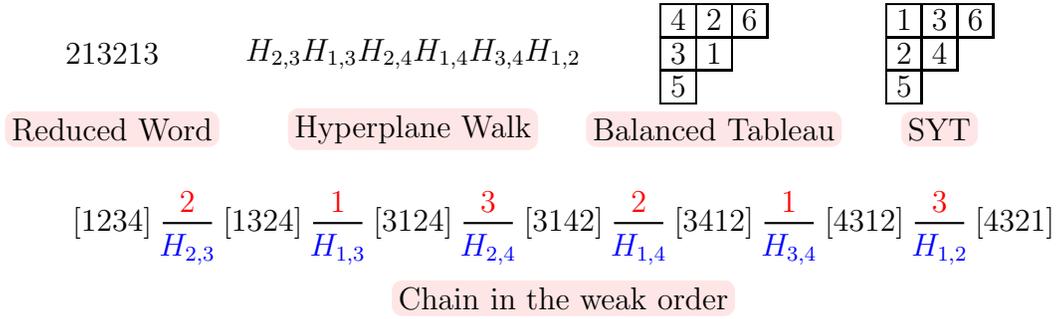
\begin{figure}
\begin{center}
\begin{tikzpicture}[information text/.style={rounded corners,fill=red!10,inner sep=.5ex}]

\def\pd{2.}
\draw (0,0) node {$213213$};
\draw (0,-1) node[information text] {Reduced Word};
\begin{scope}[shift={(4,0)}]
\draw (0,0) node {$H_{2,3}H_{1,3}H_{2,4}H_{1,4}H_{3,4}H_{1,2}$};
\draw (0,-1) node[information text] {Hyperplane Walk};
\end{scope}
\begin{scope}[shift={(8,0)}]
\draw (0,0) node {$\begin{Young} 4&2&6\cr 3&1 \cr 5\cr \end{Young}$};
\draw (0,-1) node[information text] {Balanced Tableau};
\end{scope}
\begin{scope}[shift={(11,0)}]
\draw (0,0) node {$\begin{Young} 1&3&6\cr 2&4 \cr 5\cr \end{Young}$};
\draw (0,-1) node[information text] {SYT};
\end{scope}
\begin{scope}[shift={(0,-2.25)}]
\node (p0) at (0,0) {$[1234]$};
\node (p1) at (\pd,0) {$[1324]$};
\node (p2) at (2*\pd,0) {$[3124]$};
\node (p3) at (3*\pd,0) {$[3142]$};
\node (p4) at (4*\pd,0) {$[3412]$};
\node (p5) at (5*\pd,0) {$[4312]$};
\node (p6) at (6*\pd,0) {$[4321]$};
\draw[thick,black] (p0)--(p1) node[above,midway,red]{$2$} node[below,midway,blue] {$H_{2,3}$};
\draw[thick,black] (p1)--(p2) node[above,midway,red]{$1$} node[below,midway,blue] {$H_{1,3}$};
\draw[thick,black] (p2)--(p3) node[above,midway,red]{$3$} node[below,midway,blue] {$H_{2,4}$};
\draw[thick,black] (p3)--(p4) node[above,midway,red]{$2$} node[below,midway,blue] {$H_{1,4}$};
\draw[thick,black] (p4)--(p5) node[above,midway,red]{$1$} node[below,midway,blue] {$H_{3,4}$};
\draw[thick,black] (p5)--(p6) node[above,midway,red]{$3$} node[below,midway,blue] {$H_{1,2}$};
\node[information text] at (3*\pd,-1) {Chain in the weak order};
\end{scope}
\end{tikzpicture}
\end{center}
\caption{A reduced word and its associated tableaux, hyperplane walk, and chain of permutations in the weak order.}
\label{fig:fourassoc}
\end{figure}
\omitt{I think $H_{2,4}$ and $H_{1,3}$ should be swapped in Figure~\ref{fig:fourassoc} this reduced word. The BT would need a change then too, but I think the SYT is fine. changed 12/11/23}
\subsubsection{Standard Young tableaux}

If we instead fill the Young diagram of $\lambda\vdash N$ with
positive integers such that the entries weakly increase left to right across rows and
columns strictly increase top to bottom, then we call the filled Young diagram a
{\it semi-standard Young tableau}.  If each number in $[N]$ appears
exactly once we call it a {\it standard Young tableau}. Let
$\SYT(\lambda)$ be the collection of standard Young tableaux of shape
$\lambda$ where $\lambda$ is the underlying {\it shape} of the
tableau.

There is a bijection from reduced words $\cR(\omega^{(N)})$ to
$\text{SYT}(\stair_N)$. The bijection is known as Edelman-Greene
insertion \cite{EG}, although the authors named it Coxeter-Knuth.

\begin{definition}[\cite{EG},\cite{LP20}]
  \label{def:EG}
We define Edelman-Greene insertion recursively. Let $\alpha=\alpha_1\cdots\alpha_{\binom{N}{2}}\in \redw{\longestk{N}}$ be a reduced word. We start with a pair of empty tableaux with the shape of the empty partition $(P_0,Q_0)$. Now suppose we have inserted $\alpha_1\alpha_2\ldots \alpha_{k-1}$, resulting in $(P^{(k-1)},Q^{(k-1)})$, where $P^{(k-1)}$ is a semi-standard Young tableau and $Q^{(k-1)}$ is a standard Young tableau of the same shape. Denote the rows of $P^{(k-1)}$ by $R_1,R_2,\ldots,R_h$. 
We describe how we insert $\alpha_{k}$ to obtain $(P^{(k)},Q^{(k)})$. We set $x_0$ to be $\alpha_{k}$ to begin. At step $i\ge1$, we attempt to insert $x_{i-1}$ into row $R_i$.  
\begin{enumerate}
\item If $x_{i-1}$ is larger than all entries in $R_i$ or $R_i$ is empty, then place $x_{i-1}$ at the end of $R_i$. This completes the insertion of $x_{i-1}$. 
\item In any other case there exists a smallest number $z$ in $R_i$ greater than or equal to $x_{i-1}$. 
\begin{enumerate}
\item If $z=x_{i-1}$, set $x_i=z+1$ and insert $x_i$ into $R_{i+1}$. We say $x_{i-1}$ \mydef{bumps} $z$.
\item If $z>x_{i-1}$, replace $z$ in $R_i$ by $x_{i-1}$ and insert $x_i=z$ in $R_i$. Again, we say $x_{i-1}$ \mydef{bumps} $z$.
\end{enumerate}
Continue the bumping process and insertion, inserting $x_i$ into $R_{i+1}$, until $x_i$ is placed at the end of a row or into an empty row. 
\end{enumerate}
The insertion of $x_0=\alpha_k$ into $P^{(k-1)}$ produces a sequence of length less than or equal to $h$ of bumped elements $x_0,x_1,\ldots$. The \mydef{bump path} is the sequence of cells holding $x_1,x_2,\ldots$ in $P^{(k-1)}$. 
After the insertion of $\alpha_k$, $P^{(k)}$ has one more cell than $P^{(k-1)}$ in row $i$ and column $j$. In $Q^{(k-1)}$ place a $k$ in a new cell in row $i$ and column $j$ to form $Q^{(k)}$. 
After inserting $\alpha_k$ for each $k$, $1\le k\le \binom{N}{2}$, we have $\EG(\alpha)=(P^{(L)},Q^{(L)})=(P_{\alpha},Q_{\alpha})$ for $L=\binom{N}{2}$. Both $P$ and $Q$ have shape $\stair_N$. The standard Young tableau associated to $\alpha$ is $\EGQ(\alpha)=Q_{\alpha}=Q$. The tableau $P_{\alpha}$ is the same for all $\alpha\in\redw{\longestk{N}}$. 
\end{definition}

 See Figure~\ref{fig:fourassoc} for an example.  We have a standard Young tableau $\EGQ(\rho)$ associated to a reduced word $\rho$ via the Edelman-Greene insertion. We mention the insertion is defined for the reduced word of any permutation.

With the map in Definition~\ref{def:EG},  Edelman and Greene provided a bijective proof of Stanley's celebrated
result \cite{stan84} that the number of reduced words for
$\longestk{N}$ is equal to the number of standard Young tableaux of
shape $(N-1,N-2,\ldots,1)$. They also define the inverse $\Gamma$ to $\EG$, which we describe now. 

\begin{definition}[\cite{EG}]
  \label{def:promo}
  Let $T$ be a tableau of shape $(\lambda_1,\lambda_2,\ldots,\lambda_k)$ with entries increasing from left to right along rows and increasing down columns, and let $M=\sum_i\lambda_i$. Let $T_{i,j}$ denote the entry in the cell in row $i$ and column $j$. We allow entries in $T$ to be negative. Suppose the largest entry of $T$ is $T_{p,q}$.

  The \mydef{evacuation path} $\pi=(\pi_1,\pi_2,\ldots,\pi_{p+q-1})$ is a sequence of cells $\pi_i=(p_i,q_i)$ constructed as follows:
  \begin{enumerate}
  \item $\pi_1=(p_1,q_1)=(p,q)$
  \item \label{part:promo_comp}If $\pi_i=(p_i,q_i)$ \omitt{and $i<p_i+q_i$}, then $\pi_{i+1}$ is
    \begin{enumerate}
    \item $(p_i-1,q_i)$ if $q_i=1$ or $T_{p_i-1,q_i}>T_{p_i,q_i-1}$;
    \item $(p_i,q_i-1)$ if $p_i=1$ or $T_{p_i-1,q_i}<T_{p_i,q_i-1}$;
      \end{enumerate}
  \end{enumerate}

    The tableau $\delta(T)$ is obtained from $T$ by replacing $T_{p_i,q_i}$ by $T_{p_{i+1},q_{i+1}}$ for $1\leq i<p+q-1$ and replacing $T_{11}$ by the minimum entry of $T$, minus 1. Following \cite{EG}, a single application of $\delta$ is called \mydef{elementary promotion}.

    Now fix $T$ to be a standard Young tableau with entries $\{1,2,\ldots,M\}$. Define $T^{(k)}$ to be the result of applying $\delta$ $k$ times to $T$. Finally, define $\Gamma(T)$ to be the reduced word $\rho=\rho_1\cdots\rho_M$, where $\rho_{M-k+1}$ is the column in $T^{(k-1)}$ of its largest entry.

\end{definition}

We give an example of $\Gamma$ in Figure~\ref{fig:Gamma}. 

\ytableausetup{boxsize=1.25em}

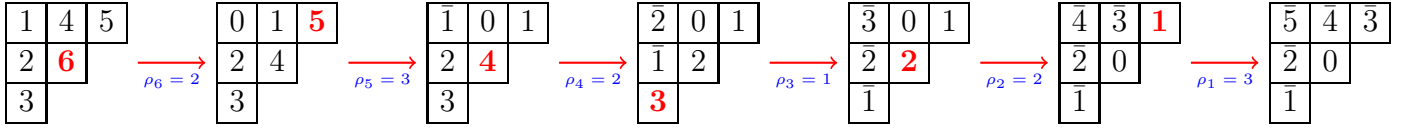
\begin{figure}[h]
  \centering
  \begin{tikzpicture}
    \def\h{2.8}
    \node (t1) at (0,0) {\begin{ytableau}1&4&5\\2&\smallcircled{6}\\3\end{ytableau}};
    \node (t2) at (\h,0) {\begin{ytableau}0&1&\smallcircled{5}\\2&4\\3\end{ytableau}};
    \draw[thick,red,->] (t1.east)--(t2.west)node[midway,below,blue,font=\tiny]{$\rho_6=2$};
    \node (t3) at (2*\h,0) {\begin{ytableau}\bar{1}&0&1\\2&\smallcircled{4}\\3\end{ytableau}};
    \draw[thick,red,->] (t2.east)--(t3.west)node[midway,below,blue,font=\tiny]{$\rho_5=3$};
    \node (t4) at (3*\h,0) {\begin{ytableau}\bar{2}&0&1\\\bar{1}&2\\\smallcircled{3}\end{ytableau}};
    \draw[thick,red,->] (t3.east)--(t4.west)node[midway,below,blue,font=\tiny]{$\rho_4=2$};
    \node (t5) at (4*\h,0) {\begin{ytableau}\bar{3}&0&1\\\bar{2}&\smallcircled{2}\\\bar{1}\end{ytableau}};
    \draw[thick,red,->] (t4.east)--(t5.west)node[midway,below,blue,font=\tiny]{$\rho_3=1$};
    \node (t6) at (5*\h,0) {\begin{ytableau}\bar{4}&\bar{3}&\smallcircled{1}\\\bar{2}&0\\\bar{1}\end{ytableau}};
    \draw[thick,red,->] (t5.east)--(t6.west)node[midway,below,blue,font=\tiny]{$\rho_2=2$};    
    \node (t7) at (6*\h,0) {\begin{ytableau}\bar{5}&\bar{4}&\bar{3}\\\bar{2}&0\\\bar{1}\end{ytableau}};
    \draw[thick,red,->] (t6.east)--(t7.west)node[midway,below,blue,font=\tiny]{$\rho_1=3$};    
    
    \end{tikzpicture}
  \caption{Let $T$ be the first tableau on the left. The largest number in each tableaux is in red and negative numbers are overlined. The tableau $\delta(T)$
    is the second, $\delta^2(T)$ the third, etc. The evacuation path
    for the first application of $\delta$ is
    $\{(2,2),(1,2),(1,1)\}$. The reduced word $\Gamma(T)$ is
    $321232$. Please see Definition~\ref{def:promo}.}
  \label{fig:Gamma}
  \end{figure}

\ytableausetup{boxsize=normal}

We record the correspondences among chains, walks, and tableaux in Theorem~\ref{thm:longestwordSets}. 
\begin{theorem}
\label{thm:longestwordSets}
The following sets are in bijection with each other. 
\begin{enumerate}
\item \label{part:redw}The set of reduced words the longest permutation $\longestk{N}\in\fS_N$, $\redw{\longestk{N}}$.
\item \label{part:walks}The set of hyperplane walks, $\cH(\longestk{N})$. 
\item \label{part:satchains}The set of maximal chains in the weak order $\lc{\fS_N}$. 
\item \label{part:BT}The set of balanced tableaux of the staircase shape, $\balanced{\stair_N}$. 
\item \label{part:SYT}The set of of standard Young tableaux of the staircase shape $\SYT(\stair_n)$. 
\end{enumerate}
\label{thm:chain_bijection}
\end{theorem}
\begin{proof}
The first three \eqref{part:redw}-\eqref{part:satchains}  are in bijection with each other by construction. Item \eqref{part:BT} and \eqref{part:SYT} are in bijection with \eqref{part:redw} by \cite{EG, stan84}. 
\end{proof}


\section{The tubing poset $\tp{G}$}
\label{sec:L_Gintro}

In this section we introduce the lattices we study in this paper. These
are lattice quotients of the weak Bruhat order determined by a
graph $G$. We begin with graph terminology.

A {\it graph} $G$ is a collection of vertices $V$ and a collection of
edges $E$ that connect pairs of vertices. In this paper we will only
consider simple graphs without loops and without multi-edges. Edges
will be represented as $uv$ for $u,v\in V$ when the expressions for $u$ and $v$ are simple, and as $\{u,v\}$ when they are more complicated. Since our edges do not
have direction, $uv=vu$. Two edges are {\it adjacent} if they share a
vertex and two vertices are {\it adjacent} if they are on opposite
ends of an edge.  A {\it path} between two vertices $u$ and $v$ is a
sequence of adjacent edges, $uu_1,u_1u_2,\ldots, u_kv$ that starts at
$u$ and ends at $v$. The {\it length} of the path is the number of
edges. The {\it distance} between two vertices $u$ and $v$, $d(u,v)$,
is the length of the shortest path, if a path exists. We call a graph
{\it connected} if there is a path between all pairs of vertices.

\omitt{Recall that a graph is connected if there exists a path between every pair of vertices.} Given a subset of vertices $W\subseteq V$, the {\it induced graph}, $G|_W$, is the graph on vertices $W$ with edges $uv$ if $u,v\in W$ and $uv$ was an edge of $G$.  There are a few common graphs that we will be consistently using. The {\it complete graph} $K_n$ is a graph on vertices $V=[n]=\{1,2,3,\ldots, n\}$ with edges $ij$ for all $i\neq j$ in $[n]$. The {\it path} $P_n$ is a graph on vertices $V=[n]$ with edges $i(i+1)$ for all $i\in[n-1]$. A {\it lollipop } $\lolli m n$ is the graph on vertices $V=[m+n]$ such that there is a complete graph on vertices $[m]$ and a path on vertices in $\{m,m+1,\ldots m+n\}$. There is some overlap among graph notation, since $L_{1,n-1}=L_{0,n}=P_{n}$ and $L_{m,0}=K_m$. We set $P_n=\lolli{1}{n-1}$ for technical reasons and we assume that $m\ge 1$ and $n\ge 0$. The path graph $P_n$ is also equal to $\lolli{2}{n-2}$. \omitt{so typically for lollipop graphs $L_{m,n}$ we will assume that $m,n\geq 1$.} We call a graph on vertices $[n]$ {\it filled} if for all edges $ij$ for $i<j$ we have the edges $ik$ and $kj$ for all $i<k<j$. Complete graphs, paths and lollipop graphs are all filled graphs. See Figure~\ref{fig:graphs} for examples. 

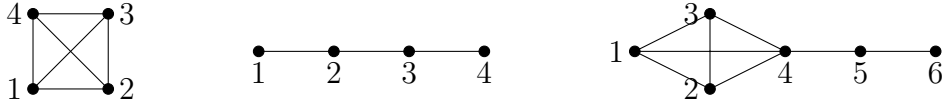
\begin{figure}
\begin{center}
\begin{tikzpicture}
\coordinate (A) at (0,0);
\coordinate (B) at (1,0);
\coordinate (C) at (1,1);
\coordinate (D) at (0,1);
\draw[black] (A)--(B)--(C)--(D)--(A)--(C);
\draw[black] (B)--(D);
\filldraw[black] (A) circle [radius=2pt] node[left] {$1$};
\filldraw[black] (B) circle [radius=2pt] node[right] {$2$};
\filldraw[black] (C) circle [radius=2pt] node[right] {$3$};
\filldraw[black] (D) circle [radius=2pt] node[left] {$4$};
\begin{scope}[shift={(3,.5)}]
\coordinate (A) at (0,0);
\coordinate (B) at (1,0);
\coordinate (C) at (2,0);
\coordinate (D) at (3,0);
\draw[black] (A)--(B)--(C)--(D);
\filldraw[black] (A) circle [radius=2pt] node[below] {$1$};
\filldraw[black] (B) circle [radius=2pt] node[below] {$2$};
\filldraw[black] (C) circle [radius=2pt] node[below] {$3$};
\filldraw[black] (D) circle [radius=2pt] node[below] {$4$};
\end{scope}
\begin{scope}[shift={(8,0)}]
\coordinate (A) at (0,.5);
\coordinate (B) at (1,0);
\coordinate (C) at (1,1);
\coordinate (D) at (2,.5);
\coordinate (E) at (3,.5);
\coordinate (F) at (4,.5);
\draw[black] (A)--(C);
\draw[black] (B)--(D)--(A)--(B)--(C)--(D)--(F);
\filldraw[black] (A) circle [radius=2pt] node[left] {$1$};
\filldraw[black] (B) circle [radius=2pt] node[left] {$2$};
\filldraw[black] (C) circle [radius=2pt] node[left] {$3$};
\filldraw[black] (D) circle [radius=2pt] node[below] {$4$};
\filldraw[black] (E) circle [radius=2pt] node[below] {$5$};
\filldraw[black] (F) circle [radius=2pt] node[below] {$6$};
\end{scope}
\end{tikzpicture}
\end{center}
\caption{From left to right we have $K_4$, $P_4$, and $\tp{\lolli 4 2 }$}.
\label{fig:graphs}
\end{figure}

We study the
\mydef{tubing poset} $\tp{G}$ of $G$, for $G$ a filled and connected graph on $N$ vertices. We don't use tubings, or even
define them. We use an alternate description, found in \cite{BM}, for
the case of filled connected graphs. It turns out, that since $G$ is
filled, $\tp{G}$ is a lattice and in fact the lattice quotient of the weak
order $\fS_N$ by the \mydef{$G$-equivalence relation}.  The
$G$-equivalence relation is based on the $G$-tree of a
permutation. A $G$-tree is a labeled rooted tree, meaning a graph
without a cycle, and a special vertex to be called a {\it root}. If
$uv$ is an edge of the tree, and $u$ is closer to the root than $v$,
we call $v$ a {\it child} of $u$. Given a graph $G$ and a permutation
$w\in\fS_N$ we define the {\it $G$-tree of $w$} on vertex set $[N]$
inductively.
\begin{enumerate}
\item Let $w(N)$ be the root. 
\item $G|_{[N]-\{w(N)\}}$ has connected components $G_1, G_2, \ldots, G_k$ on vertex sets $V_1, V_2, \ldots, V_k$. Let $w_1, w_2, \ldots w_k$ be the associated subwords of $w$ where $w_i$ is the subword of $w$ containing all numbers in $V_i$. By induction each $w_i$ has a $G_i$-tree $T_i$. 
\item Let the roots of $T_1, T_2, \ldots, T_k$ be the children of $w(N)$ in the $G$-tree of $w$. 
\end{enumerate}
Two permutations $u$ and $v$ are in the same equivalence class, or {\it $G$-equivalent},  if they have the same $G$-tree. In that case, we write $[\![u]\!]=[\![v]\!]$. See Figure~\ref{fig:Gtree} for an example. 
 
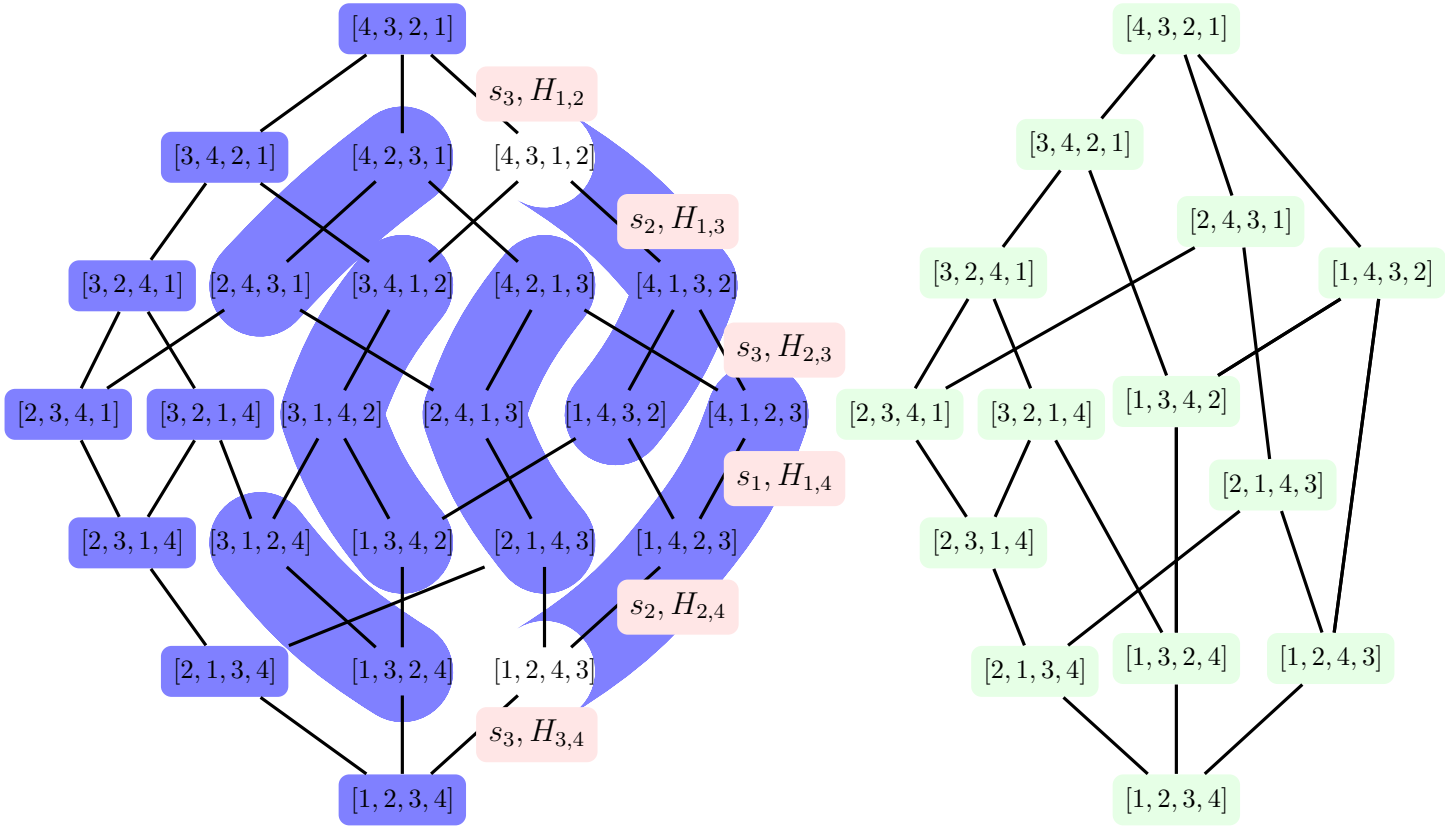
\begin{figure}
\begin{center}
\begin{tikzpicture}
\begin{scope}

\def\h{2.1cm}
\def\v{1.9cm}
\def\th{.58cm}
\def\w{1.35cm}
\begin{scope}[scale=.9,perm/.style={font=\footnotesize},
information text/.style={rounded corners,fill=red!10,inner sep=.9ex},
singleton perm/.style={rounded corners,fill=blue!50,inner sep=.9ex,font=\footnotesize}]
\coordinate (1234) at (0*\h,0*\v);

\coordinate (2134) at (-1.25*\h,1*\v);
\coordinate (1324) at (0*\h,1*\v);
\coordinate (1243) at (1*\h,1*\v);

\coordinate (2314) at (-1.9*\h,2*\v);
\coordinate (3124) at (-1*\h,2*\v);
\coordinate (1342) at (0*\h,2*\v);
\coordinate (2143) at (1*\h,2*\v);
\coordinate (1423) at (2*\h,2*\v);

\coordinate (2341) at (-2.35*\h,3*\v);
\coordinate (3214) at (-1.35*\h,3*\v);
\coordinate (3142) at (-.5*\h,3*\v);
\coordinate (2413) at (.5*\h,3*\v);
\coordinate (1432) at (1.5*\h,3*\v);
\coordinate (4123) at (2.5*\h,3*\v);

\coordinate (3241) at (-1.9*\h,4*\v);
\coordinate (2431) at (-1*\h,4*\v);
\coordinate (3412) at (0*\h,4*\v);
\coordinate (4213) at (1*\h,4*\v);
\coordinate (4132) at (2*\h,4*\v);

\coordinate (3421) at (-1.25*\h,5*\v);
\coordinate (4231) at (0*\h,5*\v);
\coordinate (4312) at (1*\h,5*\v);

\coordinate (4321) at (0*\h,6*\v);


\path[fill=blue!10,draw=blue!50, line width=\w,line cap=round, line join=round] 
    (1342) to [bend left=10] (3142) to [bend left=10]
        (3412)to [bend right=10](3142) to[bend right=10] (1342)--cycle;

\path[fill=blue!10,draw=blue!50, line width=\w,line cap=round, line join=round] 
    (2143) to [bend left=10] (2413) to [bend left=10]
        (4213)to [bend right=10](2413) to[bend right=10] (2143)--cycle;
        
\path[fill=blue!10,draw=blue!50, line width=\w,line cap=round, line join=round] 
    (1432) to [bend right=10] (4132) to [bend right=10]
        (4312)to [bend left=10](4132) to[bend left=10] (1432)--cycle;
        
\path[fill=blue!10,draw=blue!50, line width=\w,line cap=round, line join=round] 
    (1243) to [bend right=10] (1423) to [bend right=10]
        (4123)to [bend left=10](1423) to[bend left=10] (1243)--cycle;

\path[fill=blue!10,draw=blue!50, line width=\w,line cap=round, line join=round] 
    (1432) to [bend right=10] (4132) to [bend right=10]
        (4312)to [bend left=10](4132) to[bend left=10] (1432)--cycle;

\path[fill=blue!10,draw=blue!50, line width=\w,line cap=round, line join=round] 
    (1324) to [bend left=10] (3124) to [bend right=10]
        (1324)--cycle;

\path[fill=blue!10,draw=blue!50, line width=\w,line cap=round, line join=round] 
    (2431) to [bend left=5] (4231) to [bend right=5](2431)--cycle;

\node[perm,singleton perm] (e) at (1234) {$[1,2,3,4]$};

\node[perm,singleton perm] (r11) at (2134) {$[2,1,3,4]$};
\node[perm] (r12) at (1324) {$[1,3,2,4]$};
\node[perm] (r13) at (1243) {$[1,2,4,3]$};

\node[perm,singleton perm] (r21)at (2314) {$[2,3,1,4]$};
\node[perm] (r22)at (3124) {$[3,1,2,4]$};
\node[perm] (r23)at (1342) {$[1,3,4,2]$};
\node[perm] (r24)at (2143) {$[2,1,4,3]$};
\node[perm] (r25)at (1423) {$[1,4,2,3]$};

\node[perm,singleton perm] (r31) at (2341) {$[2,3,4,1]$};
\node[perm,singleton perm,] (r32) at (3214) {$[3,2,1,4]$};
\node[perm] (r33) at (3142) {$[3,1,4,2]$};
\node[perm] (r34) at (2413) {$[2,4,1,3]$};
\node[perm] (r35) at (1432) {$[1,4,3,2]$};
\node[perm] (r36) at (4123) {$[4,1,2,3]$};

\node[perm,singleton perm] (r41) at (3241){$[3,2,4,1]$};
\node[perm] (r42) at (2431){$[2,4,3,1]$};
\node[perm] (r43) at (3412){$[3,4,1,2]$};
\node[perm] (r44) at (4213){$[4,2,1,3]$};
\node[perm] (r45) at (4132){$[4,1,3,2]$};

\node[perm,singleton perm] (r51) at (3421){$[3,4,2,1]$};
\node[perm] (r52) at (4231){$[4,2,3,1]$};
\node[perm] (r53) at (4312){$[4,3,1,2]$};

\node[perm,singleton perm] (r61) at (4321){$[4,3,2,1]$};

\draw[very thick] (e)--(r11);
\draw[very thick] (e)--(r12);
\draw[very thick] (e)--(r13) node[midway,right,information text]{$s_3,H_{3,4}$};

\draw[very thick] (r11)--(r21);
\draw[very thick] (r11)--(r24);
\draw[very thick] (r12)--(r22);
\draw[very thick] (r12)--(r23);
\draw[very thick] (r13)--(r24);
\draw[very thick] (r13)--(r25)node[midway,right,information text]{$s_2,H_{2,4}$};

\draw[very thick] (r21)--(r31);
\draw[very thick] (r21)--(r32);
\draw[very thick] (r22)--(r32);
\draw[very thick] (r22)--(r33);
\draw[very thick] (r23)--(r33);
\draw[very thick] (r23)--(r35);
\draw[very thick] (r24)--(r34);
\draw[very thick] (r25)--(r35);
\draw[very thick] (r25)--(r36)node[midway,right,information text]{$s_1,H_{1,4}$};

\draw[very thick] (r31)--(r41);
\draw[very thick] (r31)--(r42);
\draw[very thick] (r32)--(r41);
\draw[very thick] (r33)--(r43);
\draw[very thick] (r34)--(r42);
\draw[very thick] (r34)--(r44);
\draw[very thick] (r35)--(r45);
\draw[very thick] (r36)--(r45)node[midway,right,information text]{$s_3,H_{2,3}$};
\draw[very thick] (r36)--(r44);

\draw[very thick] (r41)--(r51);
\draw[very thick] (r42)--(r52);
\draw[very thick] (r43)--(r51);
\draw[very thick] (r43)--(r53);
\draw[very thick] (r44)--(r52);
\draw[very thick] (r45)--(r53)node[midway,right,information text]{$s_2,H_{1,3}$};

\draw[very thick] (r51)--(r61);
\draw[very thick] (r52)--(r61);
\draw[very thick] (r53)--(r61)node[midway,right,information text]{$s_3,H_{1,2}$};
\end{scope}

\begin{scope}[shift={(4.9*\h,0*\v)},scale=.9,every node/.style={font=\footnotesize,rounded corners, fill=green!10, inner sep=.9ex},
information text/.style={rounded corners,fill=red!10,inner sep=.9ex}]
\def\h{1.9cm}
\coordinate (1234) at (0*\h,0*\v);
\coordinate (2134) at (-1.1*\h,1*\v);
\coordinate (1324) at (0*\h,1.1*\v);
\coordinate (1243) at (1.2*\h,1.1*\v);
\coordinate (2314) at (-1.5*\h,2*\v);
\coordinate (2341) at (-2.15*\h,3*\v);
\coordinate (3214) at (-1.05*\h,3*\v);
\coordinate (3241) at (-1.5*\h,4.1*\v);
\coordinate (3421) at (-.75*\h,5.1*\v);
\coordinate (4321) at (0*\h,6*\v);
\coordinate (2143) at (.75*\h,2.45*\v);
\coordinate (1342) at (0*\h,3.1*\v);
\coordinate (2431) at (.5*\h,4.5*\v);
\coordinate (1432) at (1.6*\h,4.1*\v);

\node (e) at (1234) {$[1,2,3,4]$};

\node (r11) at (2134) {$[2,1,3,4]$};
\node (r12) at (1324) {$[1,3,2,4]$};
\node (r13) at (1243) {$[1,2,4,3]$};

\node (r21)at (2314) {$[2,3,1,4]$};
\node (r23)at (1342) {$[1,3,4,2]$};
\node (r24)at (2143) {$[2,1,4,3]$};

\node (r31) at (2341) {$[2,3,4,1]$};
\node (r32) at (3214) {$[3,2,1,4]$};
\node (r35) at (1432) {$[1,4,3,2]$};

\node (r41) at (3241){$[3,2,4,1]$};
\node (r42) at (2431){$[2,4,3,1]$};

\node (r51) at (3421){$[3,4,2,1]$};

\node (r61) at (4321){$[4,3,2,1]$};

\draw[very thick] (e)--(r11);
\draw[very thick] (e)--(r12);
\draw[very thick] (e)--(r13);

\draw[very thick] (r11)--(r21);
\draw[very thick] (r11)--(r24);
\draw[very thick] (r12)--(r23);
\draw[very thick] (r13)--(r24);

\draw[very thick] (r21)--(r31);
\draw[very thick] (r21)--(r32);
\draw[very thick] (r12)--(r32); 
\draw[very thick] (r12)--(r23); 
\draw[very thick] (r23)--(r35);
\draw[very thick] (r13)--(r35);

\draw[very thick] (r31)--(r41);
\draw[very thick] (r31)--(r42);
\draw[very thick] (r32)--(r41);
\draw[very thick] (r24)--(r42); 
\draw[very thick] (r13)--(r35); 

\draw[very thick] (r41)--(r51);
\draw[very thick] (r23)--(r51); 
\draw[very thick] (r23)--(r35);

\draw[very thick] (r51)--(r61);
\draw[very thick] (r42)--(r61);
\draw[very thick] (r35)--(r61); 
\end{scope}

\end{scope}    
\end{tikzpicture}
\end{center}
\caption{The left figure is the weak order $\fS_4$ with the equivalence classes under the $G$-equivalence relation shaded. To the right we have $\tp{P_4}$, written as induced subposet of $\fS_4$ on $G$-permutations. The adjacent transpositions and hyperplanes are marked in the rightmost maximal chain of $\fS_4$. }
\label{fig:PathEquiv}
\end{figure}

\omitt{\tvi{Sam (12/18): Could we make the pink color darker and the circles around the numbers the same thickness as the pink line segments? Could we also add these circles on the diagram on the right? If not that's ok too.}
\susanna{Is this better? 3/24/24}
\tvi{Beautiful (4/4)}
}
Barnard and McConville~\cite{BM} proved that $G$-equivalence is a
lattice congruence when $G$ is filled and connected, so that the
quotient $\tp{G}$ is a lattice and the projection is
order-preserving. They described the minimal elements in each class:
the \mydef{$G$-permutations.} We use this description of the tubing
lattice as our working definition: the sublattice of $\fS_N$ induced
by the $G$-permutations. We also reference the projection from $\fS_N$ and the
description of the equivalence classes.

\begin{definition}\label{def:Gperm}
 For a graph $G$ we call a permutation $w=[w(1),w(2),\cdots, w(n)]\in\fS_n$ a {\it $G$-permutation} if $w(i)$ is in the same connected component as $\max\{w(1),w(2),\ldots, w(i)\}$ in the induced graph $G|_{\{w(1),w(2),\ldots, w(i)\}}$ for all $i\in[n]$.
 \end{definition}
 
See Figure~\ref{fig:Gtree} for an example of a $G$-permutation.  Each
$G$-permutation $w$ will be in a different $G$-equivalence class of
$\fS_n$, and each class has
exactly one $G$-permutation. The set of vertices for $\tp{G}$ is
$\{[\![w]\!]:w\text{ is a $G$-permutation}\}$. The partial order on 
$\tp{G}$ is $[\![u]\!]\leq[\![v]\!]$ if $w\le v$ in the weak order. See Figure~\ref{fig:PathEquiv} for an example.

\begin{figure}
\begin{center}
\begin{tikzpicture}[scale=.7]
\coordinate (A) at (0,0);
\coordinate (B) at (-1,-1);
\coordinate (C) at (1,-1);
\coordinate (D) at (1,-2);
\coordinate (E) at (-1,-2);
\coordinate (F) at (-1,-3);
\draw[black] (D)--(C)--(A)--(B)--(E)--(F);
\filldraw[black] (A) circle [radius=2pt] node[above] {$4$};
\filldraw[black] (B) circle [radius=2pt] node[left] {$1$};
\filldraw[black] (C) circle [radius=2pt] node[right] {$5$};
\filldraw[black] (D) circle [radius=2pt] node[right] {$6$};
\filldraw[black] (E) circle [radius=2pt] node[left] {$3$};
\filldraw[black] (F) circle [radius=2pt] node[left] {$2$};
\end{tikzpicture}
\end{center}
\caption{The permutation $[2,3,1,6,5,4]$ is a $L_{4,2}$-permutation and $[6,2,3,5,1,4]$ is not, but $[6,2,3,5,1,4]\in[\![231654]\!]$ since their $L_{4,2}$-trees are the same. }
\label{fig:Gtree}
\end{figure}
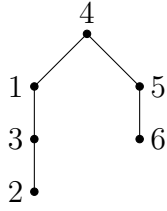

\omitt{ 
\begin{theorem}\cite[Theorem 1.1]{BM}
The map $\Psi_G:\fS_n\rightarrow \tp{G}$, $w\mapsto [\![w]\!]$, is a lattice quotient map if and only if $G$ is filled.
\end{theorem}
 }
 
\begin{remark}The weak order lattice $\fS_n$ is $\tp{K_n}$ and $\tp{P_n}$ is the Tamari lattice \cite{ronco2012}.
\end{remark}
 
Note that some chains in the weak order may project to a chain of the same length in $\tp{G}$ and other chains become shorter because they contain permutations that
are not $G$-permutations.  In chains which become shorter, some cover relations in weak order are between two permutation in the same $G$-equivalence class.  Lemma~\ref{lem:ab-cut} gives a condition for determining when a cover relation in the weak order ``collapses'' under the projection to $\tp{G}$. 


\omitt{Using the definition of the equivalence classes by $G$-trees  we can determine whether a cover relation in the weak order is between two permutations in the same $G$-equivalence class.}  
We say a subset $S\subseteq V$ of  vertices is an {\it $ab$-cut set} if vertices $a,b\notin S$ and   $a$ and $b$ are in different connected components of $G|_{V-S}$. 

\begin{lemma} Let $G$ be a connected graph on vertex set $V$, where $V$ has $N$ vertices. Two permutations $u\lessdot v$ in the weak order (so $u=v\circ s_i$ for some $i$) are in the same equivalence class, $[\![u]\!]=[\![v]\!]$, if and only if the set $S=\{u(i+2),u(i+3),\ldots, u(N)\}$ is  a $u(i)u(i+1)$-cut set in $G$. 
\label{lem:ab-cut}
\end{lemma}

\begin{proof}
Suppose that $u$ and $v$ are as in the lemma's statement. We will show
that $S=\{u(i+2),u(i+3),\ldots, u(N)\}$ is a $u(i)u(i+1)$-cut set in
$G$ if and only if $u$ and $v$ have the same $G$-tree. This would prove the 
statement by the definition of $G$-equivalence.

We use induction on $N$. The base case is 
$N=2$. In this case $u=[1,2]$, $v=[2,1]$ and $G=K_2$. We can confirm $u$ and
$v$ have distinct $G$-trees and are in distinct equivalence classes.

Now assume that $N>2$. First we consider the case that
$S=\{u(i+2),u(i+3),\ldots, u(N)\}$ is a $u(i)u(i+1)$-cut set.  Because
$G$ is connected and $S$ is nonempty, $i+2\le N$. Thus $u(N)=v(N)$ and
$u$'s and $v$'s $G$-trees have the same root. Let $V'=V-\{u(N)\}$, where $V=V(G)$, and $S'=S-\{u(N)\}$. We have $V-S=V'-S'$, and $G|_{V-S}=G|_{V'-S'}=G_1\cup G_2\cup\cdots\cup G_k$, where
$G_1,G_2,\ldots,G_k$ are the connected components of
$G|_{V-S}$. Since $S$ is a $u(i)u(i+1)$-cut set, $u(i)$ and
$u(i+1)$ are in different components of $G|_{V'-S'}$. 
\omitt{
\tvi{Sam (Dec 18): Comment not addressed yet. Note that $u(i)$ and $u(i+1)$ are in different $G_j$ only if $u(N)$ by itself makes a $u(i)u(i+1)$-cut set. We are  assuming that it might take all of $\{u(i+2),\ldots, u(N)\}$ to form the cut set, so we also need to consider the case where $u(i)$ and $u(i+1)$ are in the same $G_j$. }
\susanna{I think the converse is in the last paragraph. 12/15/23}
\susanna{12/24/23 I think I understand my confusion and have changed $V-\{u(N)\}$.  }
\tvi{(Dec 26) The proof is correct and looks great. One comment: in $V'=V-\{u(N)\}$ and further, can we use $V(G)$ instead of $V$? It took me a moment to remember that $V$ might be the collection of vertices, but if we use $V(G)$ we won't confuse our readers.}
\susanna{There wasn't even a definition of $V$! I put it in the statement of the lemma. Is this ok? If so, we can remove all the comments.}
\tvi{Yes, this looks good. Comments are ready to be deleted. (3-19-24)}
}
The subwords of $u$ and $v$ restricted to the connected components are
therefore identical and produce the same $G$-trees on each component. The $G$-trees of $u$ and $v$ are identical, so $[\![u]\!]=[\![v]\!]$ when $S$ is a $u(i)u(i+1)$-cut set.

Continue to assume $N>2$, but now that $S=\{u(i+2),u(i+3),\ldots, u(N)\}$ is not a
$u(i)u(i+1)$-cut set of $G$. That is, $u(i)$ and $u(i+1)$ are in the same
component of $G|_{V-S}$. Say they are both in $G_1$ with vertices $V_1$. Denote the restriction of
$u$ to $V_1$ by $u'$ and of $v$ by $v'$. Then for some $i'\le i$,
$u's_{i'}=v'$. Since $\{u'(i'+2),\ldots,u'(|V_1|)\}\subseteq S$,
$\{u'(i'+2),\ldots,u'(|V_1|)\}$ is not a $u'(i')u'(i'+1)$-cut set. By
induction, $u'$ and $v'$ have different $G_1$-trees, which forces $u$
and $v$ to have different $G$-trees.
\end{proof}

\subsection{Hyperplane walks and the projection to $\tp{G}$}

\omitt{
\susanna{Do hyperplanes separate perms or connect them? Maybe we should have a definition when we introduce hyperplane walks.}
\susanna{Added def in Section~\ref{subsec:weakorder} for hyperplane separating two perms. 1/2/24}
}

Recall that chains in the weak order can be seen as hyperplane walks
in $\cH(\omega^{(n)})$. Fix a graph $G$ and a hyperplane walk $H$. We
will now describe which \omitt{consecutive} hyperplanes separate
permutations in the same $G$-equivalence class. If $H_{a,b}\in {H}$ separates two permutations in the same $G$-equivalence class, we say that $H_{a,b}$ is an {\it intra-class hyperplane} in ${H}$, otherwise $H_{a,b}$ is an {\it inter-class
hyperplane} in ${H}$. For an example, refer to the maximal chain on the right of $\fS_4$, shown on the left in Figure~\ref{fig:PathEquiv}. The hyperplane $H_{3,4}$ is not an intra-class hyperplane for this walk, whereas $H_{2,4}$ is.  

\omitt{In a balanced tableau we call a cell an {\it
intra-class cell} if it is associated to an intra-class
hyperplane. All other cells will be called {\it inter-class cells}.
\susanna{non edge cells and edge cells here}
}

Let $G$ be a filled, connected graph on $[N]$. Given two hyperplane walks in
$\fS_n$ and their corresponding chains of permutations, the two chains of permutations will project to the same chain in
$\tp{G}$ if removing all intra-class hyperplanes in each walk produces the same sequence of hyperplanes. We will call these two
hyperplane walks {\it $G$-equivalent}. This can be seen in
Figure~\ref{fig:PathEquiv}.
\label{rem:same_hyperplanewalk_check}  Let $\MH{G}$ be the collection of all hyperplane walks with their intra-class hyperplanes removed. 
\omitt{\sam{Let $\MH{G}$ be the collection of all hyperplane walks after a projection into $\tp{G}$ are removing all intra-class hyperplanes. (4/4)}
\tvi{New text added in red. Delete red if OK. Add example of $\MH{G}$? (4/4)}
\susanna{Is this ok? Let $\MH{G}$ be the collection of all hyperplane walks with their intra-class hyperplanes removed. These are the hyperplane walks corresponding to $\tp{G}$. }
\susanna{I changed your red text to a box here. Also, I wouldn't add an example here. 4/8/24}
} 

We want to determine whether two hyperplane walks in $\fS_n$ are $G$-equivalent in terms of their associated balanced tableaux. We will also characterize the balanced tableaux of hyperplane walks which are $G$-equivalent only to themselves and therefore correspond to longest length maximal chains in $\tp{G}$.

\omitt{

Let $\lambda=(\lambda_1,\lambda_2,\ldots, \lambda_l)$ be an integer partition and $\mu=(\mu_1,\mu_2,\ldots,\mu_k)$ be another such that $\lambda_i>\mu_i$ for all $i\in[k]$ where $l<k$. Looking at the Young diagram of $\lambda$ and $\mu$, we have $\mu$ sitting inside of $\lambda$. In any column, there is a range of cells that are outside of $\mu$, but are in $\lambda$. If $[i,j]=\{i, i+1,\ldots, j\}$ is the range of rows outside of $\mu$ but inside $\lambda$ for some column $c$, we call the interval $[i,j]$ a {\it row difference interval} in $\lambda// \mu$. See Figure~\ref{fig:intraEX} for an example. 
\tvi{When defining integer partition define these skew partitions too. }

Let $G$ be a filled graph on $n$ vertices. The {\it non-edge partition} is $\mu=(\mu_1,\mu_2,\ldots,\mu_k)$ where $\mu_i$ is $n-i$ minus the number of edges $ij$ where $i<j$. Exclude any parts that become 0. The idea is that there are $n-i$ possible edges $ij$ for any $i$ with $i<j$. We are having $\mu_i$ count the number of vertices greater than $i$ that are not adjacent to $i$. Given a filled graph $G$ on $[n]$ define the {\it row difference intervals of $G$} to be the row difference intervals of $\stair_{n} // \mu$ where $\mu$ is the non-edge partition of $G$. 
} 

\omitt{
\susanna{Do we use this lemma anywhere? Do we use $H_{a,b,c}$ anywhere?}
\tvi{Sam (Dec 18): No we don't use Lemma~\ref{lem:bal_L2} anywhere. It can be removed. }
\begin{lemma}
Let $G$ be a filled graph and  $\mathcal{H}$ be a hyperplane walk in $\mathcal{B}_n$ and $T$ be its associated balanced tableau. We have that $H_{a,b,c}\in L_2(\mathcal{H})$ if and only if the filling in cell $(a,n-c+1)$ is larger than the filling in cell $(b,n-c+1)$. 
\label{lem:bal_L2}
\end{lemma}
\begin{proof}
If $H_{a,b,c}\in L_2(\mathcal{H})$ that means that the braid move involving $H_{a,b}H_{a,c}H_{b,c}$ occurred and in $\cH$ we have these hyperplanes appearing in order $H_{b,c}H_{a,c}H_{a,b}$, but they may not be adjacent. This happens exactly when $(a,n-c+1)$ is larger than the filling in cell $(b,n-c+1)$ because $H_{a,b}H_{a,c}H_{b,c}$  can only be in two possible orders, forwards and backwards. 
\end{proof}
}
\omitt{
\susanna{Is a filled graph always connected?}
\tvi{According to the definition, not all filled graphs have to be connected. all graphs with no edges are filled.}
\tvi{(Dec 18) I think a lot of the theory we developed assumes we are working with a connected filled graph. }
}
\begin{lemma}
Let $G$ be a connected, filled graph with edge set $E$.  A set $S$ is an $ab$-cut set of $G$ if and only if we can find a $c\in[a+1,b]$  where $ac$ is not an edge of $G$, and $S$ contains all $j$  such that $j<c$ and $jc\in E$. 
\label{lem:filled_cut}
\end{lemma}
\begin{proof}
First suppose we have a vertex $c\in[a+1,b]$ such that $ac\notin E$
and a set $S$ such that $a,b\notin S$ and $\{j<c|jc\in E\}\subseteq
S$. We show that $S$ is an $ab$-cut set. Note that since $G$ is filled
and $ac$ is not an edge, there is not an edge $jc$ where $j\le
a$. Further, $\{j<c:jc\in E\}=[c',c-1]$ for some $c'>a$. By
assumption, $[c',c-1]\subseteq S$.

It will be sufficient to show the following. Every path $P$ between
$a$ and $b$ must contain some vertex in $[c',c-1]$. Suppose, for
contradiction, that we have a path $v_0=a,\ldots,v_{k-1},v_k=b$ that
does not contain a vertex of $[c',c-1]$. Let $i$ be the smallest index
such that $v_i\ge c$, which means that $v_{i-1}<c$ and therefore
$v_{i-1}<c'$. Because $G$ is filled and $v_{i-1}v_i\in E$, the edges
$v_{i-1}c$ and $cv_i$ exist. But $v_{i-1}c\in E$ forces $c'\le
v_{i-1}$, which is a contradiction.

Conversely, assume $S$ is an $ab$-cut set. The set $S\cap[a+1,b]$ is the union of disjoint, nonempty intervals $[b_1+1,c_1-1], [b_2+1,c_2-1], \ldots, [b_k+1,c_k-1]$ where $a=c_0\leq b_1<c_1\leq b_2<c_2\leq \cdots \leq b_k<c_k\leq b=b_{k+1}$.
\omitt{Let
$\mathcal{C}=\{x\in[a+1,b]|x\notin S,x-1\in S\}$ and let
$\mathcal{B}=\{x\in[a+1,b]|x\notin S,x+1\in S\}$. The set
$\mathcal{C}\ne\emptyset$ because $b\not\in S$ and
$\mathcal{B}\ne\emptyset$ because $a\not\in S$. 
\tvi{Why is $\mathcal(C)\neq \emptyset$? I can see that $b\notin S$, but we may not have $b-1\in S$. I have a similar question for $\mathcal{B}\neq \emptyset$. It might be easier to define your $b_i$'s and $c_i$'s by saying that $S\cap [a+1,b]$ is the union of disjoint intervals $[b_1+1,c_1-1], [b_2+1,c_2-1], \ldots, [b_k+1,c_k-1]$ where so $a\leq b_1<c_1\leq b_2<c_2\leq \cdots \leq b_k<c_k\leq b=b_{k+1}$. Your picture was great at explaining what you are doing. The concept of the proof looks correct up to these small details although I haven't fully read it through with detail.}
Write $\mathcal{C}=\{
a<c_1<c_2<\cdots<c_k\le b\}$ and $\mathcal{B}=\{ a\le
<b_1<b_2<\cdots<a_k< b\}$. 
}We have that $a=c_0\le
b_1<c_1\le b_2<\cdots<b_k<c_k\le b=b_{k+1}$, $G|_{[c_i,b_{i+1}]}$ is connected,
and $[c_i,b_{i+1}]$ is contained in the complement of $S$, whereas
$(b_i,c_i)$ is contained in $S$. See Figure~\ref{fig:abCut}.

\begin{figure}[h]
\centering
\begin{tikzpicture}
\def\d{.7}
\def\r{2}
\def\v{.5}
\coordinate (v-1) at (-\d*2,0);
\coordinate (v0) at (-\d*1,0);
\coordinate (v1) at (\d*0,0);
\coordinate (v2) at (\d*1,0);
\coordinate (v3) at (\d*2,0);
\coordinate (v4) at (\d*3,0);
\coordinate (v45) at (\d*3.5,0);
\coordinate (v5) at (\d*4,0);
\coordinate (v6) at (\d*5,0);
\coordinate (v7) at (\d*6,0);
\coordinate (v8) at (\d*7,0);
\coordinate (v9) at (\d*8,0);
\coordinate (v10) at (\d*9,0);
\coordinate (v11) at (\d*10,0);
\coordinate (v12) at (\d*11,0);
\coordinate (v13) at (\d*12,0);
\coordinate (v14) at (\d*13,0);
\coordinate (v15) at (\d*14,0);
\coordinate (v16) at (\d*15,0);
\coordinate (v17) at (\d*16,0);

\fill[red!30] (v45) circle[x radius=2*\d,y radius=.7*\d]; 
\fill[red!30] (v11) circle[x radius=1.5*\d,y radius=.7*\d]; 

\node[above=\v,black] at (v45) {$S$};
\node[above=\v,black] at (v11) {$S$};

\draw[black,dashed](v-1)--(v0);
\draw[black,dashed](v16)--(v17);
\draw[black](v0)--(v1)--(v2)--(v3)--(v4)--(v16);

\fill[red](v1) circle(\r pt) node {};
\fill[red](v2) circle(\r pt) node[below=\v,blue] {$b_i$};
\fill[red](v3) circle(\r pt) node {};
\fill[red](v4) circle(\r pt) node {};
\fill[red](v5) circle(\r pt) node {};
\fill[red](v6) circle(\r pt) node {};
\fill[red](v7) circle(\r pt) node[below=\v,blue] {$c_i$};
\fill[red](v8) circle(\r pt) node {};
\fill[red](v9) circle(\r pt) node[below=\v,blue] {$b_{i+1}$};
\fill[red](v10) circle(\r pt) node {};
\fill[red](v11) circle(\r pt) node {};
\fill[red](v12) circle(\r pt) node {};
\fill[red](v13) circle(\r pt) node[below=\v,blue] {$c_{i+1}$};
\fill[red](v14) circle(\r pt) node {};
\fill[red](v15) circle(\r pt) node {};
\end{tikzpicture}
\caption{The figure represents a portion of a filled, connected graph $G$: some of the vertices in $[a,b]$. The proof of Lemma~\ref{lem:filled_cut} focuses on this portion. We have drawn only edges between consecutive vertices and $G$ may have more edges.}
\label{fig:abCut}
\end{figure}
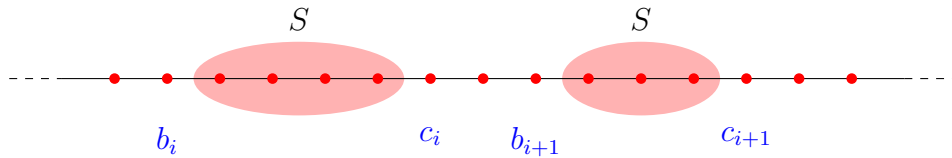

For each $i, 1\le i\le k$, let $c'_i=\min\{c|cc_i\in E\}$. It is
sufficient to show that there is an $i, 1\le i\le k$ such that
$[c_i',c_i-1]\subseteq S$ and $ac_i$ is not an edge.

Let $i, 1\le i\le k$ be such that $[c'_i,c_i-1]\not\subseteq S$. There is a $b'_i\in[c'_i,c_i-1]$ where $b_i'\notin S$. Since $G$ is filled, $b_i'c_i\in E$. There is then for each $i$ an $1\le h=h_i<i$ such that $c_{h-1}\leq
b_i'\le b_h$ and the distinct induced subgraphs $G|_{[c_{h-1},b_h]}$ and
$G|_{[c_i,b_{i+1}]}$ are connected to each other. We have shown that each interval $[c_i,b_{i+1}]$ in the complement of $S$ such that $[c'_i, c_i-1]\not\subseteq S$ is connected to at least one interval that is also in the complement of $S$ and is strictly to its left. For contradiction, suppose there is no $i, 1\le i\le k$, such that $[c'_i,c_i-1]\subseteq S$. We show by induction on $k$ that we can build a path from $a$ to $b$. If $k=1$, there is only one interval, namely $[c_0a,b_1]$, strictly to the left of $[c_1,b_2=b]$. They are therefore connected and there is a path from $a$ to $b$. Assume the statement holds for $k-1$. There is then a path from $a$ to $b_k$ and an edge from $c_k$ to $b'_k\in[c_{h-1},b_h]$ for some $h\le k$. We may then construct a path from $a$ to $b$ that is in the complement of $S$, contradicting $S$ being an $ab$-cut set.  

Therefore, there is at least
one $i,1\le i\le k$ such that $[c'_i,c_i-1]\subseteq S$. Since $[c'_i,c_i-1]\subseteq S$ and $a\notin S$, $ac_i\notin E$. We may set the $c$ described in the statement of the lemma equal to any of the $c_i$ such that $[c'_i,c_i-1]\subseteq S$. \omitt{Suppose that $G$ is filled with edge set $E$. Pick a $c\in[a+1,b]$ where $ac\notin E$. Note that since $G$ is filled and $ac$ is not an edge, there is not an edge $jc$ when $j\leq a$. Otherwise $ac$ would have to be an edge. Further, the collection of $j<c$ where $jc$ is an edge is some interval $[c',c-1]$ for $a<c'$. It will be sufficient to show the following. Given any path $P$ that starts at $a$ and ends at $b$, this path must contain some vertex in $[c',c-1]$. 

Suppose the opposite that we have a path $P$, $v_0=a,v_1,v_2,\ldots, v_{k-1},v_k=b$, that does not contain a vertex in $[c',c-1]$. Let $i$ be the smallest index with $v_i\geq c$. This means that $v_{i-1}<c'$. Because $G$ is filled if $v_i=c$, then we have the edge $v_{i-1}c$. If $v_i<c$, then we have the edges $v_{i-1}c,cv_i$. In any case we have the edge $v_{i-1}c$ with $v_{i-1}<c'$, which according to our set up is impossible. }
\end{proof}

Let  $G$ be  a graph  on vertex set $[N]$ and  $T$ a
tableau  of  shape  $(N-1,N-2,\ldots,2,1)$.   We  say  that  the  cell
$(a,N+1-b)$ is an \mydef{edge cell} of $T$  if $ab$ is an edge of $G$;
otherwise, it is a \mydef{non-edge cell}.

In addition, let $G$ be filled and connected. For each column $j$ of $T$, there
is the set $I_j$ of row indices of the edge cells in that
column. Because $G$ is a filled graph, the set is in fact an interval,
which we call an \mydef{edge interval} of $G$. In other words,
$I_j=\{i|\{i,N-j+1\}\text{ is an edge of $G$}\}=[c,d]$, for some $1\le c,d\le N-1$. Since $G$ is filled and connected, $N-j$ is always an element of $I_j$. See Figure~\ref{fig:intraEX}.

\begin{theorem}
Let $G$ be a filled and connected graph on $N$ vertices with edge set $E$, let ${H}$ in $ \cH(\omega^{(N)})$ be a hyperplane walk, and let $T$ be the balanced tableau corresponding to ${H}$.  The hyperplane $H_{a,b}\in{H}$ is an intra-class hyperplane if and only if any one of these equivalent conditions holds.
\begin{enumerate}
\item\label{thmpart:connect} $H_{a,b}$ separates two permutations $v=u\circ s_i$ where $\{u(i+2),u(i+3),\ldots, u(n)\}$ is an $ab$-cut set. 
\item \label{thmpart:before} There is a $j\in[N]$ such that $N-j+1\in[a+1,b]$, $\{a,N-j+1\}\notin E$, and for all $i$ in the edge interval $I_j$, $i<b$ and the hyperplane $H_{i,b}$ precedes $H_{a,b}$ in $H$. 
\item \label{thmpart:baltab} There is $j\in[N]$ such that $N-j+1\in[a+1,b]$, $\{a,N-j+1\}\notin E$, and for all $i$ in the edge interval $I_j$, $i<b$ and the entry in the cell $(a,N-b+1)$ in $T$ is larger than the entry in $(i,N-b+1)$. In this case, we call the cell $(a,N-b+1)$ an \mydef{intra-class cell} of the tableau $T$.
\end{enumerate}
\label{thm:intra_class_conditions}
\end{theorem}
\begin{proof} Let ${H}\in \cH(\omega^{(n)})$ for a filled and connected graph $G$ with edge set $E$ and write $H=H_{a_1,b_1}H_{a_2,b_2}\ldots H_{a_M,b_M}$ for $M=\binom{N}{2}$. Denote the chain of permutations in $\fS_N$ which corresponds to $H$ by $\{\ii=\wch{0}\lessdot\wch{1}\lessdot\wch{2}\lessdot\cdots\lessdot\wch{M}=\longestk{N}\}$ and the reduced word by $s_{i_1}s_{i_2}\cdots s_{i_M}$, where $\wch{k}=\wch{k-1}s_{i_k}$ and $H_{a_k,b_k}$ separates $\wch{k-1}$ and $\wch{k}$.   
We know that $\wch{k-1}$ and $\wch{k}$
 are in the same $G$-equivalence class if and only if $S=\{\wch{k}(i_k+2),
 \wch{k}(i_k+3), \ldots, \wch{k}(N)\}$ is an $a_{k-1}b_{k-1}$-cut set in $G$ by
 Lemma~\ref{lem:ab-cut}.  Thus \eqref{thmpart:connect} is equivalent to $H_{a,b}$ being an intra-class hyperplane.
 
 By Lemma~\ref{lem:filled_cut}, \eqref{thmpart:connect} is equivalent to the existence of $c$ such that $c\in[a_{k-1}+1, b_{k-1}]$ with $a_{k-1}c\notin
 E$ and such that the set $S$ contains all $i$ where $ic\in E$ and $i<c$. The set $\{i:i<c\text{ and }ic\in E\}$ is the edge interval $I_{N-c+1}$. For $i\in I_{N-c+1}$, we have $i<c\le b$, and $i\in S$, so that $i$ and $b$ correspond to an inversion in $\wch{k-1}$. Thus $H_{i,b}$ precedes the $k^{\text{th}}$ hyperplane, namely $H_{a,b}$, and \eqref{thmpart:connect} implies \eqref{thmpart:before}. Since the entries in the tableau $T$ indicate the position of the hyperplanes in the walk, \eqref{thmpart:before} implies \eqref{thmpart:baltab}. Likewise, \eqref{thmpart:baltab} implies \eqref{thmpart:before}, again using the definition of the correspondence between balanced tableaux and hyperplane walks. 

 Finally, suppose \eqref{thmpart:before} holds: There is a$j\in[N]$ such that $N-j+1\in[a+1,b]$, $\{a,N-j+1\}\notin E$, and for all $i$ in the edge interval $I_j$, the hyperplane $H_{i,b}$ precedes $H_{a,b}$ in $H$. Suppose $H_{a,b}$ is the $k^{\text{th}}$ hyperplane in the walk $H$ and that $h$ is such that $\wch{k-1}(h)=a$ and $\wch{k-1}(h+1)=b$. Let $S=\{\wch{k-1}(h+2),\ldots,\wch{k-1}(N)\}$ and let $c=N-j+1$. The number $c\in[N]$ is an element of $S$. By Lemma~\ref{lem:filled_cut}, in order to show \eqref{thmpart:connect}, it is enough to show that $c\in[a+1,b]$, $ac\notin E$, and $S$ contains all $i$ such that $i<c$ and $ic\in E$. By hypothesis, $c\in[a+1,b]$ and $ac$ is not an edge of $G$. Let $i<c$ and $ic\in E$. Then $i\in I_j$, so that by hypothesis, $H_{i,b}$ precedes $H_{a,b}$ in $H$. Therefore $(i,b)$ is an inversion in $\wch{k-1}$, so that $w(\ell)=i$ for $\ell>h+1$. Therefore $i\in S$ and \eqref{thmpart:before} implies \eqref{thmpart:connect}. \omitt{
 Since $G$ is filled and connected, ${c-1,c}$ is an edge of $G$, so that $c-1\in I_j$. In particular, $H_{c-1,b}$ is a hyperplane, so that $(c-1$
 
 Let $J$ be
 the collection of all these $j$'s. Note that $J$ is a row difference
 interval of $G$ \tvi{(Justify more? Perhaps a lemma how these certain
 intervals are row difference intervals of unit interval graphs.)} so
 \eqref{thmpart:connect} if and only if \eqref{thmpart:temp}.  Considering where $a_k$ is in the
 permutation $w_k$, adjacent to $b_k$. We know that all $j\in J$
 appear right of $b_ka_k$ in $w_k$.  Meaning that $j\in J$ and $b_k$
 appear in decreasing order, since all $j\in J$ have the additional
 property where $j<b_k$. This means by
 Remark~\ref{rem:hyperwalks_perms_inv} that the hyperplane $H_{j,b_k}$
 is left of $H_{a_k,b_k}$ in ${H}$. Hence (3) if and only if (2).
As a result in the balanced tableau the cell $(a_k,n-b_k+1)$ has an entry that is larger than all the entries in cells $(j,n-b_k+1)$ for all $j\in J$ by Lemma~\ref{lem:bal_L2}. Hence, (2) if and only if (4).
} 
\end{proof}

\omitt{
\begin{remark}
The first condition in Theorem~\ref{thm:intra_class_conditions} does not require $G$ to be filled. 
\end{remark}
\susanna{Maybe we can get rid of this remark? It is not so clear anymore. 12/11/2023}
\tvi{(Dec 18) I'm OK removing it. }
}
\begin{example}
We examine the hyperplane walk given in Figure~\ref{fig:fourassoc} in the case $G=P_4$ with Theorem~\ref{thm:intra_class_conditions} in mind. See also Figure~\ref{fig:PathEquiv}. The edge intervals are $I_1=\{3\}$, $I_2=\{2\}$, and $I_3=\{1\}$. The hyperplane $H_{1,3}$ is an intra-class hyperplane. The permutation $\wch{1}$ is $[1324]$, $\wch{2}$ is $\wch{1}\circ s_{i_2}=[1324]s_1=[3124]$, $a$ is $\wch1(1)=1$, $b$ is $\wch1(2)=3$ and the $ab$-cut set is $S=\{2,4\}$. Lemma~\ref{lem:filled_cut} guarantees a $c$, $c\in[a+1,b]=[2,3]$ and $ac\notin E$, such that $S$ contains all $j$ such that $j<c$ and $jc\in E$. In this case, $c=3$ and $N-c+1=2$. Since $I_2=\{2\}$, we confirm that for all $j\in I_2$, $H_{j,b}$ precedes $H_{a,b}$ in $H$. The entry in row $1$, column $2$ of the balanced tableau is larger than the entry in row $2$, column $2$.

The hyperplane $H_{1,4}$ is also intra-class for this walk and graph. For this hyperplane, the fourth in the walk, $a=1$, $b=4$, $i_4=2$, $c\in[2,4]$. It separates $\wch{3}=[3142]$ and $\wch{4}=[3412]$, so that $S=\{2\}$ and $c=2$. Again, we confirm that for all $j\in I_2$, $H_{j,4}$ precedes $H_{1,4}$ in $H$. The entry in row $1$, column $1$ of the balanced tableau is larger than the entry in row $2$, column $1$.  
\end{example}

Recall that two hyperplane walks are $G$-equivalent if removing the intra-class hyperplanes from both of them results in the same sequence of hyperplanes. Theorem~\ref{thm:intra_class_conditions} gives us the condition for determining those hyperplanes, both in terms of walks and in terms of balanced tableaux. In a balanced tableau, we remove the entries in the the intra-class cells, then standardize the remaining entries by replacing the $i^{\text{th}}$ smallest with $i$. See Figure~\ref{fig:intraEX}. The new tableau is the $G$-balanced tableau of the walk. Two walks are $G$-equivalent if their $G$-balanced tableaux are the same.   Let $\MB{G}$ be the collection of all $G$-balanced tableaux.
\omitt{\tvi{New text added in red. Delete red if OK.  (4/4)}


\tvi{Sam (12/18): Could we add a proposition that summarize all objects that are in bijection with maximal chain of $\tp{G}$ in general?}
\susanna{I can add one at the end of Section~\ref{sec:mlcinlollis} for maximal length chains and we already have the list in Theorem~\ref{thm:lollipop chains}. Do you want something different from either of these? 1/3/24}
\tvi{(3-19-24) Yes, it would be different. In this section we are talking about $\tp{G}$ in general. In Sections 4 and 5 we talk about only $\tp{L_{m,n}}$. In any case, what exactly we add or don't add can be determined as we do a full read through.}
\tvi{(4/4) Here is a proposition I am proposing to include (Theorem~\ref{thm:general chains}).}
}
\begin{theorem}
\label{thm:general chains}
The following are in bijection and are in bijection with maximal chains in $\tp{G}$ for a connected filled $N$-vertex graph $G$.
\begin{enumerate}
\item The subsequence of all $G$-permutations within a maximal chain in the weak order $\fS_N$. 
\item \omitt{The collection of hyperplane walks projected onto $\tp{G}$ with intra-hyperplanes deleted, $\MH{G}$.}The collection $\MH{G}$ of hyperplane walks of $\fS_n$ with their intra-class hyperplanes deleted. 
\omitt{\susanna{I would write this as The collection $\MH{G}$ of hyperplane walks of $\fS_n$ with their intra-class hyperplanes deleted. Otherwise, the theorem is fine. 4/17/24}}
\item The $G$-balanced tableaux, $\MB{G}$.
\end{enumerate}
\end{theorem}

\omitt{
Just like in Remark~\ref{rem:same_hyperplanewalk_check} we can determine whether two hyperplane walks are the same by removing all intra-class hyperplanes, we will determine whether two balanced tableaux are the same chain in $L_G$ by removing the intra-class cells.  In this way we will have tableaux on Young diagrams with holes, so it is no longer technically a Young diagram, but a subset of a grid of cells. Suppose that a Young diagram, with perhaps some holes, has each cell filled with a positive integer. We can {\it standardize} the tableau by replacing the $i$th smallest integer with $i$. 

\begin{corollary}
Let $G$ be a connected, filled graph. Suppose that $T_1$ and $T_2$ are two balanced tableaux associated to two hyperplane walks in $\fS_n$. The two tableaux become the same after removing all intra-class cells and standardizing, then they are the same chain in $\tp{G}$. 
\end{corollary}
\begin{proof}
This is true by Remark~\ref{rem:same_hyperplanewalk_check} and Theorem~\ref{thm:intra_class_conditions}. 
\end{proof}
}
\ytableausetup
{mathmode, boxsize=1.8em}
\begin{figure}
\begin{center}
\begin{tikzpicture}
\node at (5,3){
\begin{tikzpicture}
    
\coordinate (A) at (0,.5);
\coordinate (B) at (1,0);
\coordinate (C) at (1,1);
\coordinate (D) at (2,0);
\coordinate (E) at (2,1);
\coordinate (F) at (3,.5);
\draw[black] (B)--(C)--(A);
\draw[black] (B)--(D)--(E)--(B);
\draw[black] (F)--(E)--(D)--(F);
\draw[black] (D)--(E)--(C);
\filldraw[black] (A) circle [radius=2pt] node[left] {$1$};
\filldraw[black] (B) circle [radius=2pt] node[below] {$3$};
\filldraw[black] (C) circle [radius=2pt] node[above] {$2$};
\filldraw[black] (D) circle [radius=2pt] node[below] {$5$};
\filldraw[black] (E) circle [radius=2pt] node[above] {$4$};
\filldraw[black] (F) circle [radius=2pt] node[right] {$6$};
\end{tikzpicture}};

\node at (0,1) {\begin{ytableau}
{H_{16}} & {H_{15}} & {H_{14}} & {H_{13}} &{H_{12}}\\
    {H_{26}}&{H_{25}}&H_{24}&H_{23}\\
    H_{36}&H_{35}&H_{34}\\
    H_{46}&H_{45}\\
    H_{56}
\end{ytableau}};
\ytableausetup{mathmode, boxsize=1.2em}


\node at (4.3,0) {\begin{ytableau} *(orange)&*(orange)&*(orange)&*(orange)&{ }\\
*(orange)&*(orange)&{ }&{ }\\
*(orange)&{ }&{ }\\
{ }&{ }\\
{ }\cr \end{ytableau}};

\node at (8,0){
\begin{ytableau} {9}&{6}&{4}&{2}&{13}\cr {8}&{5}&{3}&{1}\cr {12}&{10}&{15}\cr {11}&{7}\cr {14}\cr \end{ytableau}
};
\node at (11.7,0){
\begin{ytableau} { }&{ }&{ }&{ }&{9}\cr {5}&{3}&{2}&{1}\cr {8}&{6}&{11}\cr {7}&{4}\cr {10}\cr \end{ytableau}
};
\end{tikzpicture}
\end{center}
\caption{ The graph $G$ (top and center) is connected and filled. The first tableau indicates the correspondence between hyperplanes and cells.
In the second tableau on the left, the non-edge cells are orange. The edge intervals are $I_1=\{4,5\}$, $I_2=\{3,4\}$, $I_3=\{2,3\}$, $I_4=\{2\}$, and $I_5=\{1\}$. The third tableau is balanced and for reference, call it $T$. The cell $(1,1)$ of $T$ is intra-class, since the entry $9$ in cell $(1,1)$ (row 1 and column 1) is larger then the entry in cell $(i,1)$ for all $i\in I_4$. The cells $(1,2)$, $(1,3)$, and $(1,4)$ are also intra-class. The tableau on the right is the tableau with the entries deleted from the intra-class cells and standardized.
}

\omitt{For the unit interval graph above, the non-edge partition is $\mu=421$, which is. illustrated by the x's in $\stair_6$. The row intervals include $[1,1]$, $[2,2]$, $[2,3]$, $[3,4]$, $[4,5]$. This makes cells $(1,1), (1,2), (1,3), (1,4)$ all intra-class cells because they have entries larger than those in their same row two. The tableau on the right is the tableau with the intra-class cells deleted and standardized.}  
\label{fig:intraEX}
\end{figure}
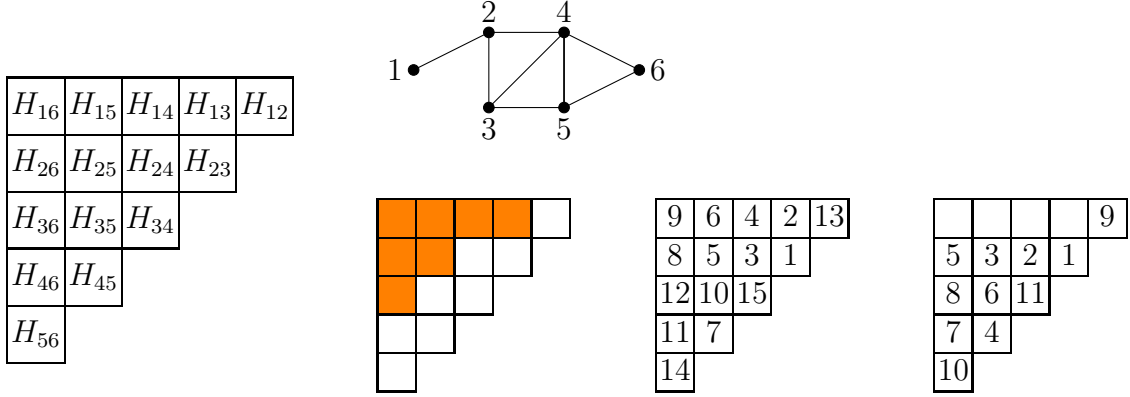


\section{The elements and chains of $\tp{\lolli m n}$ and pattern avoidance}
\label{sec:Llollipop}
Many of the results of later sections are for the lollipop graph case $G=\lolli m n$. In this section, we specialize the results of Section~\ref{sec:L_Gintro} to this case.   We begin with the specialization of Theorem~\ref{thm:intra_class_conditions} to this case. Throughout the section, let $N=m+n$, where $m,n$ are integers such that $m\ge 1$ and $n\ge 0$.

\begin{corollary}
\label{cor:intraclassLmn}
Let $H_{a,b}$ be a hyperplane in the hyperplane walk ${H}\in \cH(\longestk{N})$ and let $T$ be its corresponding balanced tableau. For a lollipop $L_{m,n}$ the hyperplane $H_{a,b}$ is an intra-class hyperplane if and only if any one of these equivalent conditions holds.
\begin{enumerate}
\item $H_{a,b}$ separates two permutations $v=u\circ s_i$ where the set $\{u(i+2),u(i+3),\ldots,u(N)\}$ contains $c$ such that $a<c<b$ and $c\ge m$.
\item There is an integer $j$, where $j\le n$ and $a<N-j\le b-1$, such that the hyperplane $H_{N-j,b}$ precedes $H_{a,b}$ in $H$.
\item There is an integer $j$, where $j\le n$ and $a<N-j\le b-1$, such that in the balanced tableau $T$, the entry in cell $(a,N-b+1)$ is larger than the entry in $(N-j,N-b+1)$.
\end{enumerate}
\label{thm:intra_class_conditions_lolli}
\end{corollary}
\begin{proof}
This is a consequence of Theorem~\ref{thm:intra_class_conditions}, using the following observation about edge intervals of the graph $\lolli m n$.

\begin{equation*}
    I_j=\begin{cases}\{N-j\}&\text{for $1\le j\le n$}\\ 
    [N-j]&\text{for $j>n$}\end{cases}
\end{equation*}

We have the conditions $j\le n$ and $a<N-j\le b-1$ in order to satisfy the requirements from Theorem~\ref{thm:intra_class_conditions} that $\{a,N-j+1\}\notin E$ and that $N-j+1\in[a+1,b]$.\omitt{For the lollipop graph $L_{m,n}$, the non-edge partition is $n^{m-1}(n-1)(n-2)\ldots 21$ where $n^{m-1}$ is $n$ repeated $m-1$ times. This means that the row difference intervals are $[1,j]$ for $j\in [m-1]$ and $[j,j]$ for $j\in [m,m+n-1]$. The first set of row difference intervals, $[1,j]$ for $j\in [m-1]$,  can not make any hyperplane an intra-class hyperplane. This means that the essential row difference intervals are $\{j\}$ for $j\in [m,m+n-1]$.} 
\end{proof}
\omitt{
\tvi{Subsection probably not needed. Can we delete it?}
\susanna{Added pattern avoidance to the section title, so that the definition is easy to find. 12/12/23}
\tvi{(Dec 18) Sounds good to me. I'm OK deleting these comments. }
} 
Let $G$ be a graph on $N$ vertices and let $w=w(1)\cdots w(N)\in\fS_N$ be a permutation. 
Recall from Definition~\ref{def:Gperm} that $w$  is a $G$-permutation if $w(i)$ is in the same connected component of $G|_{\{w(1),\ldots,w(i)\}}$ as $\max({\{w(1),\ldots,w(i)\}})$, for all $i\in[N]$. Additionally, recall that there is one $G$-permutation per $G$-equivalence class.  In the case $G=\lolli m n$, $\tp{\lolli m n}$ is a lattice quotient, so we may identify its elements with the subposet of $\fS_N$ whose elements are $\lolli m n$-permutations. 

The description of the elements of $\tp{\lolli m n}$ is particularly simple. 

\omitt{The following definition and lemmas give a convenient characterization of maximal chains in tubing lattices for the case $G=\tp{\lolli m n}$.}

\begin{definition}
\label{def:m312free}\label{def:m132free}
Let $m$ and $N$ be positive integers, with $m\leq N$. We say that $w\in\S_{N}$ is \mydef{$m$-312 avoiding} or \mydef{$m$-312 free} if there do not exist integers $1\le i<j<k\le m+n$ such that $w(j)<w(k)<w(i)$ and $w(k)\ge m$. A permutation $w\in\fS_N$ is said to be $m$-132 free if there do not exist $i,j,k$ such that $1\leq i<j<k\leq N$, $w(i)<w(k)<w(j)$, and $m\leq w(k)$.
\end{definition}

Pattern avoidance is well-studied. The usual 312-avoiding permutations are the case $m\le 2$\cite{BW}. See also \cite{ss}.

\begin{proposition}
\label{prop:m312}
Let $w$ be a permutation in $\fS_{m+n}$. The permutation $w$ is $m$-312 free if and only if $w$ is a $\lolli m n$-permutation. 
\end{proposition}
\begin{proof}
    Let $w$ be $m$-312 free and $j\in[m+n]$. Denote $\{w(1),\ldots,w(j)\}$ by $S$. We must show that $w(j)$ is in the same component of ${\lolli m n}|_S$ as $\max(S)$. Let $w(i)$, $i\le j$, be $\max(S)$. Suppose, for contradiction, that $w(i)$ and $w(j)$ are in different components.

    Since $w(i)$ and $w(j)$ are in different components, $i<j$ and $w(i)>w(j)$. If $w(i)$ and $w(j)$ were both less than or equal to $m$, they would be adjacent in $\lolli m n$, so we must have $w(i)> m$. Since $h$ is adjacent to $h+1$ in $\lolli m n$ for $h\in[m+n-1]$, we must also have $w(i)-w(j)\ge 2.$ In the case $w(j)<m$, any path in $\lolli m n$ containing the vertices $\{m,m+1,\ldots,w(i)-1\}$ would connect $w(i)$ and $w(j)$. Therefore, if $w(j)<m$, there must exist $k$, $k>j$, such that $w(k)\in\{m,m+1,\ldots,w(i)\}$. Similarly, in the case $w(j)\ge m$, any path containing $\{w(j)+1, w(j)+2,\ldots,w(i)-1\}$ would connect $w(i)$ and $w(j)$. Therefore, in this case $w(k)\in\{w(j)+1, w(i)+2, \ldots,w(i)-1\}$. But then we have $i<j<k$ such that $w(j)<w(k)<w(i)$ and $w(k)\ge m$.

    For the converse, suppose that $w$ is not $m$-312 free. That is, there exist $i<j<k$ such that $w(j)<w(k)<w(i)$ and $w(k)\ge m$. There may be more than one such a triple. Choose one such that for all triples with the same $j$ and $j$, $w(i)$ is maximum. We claim $w(i)$ and $w(j)$ are in different components of ${\lolli m n}|_{\{w(1),w(2),\ldots,w(j)\}}$, because $w(k)$ would be on any path between them. We also claim that $w(i)$ is $\max(\{w(1),\ldots,w(j)\})$ by our choice of triple. Therefore, $w$ is not a $\lolli m n$-permutation. 
\end{proof}

\begin{corollary}
\label{thm:lollipop vertices}
The vertices of $\tp{L_{m,n}}$ are in bijection with $m$-312 avoiding permutations in $\fS_{m+n}$.
\end{corollary}

\omitt{\susanna{ Will you take care of the red text? 3/20/24}
\tvi{Done. (4/4)}}
\begin{proposition}
\label{prop:max's and min's of L(G)}
Let $G=L_{m,n}$. Consider an equivalence class $[\![u]\!]$ for $u\in\fS_{m+n}$ in $\tp{G}$. There is exactly one $G$-permutation in $[\![u]\!]$, which is the minimal element and avoids $m$-312, and exactly one $m$-132 avoiding permutation in $[\![u]\!]$, which is the maximal element. 
\end{proposition}

\begin{proof}
Let $G=L_{m,n}$ and $u\in\fS_{m+n}$. By~\cite{BM} (Corollary 4.9) and
Proposition~\ref{prop:m312} we know that the unique $G$-permutation in
$[\![u]\!]$ is the minimum element. In fact, since $G$ is filled, by \cite[Theorem 4.16]{BM} and \cite{reading2004}, each equivalence class is an interval in the weak order. In particular, each class has a maximum element.

Suppose there exists a $m$-132 avoiding permutation in
$[\![u]\!]$. Call this $v$. We claim that $v$ is a maximum element in
$[\![u]\!]$. If our claim is true, then we know that there is at most
one $m$-132 avoiding permutation in each class since maximal elements
are unique in each equivalence class.  All we have to show now is
that any maximum element avoids $m$-132.

We now prove our claim that $v$ is a maximum element in $[\![u]\!]$.
We will show that by showing for any $s_i$ with $v<v\circ s_i$ that
the permutation $v\circ s_i\notin [\![u]\!]$ is not in the same
equivalence class.  By Lemma~\ref{lem:ab-cut} it will be sufficient to
show for any $i$ with $v(i)<v(i+1)$ that $\{v(i+2),v(i+3),\ldots,
v(m+n)\}$ is not a $v(i)v(i+1)$-cut set in $\lolli m n$. Suppose towards
a contradiction that $\{v(i+2),v(i+3),\ldots, v(m+n)\}$ is a
$v(i)v(i+1)$-cut set in $\lolli m n$. Considering the structure of
$\lolli m n$, the set $\{v(i+2),v(i+3),\ldots, v(m+n)\}$ must contain a
cut vertex that would put $v(i)$ and $v(i+1)$ in separate connected
components upon its removal from $\lolli m n$.  This means that for some
$j\geq 2$ that $v(i+j)$ is greater than $v(i)$, at least $m$, and less
than $v(i+1)$. Taken together this means that $v(i)v(i+1)v(i+j)$ is a
$m$-132 pattern,  contradicting our assumption that $v$ was $m$-132
avoiding.

Now we show that a maximum element $v\in [\![u]\!]$ avoids
$m$-132. Suppose instead that $v$ contains a $m$-132 pattern
$v(i)v(j)v(k)$. Because $v(i)<v(k)$ and $v(j)>v(k)$ there must exist a
$a\in[i,j-1]$ such that $v(a)<v(a+1)$, $v(a)<v(k)$ and
$v(a+1)>v(k)$. However since $v(k)\geq m$, this implies that
$\{v(k)\}$ is a $v(a)v(a+1)$-cut set. Further since $k>a+1$ we can
conclude that $\{v(a+2),v(a+3),\ldots, v(m+n)\}$ is a $v(a)v(a+1)$-cut
set contradicting our assumption that $v$ was maximal.  \end{proof}


\omitt{
\susanna{Should Theorem~\ref{thm:lollipop chains} be maximal chains or maximum length chains? 3/20/24}
\tvi{I was hoping Theorem~\ref{thm:lollipop chains} would be about all Maximal chains of any length. (4/4)}
} 
We summarize our results on maximal chains in $\tp{\lolli m n}.$
\begin{theorem}
\label{thm:lollipop chains}
The following are in bijection and are in bijection with maximal chains in $\tp{L_{m,n}}$
\begin{enumerate}
\item The subsequence of all $m$-312 avoiding permutations within a maximal chain in the weak order $\fS_{m+n}$. 
\item The subsequence of all $m$-132 avoiding permutations within a maximal chain in the weak order $\fS_{m+n}$. 
\item The collection of hyperplane walks projected onto $\tp{L_{m,n}}$ with intra-hyperplanes deleted, $\MH{L_{m,n}}$.
\item The $L_{m,n}$-balanced tableaux, $\MB{L_{m,n}}$.
\end{enumerate}
\end{theorem}
\omitt{
\tvi{(Dec 18) This Theorem~\ref{thm:lollipop chains} could use some rephrasing. The next comment is one idea on how we can phrase it more  by setting up some more definitions. }
\tvi{(Dec 18) This might be a good place to use the notations for $\MB{L_{m,n}}$ and $\MH{L_{m,n}}$ and define $\MB{G}$ and $\MH{G}$ for general connected filled graphs in the section of general filled graphs. For definitions I've been using $\MB{G}$ for the collection of objects we get from balanced tableaux after deleting the cells associated to intra-class hyperplanes and then standardizing (i.e. replace the $i$th smallest number with $i$). More easily $\MH{G}$ is the collection of hyperplane sequences we get from hyperplane walks by deleting the intra-class hyperplanes. }
\susanna{I think you added Theorem~\ref{thm:lollipop chains} in October and you may rephrase it however is best for your use of it. Can you also check to make sure all the conditions are correct? 1/2/24 }
\tvi{Sam to do: Theorem 4.5 above, rephrase with notation as needed. Also, check.}
\tvi{Rephrasing done. Susanna you are welcome to adjust as needed. (3-19-24)}
\susanna{There is no proof here, and no definition of the sets in 3 and 4. The definition of $m$-132 is later, do you want me to move it up. This seems to still need a fair amount of work. 3/23/24}
\tvi{Sam To Do: add definitions of hyperplane and balanced tableaux above (3-26)}
\tvi{Done (4-4)}
}
\section{Maximum length chains in $\tp{\lolli m n}$}
\label{sec:mlcinlollis}


In this section, the focus is on the case $G=\lolli m n$, where $m$ and $n$ are integers with $m\ge 1$ and $n\ge 0$. We begin with a characterization of the elements in $\tp{\lolli m n}$ which are maximum length chains. \omitt{The rest of the section is devoted to describing the tableaux and reduced expressions which correspond maximum length chains in $\tp{\lolli m n}$.} We then study the tableaux in this case. 
Every maximum length chain in $\tp{\lolli m n}$ is also a maximal
chain in the weak order poset and as such corresponds to a standard
Young tableau of staircase shape, as well as to a balanced tableau of
the same shape. We describe those tableaux which arise in the $\tp{\lolli m n}$ case. The descriptions of the permutations which are possible in a maximum length chain is also simple.

Recall that we may use view the tubing lattice $\tp{\lolli m n}$ as the sublattice of $\fS_{m+n}$ induced by the $\lolli m n$-permutations. By Proposition~\ref{prop:m312}, the elements of $\tp{\lolli m n}$ satisfy a pattern avoidance condition. The elements in the maximum length chains satisfy an additional avoidance condition. 
\omitt{
\susanna{Once Samantha has shown that the max in class is $m$-132 avoiding, a lot of this intro to Section~\ref{sec:mlcinlollis} can be shortened. (4/1/24)}
\begin{definition}
\label{def:m132free}
    A permutation $w\in\fS_N$ is said to be $m$-132 free if there do not exist $i,j,k$ such that $1\leq i<j<k\leq N$, $w(i)<w(k)<w(j)$, and $m\leq w(k)$.
\end{definition}
}

\begin{lemma}\label{lem:m132}
Let $C=\{\ii=\wch{0}\lessdot\wch{1}\lessdot\cdots\lessdot\wch{\ell}=\longestk{m+n}\}$ be a maximal chain in $\tp{\lolli m n}$. 
\begin{enumerate}
\item If the chain has length $\ell=\binom{m+n}{2}$, then permutation $\wch{h}$, for $1\le h\le m+n$, is $m$-132 free.
\item If every permutation in the chain is $m$-132 free, then $k=\binom{m+n}{2}$.
\end{enumerate}
\end{lemma}
\begin{proof}
We use the notation for the chain given in the statement.
\begin{enumerate}    
\item Suppose, for contradiction, $C$ has length $\binom{m+n}{2}$ and yet $\wch{h}$ is not $m$-132 free for some $h$. There is an $i,j,k$ where $1\le i<j<k\le m+n$, $\wch{h}(i)<\wch{h}(k)<\wch{h}(j)$, and $m\le \wch{h}(k)$.
Since the final element of the chain $C$ is $\longestk{n+m}$, it is a maximum length chain in $\fS_{m+n}$, and $\wch{h}(k)$ and $\wch{h}(j)$ maintain their relative positions in $\wch{\ell}$ for $\ell\ge h$, there must exist $h'>h$ where the relative order in $\wch{h'}$ is $\ldots\wch{h}(j)\ldots\wch{h}(i)\ldots\wch{h}(k)$, which means $\wch{h'}$ contains an $m$-312 pattern. That is, if $\wch{h}$ contains a $m$-132 pattern, since $\wch{\binom{m+n}{2}}$ contains a $m$-321 pattern, then there must be an $h'$ such that $\wch{h'}$ contains a $m$-312 pattern. By Proposition~\ref{prop:m312}, $\tp{\lolli m n}$ has only permutations which avoid $m$-312, our contradiction.
\item 
Since $\wch{h}$ is both $m$-132 and $m$-312 free for $0\le h\le k$, it is alone in its equivalence class by Proposition~\ref{prop:max's and min's of L(G)}. Therefore, the projection to $\tp{\lolli m n}$ preserves the length of the chain.\omitt{Suppose, again for contradiction, there is an $h$ such that $\wch{h+1}$ covers $\wch{h}$ in $\tp{\lolli m n}$ but not in $\fS_{m+n}$. There is a permutation $w\notin\tp{\lolli m n}$ such that $\wch{h}\lessdot w<\wch{h+1}$ in $\fS_{m+n}$. Since $w\notin\tp{\lolli m n}$, there are $i,j,k$ such that $1\le i<j<k\le m+n$ such that $w(j)<w(k)<w(i)$ and $w(k)\ge m$. The permutations $\wch{h}$ and $w$ differ by an adjacent transposition and since $\wch{h}$ is $m$-312 free, we must have $j=i+1$ and the transposition is in positions $i$ and $i+1$. But this would mean that $\wch{h}$ is not $m$-132 free, a contradiction.} \end{enumerate}
\end{proof}

\subsection{Commuting shuffles}
We continue our study of maximum length chains by examining their corresponding reduced words.
The set of maximum length chains we study turns out to be a union of
commutation classes of reduced words. See \cite[Section 1.2]{Stembridge1997}, for example, for the definition of commutation classes. However, we also need a more concrete
description: sets of \mydef{commuting shuffles}.  Let
$\sigma=\sigma_1\cdots\sigma_k$ and $\tau=\tau_1\cdots\tau_j$ be two
reduced words such that the concatenation of $\sigma$ and $\tau$,
$\sigma\tau$, is reduced. We define the set $\cshuf\sigma \tau$ of
commuting shuffles inductively.  The concatenation $\sigma\tau$
is in $\cshuf\sigma \tau$ with the index set of $\sigma$ being
$\{1<2<\cdots<k\}$ and the index set of $\tau$ being
$\{k+1<\cdots<k+j\}$. Suppose
$\rho=\rho_1\ldots\rho_{k+j}\in\cshuf\sigma\tau$ with $\sigma$-index
set $I$ and $\tau$-index set $J$. The reduced expression
$\rho'=\rho_1\ldots\rho_{i+1}\rho_i\ldots\rho_{k+j}$, which is $\rho$
with its $i^{\text{th}}$ and $(i+1)^{\text{st}}$ entries swapped, is a
commuting shuffle of $\sigma$ and $\tau$ if $i\in I$, $i+1\in J$, and
$\rho_i$ and $\rho_{i+1}$ commute. The $\sigma$-index set of $\rho'$
is $(I\cup\{i+1\})-\{i\}$ and the new $\tau$-index set is
$(J\cup\{i\})-\{i+1\}$. Informally, $\cshuf\sigma\tau$ is the set of
all shuffles of $\sigma$ and $\tau$ which are obtained by commuting
elements of $\tau$ with elements of $\sigma$.

We will show that the reduced words coming from maximum length chains in $\tp{\lolli m n}$ are commuting shuffles of two types, described below, of reduced words.  Let $v_{m,n}$ be the permutation in $\S_{m+n}$ written
\begin{equation}\label{eq:vmn}\vmn=[\overbrace{m,m+1,\ldots,m+n}^{n+1},\overbrace{m-1,m-2,\ldots,2,1}^{m-1}]\end{equation} in
one-line notation.  Let $T$ be a standard Young tableau $T$ shape $(\lambda_1,\ldots,\lambda_k)$ with entry $T_{i,j}$ in the cell in row $i$, column $j$. We call $T$ \mydef{shiftable} if for all $i$ and $j$ such that $1\leq i\le k-1$ and  $1\leq j\leq\lambda_i-1$, either $T_{i,j}<T_{i+1,j-1}$ or $\lambda_{i+1}<j-1$. See Definition~\ref{def:m-row-shift} for a more general concept; here, shiftable means that $T$ is the result of left-justifying a shifted tableau. Let $\shredw{\longestk{n+1}}$ denote the
reduced words of $\longestk{n+1}$ whose $Q$ tableau under Edelman-Greene insertion is a shiftable tableau.

  We define $\shuf m n$ to be the set of 
    commuting shuffles of $\sigma\in\redw {\vmn}$ and $\tau\in\shredw{\longestk{n+1}}$. In
    other words,
$$\shuf m n=\bigcup_{\substack{\sigma\in\redw {v_{m,n}}\\\tau\in\shredw{\longestk{n+1}}}}\cshuf \sigma{\tau}.$$

\begin{example}
  Let $m=3$ and $n=1$. The permutation $v_{3,1}$ is $[3,4,2,1]$ and the
  only element of $\shredw{\longestk{1+1}}$ is $\tau=1$. There are
  five reduced words for $v_{3,1}$. The reduced words and
  their commuting shuffles with $\tau$ are
  \begin{itemize}
  \item $\sigma=12312$, $\cshuf\sigma\tau=\{123121\}$;
  \item $\sigma=12132$, $\cshuf\sigma\tau=\{121321\}$;
  \item $\sigma=21232$, $\cshuf\sigma\tau=\{212321\}$;
  \item $\sigma=21323$, $\cshuf\sigma\tau=\{213231,213213\}$; and
     \item $\sigma=23123$, $\cshuf\sigma\tau=\{231231,231213\}$.   
    \end{itemize}
There are seven commuting shuffles in $\shuf 3 1$. 
  
\end{example}

We establish some facts about $\vmn$
and its reduced wordss.
\begin{claim} Fix integers $m$ and $n$, with $m\ge 1$ and $n\ge 0$.\label{claim:vmn}
  \begin{enumerate}
  \item For $\sigma\in\redw{\vmn}$ and $\tau\in\shredw{\longestk{n+1}}$, the expression $\sigma\tau$ is reduced.
  \item \label{claim:vex} The permutation $\vmn$ has the same number of reduced words as there are standard Young tableaux of shape $(m+n-1,m+n-2,\ldots,n+2,n+1)$.
  \item\label{claim:decseq} Fix $\sigma\in\redw{\vmn}$. For every $k$ with
    $m-1\le k\le n+m-1$, there exist integer sequences
    $d^k=(d^k_1<d^k_2<\cdots<d^k_{m-1})$ such that
    $\sigma_{d^k_i}=k-i+1$ and $d^{k-1}_i<d^k_i$. What's more, if $k_1\ne k_2$, then 
    $d^{k_1}$ and $d^{k_2}$ are disjoint as sets. In other words, $\sigma$ has $n+1$ disjoint decreasing sequences of length $m-1$, starting with $m-1,m,m+1,\ldots,m+n-1$. Additionally, there
    exists a sequence of indices $a=(a_1<a_2<\cdots< a_{m+n-1})$ such
    that $\sigma_{a_i}=i$. 
\omitt{    \item The length of the longest decreasing subsequence of $\vmn$ is $m-1$.}
    \end{enumerate}
  \end{claim}
\begin{example}
  We illustrate Claim~\ref{claim:vmn} Part~\ref{claim:decseq} in this example. Let $m=4$
  and $n=2$, so that $\vmn=v_{4,2}=[4,5,6,3,2,1]$. Let
  $\sigma=132145324354\in\redw{v_{4,2}}$. The decreasing sequence indices
  are $d^5=(6,9,10)$, $d^4=(5,7,8)$ and $d^3=(2,3,4)$. The ascending
  sequence index $a$ is $(1,3,7,9,11)$.
  \end{example}

\begin{proof}[Proof of Claim~\ref{claim:vmn}]
  Fix positive integers $m$ and $n$. 
  \begin{enumerate}
    \item We need to explain why $\sigma\tau$ is a reduced expression for any
$\sigma\in\redw{\vmn}$. The length of $\sigma\in\redw{\vmn}$ is
$(n+1)(m-1)+\binom{m-1}{2}$ and the length of $\tau$ is
$\binom{n+1}{2}$, so there are $\binom{m+n}{2}$ elements in
the list $\sigma\tau$. The result of applying $\tau$ to $\vmn$ is the
longest element of $\S_{m+n}$, so $(v_{nm})\tau$ also has
$\binom{n+m}{2}$ elements, which shows $\sigma\tau$ is reduced, and is an element of $\redw {\longestk{m+n}}.$
\item 
In order to use Corollary~4.2 of \cite{stan84}, which would prove this part of the claim, we must show that $\lambda(\vmn)=\mu(\vmn)$, where for a permutation $w=[w_1,\ldots,w_k]\in\S_k$, $\lambda(w)$ is the partition given by the non-increasing rearrangement of $(r_1(w),r_2(w),\ldots,r_k(w))$, $\mu(w)$ is the partition given by the conjugate of the non-increasing rearrangement of $(s_1(w),s_2(w),\ldots,s_k(w))$, and

$$r_i(w)=|\{j:j<i\text{ and }w_j>w_i\}|\text{ and }s_i(w)=|\{j:j>i\text{ and }w_j<w_i\}|.$$ By \cite[Exercise 7.22]{S99}, for example, this is the same as showing $\vmn$ is 4132-avoiding, which is true by definition of $\vmn$. 
\omitt{Since $$r_i(\vmn)=\begin{cases}0&\text{for $1\le i\le n+1$}\\i-1&\text{for $n+2\le i\le n+m$}\end{cases}\text{ and }s_i(\vmn)=\begin{cases}m-1&\text{for $1\le i\le n+1$}\\m+n-i-2&\text{for $n+2\le i\le n+m$}\end{cases},$$ we have $\lambda(w_{mn})=\mu(\vmn)=(m+n-1,m+n-2,\ldots,n+1)$ and this part of the claim is proved.} 
\item Consider $\sigma_1\sigma_2\cdots\sigma_{\ell(\vmn)}\in\redw{\vmn}$ acting on
$[1,2,\ldots, m+n]\in\fS_{m+n}$, where $\sigma_1$ is applied first and the
transposition $\sigma_k=s_{i_k}$ swaps the values in positions $i_k$ and
$i_k+1$. Fix $j$, $1\le j\le n+1$. Since $\vmn(j)=j+m-1$, there are
indices $i_1<i_2<\cdots<i_{m-1}$ such that $\sigma_{i_1}$ swaps the
value $j+m-1$ from position $j+m-1$ to position $j+m-2$ ($i_1=j+m-2$),
$\sigma_{i_2}$ swaps the value $j+m-1$ from position $j+m-2$ to
position $j+m-3$, and so on, until $\sigma_{i_{m-2}}$ swaps the value
$j+m-2$ from position $j+m-(m-1)$ to position $j$. We define
$d^{m-1+j-1}_h$ to be $i_h$, so that $\sigma_{d^k_h}=k-h+1$. Since
$\sigma_{d^k_h}$ moves the value $k+1$ from position $k+1-h$ to
position $k-h$, the sequences $d^{k_1}$ and $d^{k_2}$ are disjoint if
$k_1\ne k_2$.  We have $\sigma_{d^k_i}=k-i+1$ and $d^{k-1}_i<d^k_i$, because $j-1+m-1$ and $j+m-1$ never change their relative positions.

The reasoning to explain that there exists a sequence of indices
$a=(a_1<a_2<\cdots <a_{m+n-1})$ such that $\sigma_{a_i}=i$ is
similar, but uses the fact that $\vmn(m+n)=1$.\omitt{\item A decreasing subsequence in $\sigma$ corresponds to a pair $(i,\vmn(i))$ such that $\vmn(i)>i$, with the length of the subsequence being $\vmn(i)-i$. The definition of $\vmn$ is $$\vmn(i)=\begin{cases}i+m-1&\text{if $1\le i\le m+1$}\\m+n-i+1&\text{if $n+2\le i\le m+n$}.\end{cases}$$ Based on the definition, we have $v(i)-i\le m-1$ for all $i$.}
  \end{enumerate}\end{proof}

We will need some properties of elements of $\shredw{\longest^{(n+1)}}$. A \mydef{lattice word}
is a sequence composed of positive integers, in which every prefix contains at least as many positive integers $i$ as integers $i + 1$. A \mydef{reverse lattice word}, or \mydef{Yamanouchi word}, is a string whose reversal is a lattice word \cite{Fulton}.

\begin{proposition}[\cite{FN}]
  \label{prop:FN}
  Let $\rho$ be a reduced expression for $\longest^{(n+1)}$ and $C$ the corresponding maximal (and necessarily maximum length) chain in the weak order on $\S_{n+1}$. 
  \begin{enumerate}
  \item The reduced expression $\rho$ also corresponds to a maximum length chain in the Tamari lattice $\Tam_{n+1}$ if and only if $\rho$ is a lattice word and reverse lattice word with $n-i+1$ occurrences of $i$.
   \item The reduced expression $\rho$ also corresponds to a maximum length chain in the Tamari lattice $\Tam_{n+1}$ if and only if $\EG(\rho)$ is a shiftable tableau of shape $(n,n-1,\ldots,1)$.  
  \end{enumerate}
  \end{proposition}

\omitt{
We say a chain in the symmetric group $\S_{m+n}$ is $m$-132 free if all permutations in it avoid $m$-132 and we say a reduced expression for $\longestk{m+n}$ is $m$132 free if the chain it corresponds to is.
}

\omitt{\susanna{Add something at the start about the correspondence between chains and reduced words and between reduced words and tabs}
}
\omitt{
\susanna{Say something about the existence of chains of length $\binom{m+n}{2}$}
}
\begin{lemma}\label{lem:commute}
Suppose $C$ is a maximum length chain in $\tp{\lolli m n}$ and $D$ is
related to $C$ by commutation relations only. Then $D$ is also a
maximum length chain.
\end{lemma}

\begin{proof} We need only show that a single commutation relation cannot create a $m$-132 pattern in a maximum length chain in $\tp{\lolli m n}$. Suppose $C$ is a maximum length chain and its corresponding reduced expression is $\rho_1\rho_2\cdots\rho_k\rho_{k+1}\cdots\rho_{\binom{m+n}{2}}$. Let $D$ correspond to $\rho_1\rho_2\cdots\rho_{k+1}\rho_{k}\cdots\rho_{\binom{m+n}{2}}$, where $|\rho_k-\rho_{k+1}|>1$. In terms of permutations, write $$C=\{\wch{0}\lessdot\wch{1}\lessdot\cdots\lessdot\wch{\binom{m+n}{2}}=\longestk{m+n}\},$$where $\wch{j+1}=\wch{j}\rho_{j+1}$. The chain $D$ is then $D=\{\wch{0}\lessdot\wch{1}\lessdot\cdots\lessdot\wch{k-1}\lessdot\vch{k}\lessdot\wch{k+1}\cdots\lessdot\wch{\binom{m+n}{2}}\}$, where $\vch{k}=\wch{k-1}\rho_{k+1}$ and $\wch{k+1}=\vch{k}\rho_k$. See Figure~\ref{fig:twochains}. For contradiction, suppose $D$ is not $m$-312 free. If it is not, then it would have to be $\vch{k}$ which contains the $m$-312 pattern.  There are $1\le a<b<c\leq m+n$ such that $\vch{k}(b)<\vch{k}(c)<\vch{k}(a)$ and $\vch{k}(c)\ge m $. In the one-line notation for $\vch{k}$, we have $[\ldots\vch{k}(a)\ldots\vch{k}(b)\ldots\vch{k}(c)\ldots]$. Since $\wch{k-1}$ is $m$-312 free, the adjacent transposition $\rho_{k+1}$ must transpose $\vch{k}(b)$ and $\vch{k}(a)$, so that $\wch{k-1}$'s one-line form is $[\ldots\vch{k}(b)\ldots\vch{k}(a)\ldots\vch{k}(c)\ldots]$, $\rho_{k+1}=a$ and $b=a+1$. Similarly, since $\wch{k+1}$ is also $m$-312 free, $\rho_k$ must transpose $\vch{k}(b)$ and $\vch{k}(c)$, so that $\rho_k=b$ and $c=b+1$. But this means that $\rho_{k+1}=\rho_k-1$, contradicting $|\rho_{k+1}-\rho{k}|>1$.
\end{proof}

\begin{figure}
\centering
\begin{tikzpicture}
\def\h{2.1}
\def\v{.7071067811}
\node (w0) at (0,0) {$\ii$};
\node (w1) at (\h,0) {$\wch{1}$};
\node (w2) at (2*\h,0) {$\wch{2}$};
\node (wk-1) at (3*\h,0) {$\wch{k-1}$};
\node (wk) at (3*\h+\h*\v,\h*\v) {$\wch{k}$};
\node (wkp) at (3*\h+\h*\v,-\h*\v) {$\vch{k}$};
\node (wk+1) at (3*\h+2*\h*\v,0) {$\wch{k+1}$};
\node (wk+2) at (4*\h+2*\h*\v,0) {$\wch{k+2}$};
\node (end-1) at (5*\h+2*\h*\v,0) {$\wch{\binom{m+n}{2}-1}$};
\node (end) at (6*\h+2*\h*\v,0) {$\longestk{m+n}$};
\draw[blue,thick] (w0)--(w1) node[above,red,midway]{$\rho_1$};
\draw[blue,thick] (w1)--(w2) node[above,red,midway]{$\rho_2$};
\draw[blue,thick,dashed] (w2)--(wk-1);
\draw[blue,thick] (wk-1)--(wk) node[above,red,midway,left]{$\rho_k$};
\draw[blue,thick] (wk-1)--(wkp) node[below,red,midway,left]{$\rho_{k+1}$};
\draw[blue,thick] (wkp)--(wk+1) node[below,red,midway,right]{$\rho_{k}$};
\draw[blue,thick] (wk)--(wk+1) node[above,red,midway,right]{$\rho_{k+1}$};
\draw[blue,thick] (wk+1)--(wk+2) node[above,red,midway]{$\rho_{k+2}$};
\draw[blue,thick,dashed] (wk+2)--(end-1);
\draw[blue,thick] (end-1)--(end) node[above,red,midway]{$\rho_{\binom{m+n}{2}}$};
\end{tikzpicture}
\caption{Two chains in $\S_{m+n}$ that differ by a commutation move.}
\label{fig:twochains}
\end{figure}
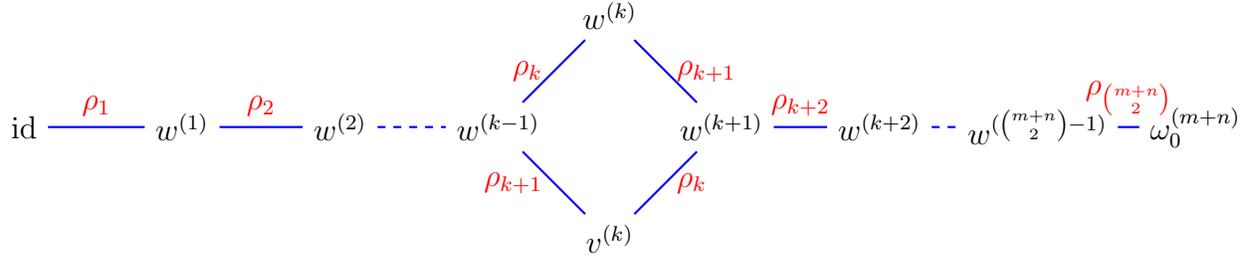

\begin{proposition}
  \label{prop:MLCsAndCSs} Let $m$ and $n$ be integers with $m\ge1$ and $n\ge0$ and let
  $\rho$ be a reduced expression for $\longest^{(m+n)}$. Then $\rho$
  represents a maximum length chain in $\tp{\lolli mn}$ if and only if
  $\rho\in\shuf mn$.  \end{proposition}
\begin{proof}
\omitt{  Use $M$ as shorthand for $\binom{n+m}{2}$; $A$ for $(n+1)(m-1)+\binom{m-1}{2}$, the length of $\vmn$; and $B$ for $\binom{n+1}{2}$, the length of $\longest^{(n+1)}$.}

We use $A$ as shorthand for $(n+1)(m-1)+\binom{m-1}{2}$, the length of
$\vmn$ and $B$ for $\binom{m+1}{2}$.

Let $\rho\in\redw{\longestk{m+n}}$ and let $C=\{\uch{0}\lessdot\uch{1}\lessdot\cdots\lessdot\uch{\binom{(n+m)}{2}}=\longestk{m+n}\}$ its corresponding maximal chain in $\fS_{m+n}$, and assume
$C$ and $\rho$ represent a maximum length chain in $\tp{\lolli mn}$. The chain $C$ is $m$-132 and $m$-312 free by Proposition~\ref{prop:m312} and Lemma~\ref{lem:m132}. We must produce
$\sigma\in\redw{\vmn}$ and $\tau\in\shredw{\longest^{(n+1)}}$ such
that $\rho\in\cshuf\sigma\tau$. We define $\tau$ inductively, through
defining its index set $J=\{j_1<j_2<\cdots <j_B\}$ in $\rho$. Let $j_1$
be the smallest $j$ such that there exists $i$ where $m\leq
u^{(j)}(i+1)<u^{(j)}(i)$. Set $\tau_1$ equal to
$\rho_{j_1}=i$. That is, $j_1$ is the first index $j$ where $\rho_j$ causes an
inversion between two numbers which are at least $m$. Assume
$j_1<\cdots<j_k$ have been defined. Let $j_{k+1}$ be the smallest
$j>j_k$ such that there exists $i$ where $m\leq
u^{(j)}(i+1)<u^{(j)}(i)$ and $m\leq u^{(j-1)}(i)<u^{(j-1)}(i+1)$.
Again, we set $\tau_{k+1}=\rho_{j_{k+1}}=i$. There are $n+1$ numbers
$m,m+1,\ldots,m+n$ whose swaps $\tau$ records, so $\tau=\tau_1\tau_2\cdots\tau_{\binom{n+1}{2}}$.

We claim $\tau_j\le n$ for $j=1,\ldots,B$. If not, there is an $i>n$
such that $\uch{j}(i)>\uch{j}(i+1)\geq m$. Since only $n+1$ of the elements of $[m+n]$ are at least $m$, there must exist
$h\le n<i$ such that $\uch{j}(h)<m$. The permutation $\uch{j}$ thus contains an $m$-132 pattern, contradicting the condition to be maximum length.

We must show that $\tau$ is reduced. Let $\vch{k}$ be obtained from $\uch{j_k}$ by restricting $\uch{j_k}$, in one-line notation, to $\{m,\ldots,m+n\}$, then replacing $m+h$ by $h+1$ for $1\le h\le n$ and set $\vch{0}=\ii_{n+1}$. In fact, if we define $p_k$ to be the largest integer $i$ such that $s_i$ appears in $\tau_1\tau_2\cdots\tau_k$, then 

\begin{equation}\label{eq:same_start}\vch{k}(h)=\uch{j_k}(h)-m+1\text{, for }1\le h\le p_k+1.\end{equation} Note that \eqref{eq:same_start} implies 
\begin{equation}
\label{eq:v rec}
    \vch{k}=\vch{k-1}\tau_k.
\end{equation}
We justify \eqref{eq:same_start} by induction on $k$. Since $\uch{i}$ is $m$-132 avoiding, $\tau_1=s_1$ and the claim holds. Let $k>1$. Assume that $\vch{k-1}(h)=\uch{j_{k-1}}(h)-m+1$ for $1\le h\le p_{k-1}+1$. It must then be true that $\vch{k-1}(h)=\uch{j}(h)-m+1$ for $1\le h\le p_{k-1}$ and $j_{k-1}< j<j_k$, because for $j, h$ in these ranges, $\uch{j}(h)\ge m$ and $\rho_j$ does not invert $x$ and $y$ where both $x$ and $y$ are greater than $m$. Suppose $\tau_k=s_r$. If $r\le p_{k-1}$, then \eqref{eq:same_start} is true for $\vch{k}$ by the induction hypothesis. Now consider $p_k=r>p_{k-1}$ and let $x=\uch{j_k-1}(r)=\vch{k-1}(r)$ and $y=\uch{j_k-1}(r+1)$, where $x<y$. Since $\uch{j_k}$ is $m$-132 avoiding, $\uch{j_k}(r-1)>x\ge m$, which means there is an inversion in positions $r-1$ and $r$ of $\uch{j_{k-1}}$. In particular, $p_{k-1}\ge r-1$, so that $p_k=p_{k-1}+1$. By the induction hypothesis, $\vch{k}(h)=\uch{j_k}(h)-m+1\text{, for }1\le h\le p_k-1$ so we are only left with showing that $\vch{k}(r)$ is equal to $\uch{j_k}(r)-m+1$ and that $\vch{k}(r+1)$ is equal to $\uch{j_k}(r+1)-m+1$. But this holds by definition of $\vch{k}$, since $\uch{j_k}(r)\geq m$, as is $\uch{j_k}(r+1)$. \omitt{and $\vch{k}(r+1)=\vch{k-1}(r)=\uch{j_k}(r+1)-m+1$, so we are only left with showing that $\vch{k}(r)$ is equal to $\uch{j_k}(r)-m+1$. Note that $\vch{k}(r)=\vch{k-1}(r+1)=r+1$ and $\uch{j_k}(r)=\uch{j_k-1}(r+1)$. Since $\tau_{i}\in\{s_1,s_2,\ldots,s_{r-1}\}$ for $1\le i\le k-1$, $\uch{j_k-1}(r+1)>\uch{j_k-1}(h)\ge m$ for $1\le h\le r$, so that $\uch{j_k-1}(r+1)\ge m+r$. We must have $\uch{j_k-1}(r+1)=m+r$, because otherwise there would be an $h$, $h>r+1$, such that $\uch{j_k-1}(r+1)>\uch{j_k-1}(h)\geq m$, but this inversion could not have taken place by the restrictions on $\tau$.} Thus \eqref{eq:same_start} and therefore \eqref{eq:v rec} holds.

To complete the proof that $\tau$ is reduced, assume for contradiction that it is not. Let $k$ be the smallest such that $\tau_1\tau_2\cdots\tau_{k-1}=\vch{k-1}$ is reduced and $\tau_1\tau_2\cdots\tau_{k-1}\tau_{k}=\vch{k}$ is not, and suppose $\tau_k=s_r$. Then $\vch{k-1}(r)>\vch{k-1}(r+1)$ and $\vch{k}(r)<\vch{k}(r+1)$-- in other words, $(\vch{k-1}(r),\vch{k-1}(r+1))$ is an inversion for $\vch{k-1}$ and not one for $\vch{k}$. By \eqref{eq:same_start}, $(\vch{k-1}(r)+m-1,\vch{k-1}(r+1)+m-1)$ is thus an inversion for $\uch{j_{k-1}}$ and not for $\uch{j_k}$, contradicting our assumption that $C$ is a maximal chain. The permutation $\tau$ is reduced.


\omitt{
Since the final element of the chain $C$ is $\longestk{n+m}$ and $\uch{j}(i)$ and $\uch{j}(i+1)$ maintain their relative positions in $\uch{k}$ for $k\ge j$, there must exist $j'>j$ where the relative order is $\ldots\uch{j}(i)\ldots\uch{j}(h)\ldots\uch{j}(i+1)$, which means $\uch{j'}$ contains an $m$-312 pattern. Thus $\tau_j\le n$ for $1\le j\le B.$}

  We may view $\tau$ as a an element of $\redw{\longest^{(n+1)}}$,
  since when restricted to $\{m,m+1,\ldots,n+m\}$, it transforms
  $[m,m+1,\ldots,m+n]$ to $[m+n,m+n-1,\ldots,m]$. What's more, the
  permutations $u^{(j_1)},u^{(j_2)},\ldots,u^{(j_B)}$ are all
  312-avoiding on the values $\{m,m+1,\ldots,n+m\}$-- this can be seen
  by induction and the fact that $C$ is $m$-312 avoiding. Bj\"orner and Wachs \cite{BW}
  showed that the induced subposet of the weak order on permutations
  which are 312-avoiding is the Tamari, thus $\tau$ is a reduced
  expression which which corresponds to a maximal length chain in the
  Tamari lattice for $n+1$. By Proposition~\ref{prop:FN}, we have that
  $\tau\in\shredw{\longest^{(n+1)}}$.

Let $\sigma$ be $\rho$ with $\rho_{j_1},\rho_{j_2}, \dots,\rho_{j_B}$
removed. We can show that $\sigma$ is reduced by a proof similar to the proof that $\tau$ is reduced. \omitt{All adjacent transpositions which affected the relative order
of $m,m+1,\ldots,m+n$ have been removed and
$\ell(\rho)-\ell(\tau)=\ell(\sigma)$, so $\sigma\in\redw{\vmn}$.}

Finally, we must explain why $\rho\in\cshuf \sigma \tau$. Since $\tau$
is a lattice word, $\tau_1=1$, and both $\uch{j_1}(1)$ and $\uch{j_1}(2)$ are at least $m$. For $k>j_1$, $\uch{k}(1)>\uch{k}(2)\geq m$, so if $k>j_1$ and $\rho_k=1$ or $\rho_k=2$, then $k\in J$. \omitt{Any inversion in positions $1$, $2$, or $3$ would therefore have  If $\rho_k$ were equal to $1$ or $2$ for $k>j_1$ and $k\not\in J$, then either $u^{(k)}(1)$ or $u^{(k)}(2)$ would be strictly less than $m$ and we would have a violation of $m$-132 free.} Now suppose $\tau_f=i$, $f>1$. By \eqref{eq:same_start}, $\uch{k}(h)\geq m$ for $1\leq h\leq i+1$ and $k\geq j_f$. Therefore, if $\rho_k$ is $i-1$ or $i$ and $k>j_1$, then $k\in J$. If $k>j_1$ and $\rho_k=i+1$, then $\uch{k-1}(i+2)>\uch{k-1}(i+1)\geq m$, so again $k\in J$. \omitt{If $\rho_k$ were equal to $i-1, i, i+1$ for $k>j_f$ and $k\not\in J$,
then at least one of $u^{(k)}(i-1), u^{(k)}(i), u^{(k)}(i+1)$ would be
strictly less than $m$ and we would have a permutation which is not $m$-132 free. }

For the converse, if $\rho\in\shuf m n$, then there is $\sigma\in\redw{\vmn}$ and
$\tau\in\shredw{\longestk{n+1}}$ such that
$\rho\in\shuf{\sigma}{\tau}$. Then $\rho$ is related to $\sigma\tau$
by commutation relations, so by Lemma~\ref{lem:commute}, it is enough
to show that $\sigma\tau$ represents a maximal length chain in
$\tp{\lolli m n}$.

Let $C$ be the chain corresponding to $\sigma\tau$ and write
\begin{equation}\label{eq:chain}C=\{\ii=\uch{0}\lessdot
  \uch{1}\lessdot\cdots\lessdot \uch{\binom{m+n}{2}}\}.\end{equation} Since $\uch{A}=\vmn$ and there are no inversions among values that are at least $m$ in $\vmn$, $\uch{k}$ is $m$-132 free for $1\le k\le A$. For $A+1\le k\le\binom{m+n}{2}$, $\tau_{k-A}=\rho_k$ acts on $\{m,m+1,\ldots,m+n\}$ only. A reduced word in $\shredw{M}$, for any $M$, corresponds to a maximal length chain in the Tamari and is therefore both 132 and 312 free \cite{BW}. Thus $\tau$ is 132 free since it is in $\shredw{n+1}$ and cannot produce a $m$-132 pattern. By \eqref{eq:same_start}, there are no $a<b<c$ where $\uch{k}(a)<\uch{k}(c)<\uch{k}(b)$ and $\uch{k}(a)<m$ 
\end{proof}

\subsection{Standard Young tableaux}
In this subsection, we discuss the image of commuting shuffles under
the Edelman-Greene bijection. We begin with a lemma on the effect of a
single transposition.

\omitt{
\susanna{Is there a $yxz\sim yzx$ around we could use instead?}
\tvi{(Dec 18) There isn't anything currently set up, but you are welcome to define it and use any new notation that helps. }
}
\begin{lemma}
  \label{lem:commutingInsert}
  Let $\rho\in\cshuf m n$ and suppose $\rho$ is a commuting shuffle of $\sigma\in\redw{\vmn}$ and $\tau\in\shredw{\longestk{n+1}}$. Let $I$ and $J$ be the indices of $\sigma$ and $\tau$ respectively in $\rho$. Suppose $i\in I$ and $i+1\in J$ and that $\rho_i$ and $\rho_{i+1}$ commute. Let $\rho'=\rho_1\rho_2\cdots\rho_{i+1}\rho_i\rho_{i+2}\cdots\rho_N$. If $Q=\EGQ(\rho)$ and $Q'=\EGQ(\rho')$, then $Q'=s_i(Q)$, where $s_i$ acts on $Q$ by swapping the places in $Q$ of $i$ and $i+1$. 
  \end{lemma}



\begin{proof}
For $h\le \binom{n+n}{2}$, let $(P(h),Q(h))$ denote the Edelman-Greene tableaux for $\rho_1\rho_2\cdots\rho_{h}$. It is enough to show that inserting $\rho_{i+1}$ and then $\rho_i$ into $(P(i-1),Q(i-1))$ simply swaps $i$ and $i+1$ in the recording tableau obtained from inserting first $\rho_i$ and then $\rho_{i+1}$. We do this by determining their bump paths and the cells containing $i$ and $i+1$. \omitt{Denote $\rho_i$ by $x$ and $\rho_{i+1}$ by $y$. Since $y$ came from $\tau$, $y\le n$.} We begin by describing $P(\ell-1)$ for any $\ell\in J$.

Let $z$ denote $\rho_{\ell}$. We know a great deal about $P(\ell-1)$ using
Claim~\ref{claim:vmn} and the properties of words that insert to
shiftable tableaux. The reduced expression $\rho$ is the result of
commuting some, or possibly no, elements of $\tau\in\shredw{\longestk{n+1}}$ leftward past
elements of $\sigma\in\redw{\vmn}$ in $\sigma\tau$.   By
Claim~\ref{claim:vmn}, Part~\ref{claim:decseq}, we know there is a increasing subsequence $12\cdots m+n-1$
in $\sigma$. The simple transposition $z$ could not have commuted past $z+1$, so $1\ldots z+1$
are among $\rho_1\rho_2\ldots\rho_{j-1}$ and so are in row 1 of
$P(j-1)$. Further, we have decreasing subsequences in $\sigma$ of the
form $k,k-1,\ldots,k-m+2$ for $k=m-1,m,\ldots,m+n-1$. The $1$s before
$z$ in $\tau$ cannot commute past the decreasing sequences ending
with $1$ or $2$ in $\rho$ coming from $\sigma$, so these decreasing sequences precede $z=\rho_{j}$ in $\rho$. Similarly,
the $2$s before $z$ in $\tau$ cannot commute past the decreasing
subsequences ending with $1$, $2$, or $3$ in $\rho$ coming from $\sigma$, so the decreasing
subsequences for $k=m-1,m,m+1$ all precede $z$ in $\rho$. In short,
disjoint decreasing subsequences of the form $k,k-1,\ldots,k-m+2$
precede $z$ in $\rho$ for $k=m-1,m,\ldots,m+z-1$.
When the decreasing subsequence $k,k-1,\ldots,k-m+2$ from $\sigma$ is
fully inserted, then there is a $k-m+1+h$ in row $h$ of the
insertion tableau for $h$ from $1$ to $m-1$. Each element bumps the previous element from the subsequence down one row. In
summary, row $h$ of $P(j-1)$ contains at least $h,h+1,\ldots,z+h$ for $h$ from $1$ to $m-1$. See Figure~\ref{fig:commutingInsertSchematic}. We see immediately that the bump path of $z$ will be in column $z$ for rows $1$ through $m-1$ and that the cell labelled $\ell$ in $Q(\ell)$ will be in row $h\ge m$. 
\ytableausetup{boxsize=2.6em,mathmode}
\begin{figure}
    \centering
    \begin{ytableau}
1&2&3&\none[\cdots]&z&z+1&\none[\cdots]\\
2&3&4&\none[\cdots]&z+1&z+2&\none[\cdots]\\
\none[\vdots]&\none[\vdots]&\none[\vdots]&\none[\vdots]&\none[\vdots]&\none[\vdots]&\none[\cdots]\\
\scriptstyle m&\scriptstyle m+1&\scriptstyle m+2&\none[\cdots]&\scriptstyle z+m-1&\scriptstyle z+m&\none[\cdots]\\
\end{ytableau}
    \caption{The tableau $P(\ell-1)$ described in the proof of Lemma~\ref{lem:commutingInsert}. It is the Edelman-Green $P$-tableau for $\rho_1\rho_2\cdots\rho_{\ell-1}$ where $\ell\in J$.}
    \label{fig:commutingInsertSchematic}
\end{figure}

\ytableausetup{mathmode,boxsize=normal}

We are now in a position to consider what occurs when we insert
$x=\rho_{\ell}$ for some $\ell\in I$. If no elements of $\tau$ have commuted past $x$, then the position of the cell added when $x$ is inserted as element of $\rho$ is in the same row and column as the cell added when $x$ is inserted as an element of $\sigma$. Now suppose at least one element of $\tau$ has commuted past $x$. Any elements of $\tau$ that have commuted past $x$ have bumped through to a row with index strictly greater than $m-1$. If some $z$ from $\tau$ has commuted to the left of $x$ in $\rho$, then $|z-x|>1$. The reduced expression $\tau\in\shredw{\longestk{n+1}}$ is a (double) lattice word \cite{FN} so $1,2,\ldots,z-1$ precede $z$ in $\tau$ and have thus also commuted past $x$ in $\rho$, showing that $x-z>1$. From this, we see that none of the entries in rows $1$ through $m-1$ were changed by the element of $\tau$.  Therefore, as in the case when no elements from $\tau$ have commuted past $x$, the cell added when $x$ is inserted as an element of $\rho$ is in the same row and column as the cell added when $x$ is inserted as an element of $\sigma$.   

In particular, when $x=\rho_i$ and $z=\rho_{i+1}$, where $i\in I$ and $i+1\in J$ we see that the bump paths of $x$ and $z$ are independent
of the order in which $x$ and $z$ are inserted.
\end{proof}

\begin{example}\label{ex:commutingInsert}
Let $m=5$ and $n=3$, so that $\vmn$ is the permutation
$[5,6,7,8,4,3,2,1]$. We take the reduced words
$\sigma=2132435467567213425345\in\redw{v_{53}}$,
$\tau=121321\in\shredw{\longestk{4}}$, and
$\rho=213243546756721342\mathbf{1}53\mathbf{2}4\mathbf{1321}5\in\cshuf{\sigma}{\tau}$,
where the elements of $\tau$ are bold-faced in $\rho$. The indices for the elements of $\tau$ in $\rho$ are $J=\{19,22,24,25,26,27\}$. We may commute the elements $\rho_{23}=4$ and $\rho_{24}=1$.  Please see Figure~\ref{fig:commutingInsert} for the bump paths.
  
\end{example}

\begin{figure}[h]
\centering
  \begin{tikzpicture}
\node at (0,0){\begin{ytableau}1&2&3&4&5&6&7\\2&3&4&5&6&7\\3&4&5&7\\4&5&6\\5&6\end{ytableau}};
\node at (5,0){\begin{ytableau}1&2&3&\mathbf{4}&5&6&7\\2&3&4&\mathbf{5}&6&7\\3&4&5&\mathbf{6}\\4&5&6&\circled{7}\\5&6\end{ytableau}};
    \node at (10,0){\begin{ytableau}\mathbf{1}&2&3&4&5&6&7\\\mathbf{2}&3&4&5&6&7\\\mathbf{3}&4&5&7\\\mathbf{4}&5&6\\\mathbf{5}&6\\\circled{6}\end{ytableau}};

    \end{tikzpicture}

\caption{The tableau $P(22)$ on the left is the result of Edelman-Greene insertion of $\rho_1\cdots\rho_{22}$ from Example~\ref{ex:commutingInsert}. The tableau in the middle shows the bump path, in bold, after inserting $4=\rho_{23}$ into $P(22)$. On the right, we have the result of inserting $1=\rho_{24}$ into $P(22)$. The entries in the cell added by insertion are circled. The bump paths do not interact.}
\label{fig:commutingInsert}
  \end{figure}
  \omitt{
\susanna{Make sure bump path matches description of EG, $x$ not $x+1$. Use Linnusson-Potka description of EF}
\tvi{We should make a name for the Edelman-Green bijection. I would like to use notation so I can more clearly describe a proof later. How does $EG:\mathcal{R}\rightarrow SYT$ sound. Or just $EG(\sigma)$ is the $Q$-tableaux of the Edelman-Green bijection for reduced word $\sigma$? Or just $Q(\sigma)$?}
\tvi{(Dec 18) We addressed this wonderfuly. We can likely delete these comments when ready.}
}

The \mydef{$k$-staircase partition} is the partition $(k-1,k-2,\ldots,2,1)$.

\omitt{
\susanna{

\begin{definition}
\label{def:n-row-shift}
(What I'd like: )
  Let $N$ and $n$ be nonnegative integers and $n<N$. Let $T$ be a standard Young tableau of shape $\lambda=(\lambda_1,\ldots,\lambda_{N-1})$, where the rows $\lambda_{N-n-1}, \lambda_{N-n+1}, \ldots,\lambda_{N-1}$ are distinct. We say $T$ is \mydef{$n$-row-shiftable} if the entry in row $i$ and column $j$ of $T$ is less than the entry in row $i+1$ and column $j-1$ for $i$ from $N-n-1$ to $N-2$ and for $j$ from $2$ to $\lambda_i$. For the case $n=N-1$, we consider the row above the top row of $T$ to have index $0$ and the entries in row $0$ to be $0$.
  \end{definition}
  
  (3/24/24)}

\susanna{In essence, $n$-row shiftable means that if we only consider the last $n$ rows of a tableau $T$, we have a shiftable tableau: a standard shifted tableau which has been left-justified. All tableaux are trivially $0$-shiftable. If all rows of a standard shifted tableau with $n$ rows are left-justified, then it is an $n$-row-shiftable tableau, which is the same as an $(n-1)$-row-shiftable tableau. (3/24/24),4/17/24} 
\susanna{In the case $G=\lolli m n$, where we set $N=m+n$, the $n$-row-shiftable tableaux correspond under the Edelman-Greene bijection to maximal chains in $\tp{\lolli m n}$ which are $m$-132 and $m$-312 free.}
\susanna{The original definition of shiftable for a tableau with $N$-rows is equivalent to $(N-1)$-shiftable or $N$-shiftable. (3/24/24)
I think the comment in the top part of this box is from Sam 4/14/24}
} 

\begin{definition}
\label{def:n-row-shift}
\label{def:m-row-shift}
Let $N$ and $n$ be nonnegative integers and $n<N$. Let $T$ be a standard Young tableau of shape $\lambda=(\lambda_1,\ldots,\lambda_{N-1})$, where the rows $\lambda_{N-n-1}, \lambda_{N-n+1}, \ldots,\lambda_{N-1}$ are distinct. We say $T$ is \mydef{$n$-row-shiftable} if the entry in row $i$ and column $j$ of $T$ is less than the entry in row $i+1$ and column $j-1$ for $i$ from $N-n-1$ to $N-2$ and for $j$ from $2$ to $\lambda_i$. For the case $n=N-1$, we consider the row above the top row of $T$ to have index $0$ and the entries in row $0$ to be $0$.
\end{definition}

In essence, $n$-row shiftable means that if we only consider the last $n+1$ rows of a tableau $T$, we have a shiftable tableau: a standard shifted tableau which has been left-justified. All tableaux are trivially $0$-shiftable. If all rows of a standard shifted tableau with $n$ rows are left-justified, then it is an $n$-row-shiftable tableau, which is the same as an $(n-1)$-row-shiftable tableau. 

\omitt{ 
(Sam's rephrasing)
  Let $N$ and $m$ be nonnegative integers and $m\leq N$ (We could restrict $m<N$, but then we would need to delete our reference to the top row having index 0). Let $T$ be a standard Young tableau of shape $\lambda=(\lambda_1,\ldots,\lambda_N)$, where the last $m+1$ rows, $\lambda_{N-m}, \lambda_{N-m+1}, \ldots,\lambda_N$, are distinct. We say $T$ is \mydef{$m$-row-shiftable} if the entry in row $i$ and column $j$ of $T$ is less than the entry in row $i+1$ and column $j-1$ for $i$ from $N-m$ to $N-1$ and for $j$ from $2$ to $\lambda_i$. We consider the row above the top row of $T$ to have index $0$ and if needed any entry in row 0 to be 0.

\ttre{In essence, $m$-row shiftable means that if restrict ourselves to only consider the last $(m+1)$ rows of a tableau $T$ we have a standard shiftable tableau. All tableaux are trivially $0$-shiftable. The original definition of shiftable for a tableau with $N$-rows is equivalent to $(N-1)$-shiftable or $N$-shiftable.}

\susanna{Roughly speaking, instead of only fully shifted tableau, we now consider tableaux where the last $N-m$ rows least $m$ could be shifted. 

Dump this. 4/17/24}
} 
\omitt{
\tvi{Suggestion: Can we call these $m$-row-shiftable, meaning that rows $i\geq m$  have the shiftable condition? While I like the terminology of $(m,n)$-shiftable, it will help one of my proofs to talk about shiftability more generally. While I think it is good to introduce the technicality of the indices when defining, the technicality of indices in proofs will make it hard for a reader to follow.}
\susanna{Is this ok? Oct 2023}
\tvi{(Dec 18) Looks good to me. Could we add the intuitive definition too, that the last $n$ rows are shiftable if we are $n$-row shiftable? I'll make sure I'm matching these definitions in my section. }
\tvi{I am a little confused about this definition. The top row has index 0? I am not sure what that means.  (2-29-24)}
\susanna{Added an intuitive def.  It's the row above the top row which has index 0, for the m=1 case. Maybe I should add that the entries of row 0 are 0, but seems like overkill. Also, I didn't read your Dec 18 entry closely enough the last time I changed this def, and I hope my changes today are right. It means both m=1 and m=2 are shifted tabs. 3/14/24}
\tvi{TO DO SAM (write/ check definition so it works with my work later (march 15)}
\tvi{Task completed. See in red above. Old definition is still present, just not in red. (3-19-24)}
} 
Let $\stab m n$ denote the set of $n$-row-shiftable tableaux of staircase shape $\stair_N$
 where $N=m+n$. The rest of this subsection is devoted to showing that $\cshuf mn$ maps to pairs $(P,Q)$ with $Q\in\stab mn$ under the Edelman-Greene bijection.  
\omitt{In Table~\ref{tab:syt}, tableaux 1 and 2 are $(2,2)$-shifted and tableaux 1,2,6-10 are $(3,1)$-shifted. 

\begin{table}[h]
\input{sytTable}
\caption{All standard Young tableaux of shape (3,2,1)}
\label{tab:syt}
\end{table}

\begin{theorem}
A reduced word $w\in\redw{\longest^{(m+n)}}$ represents a longest chain in $\tp{\lolli mn}$ if and only if it is mapped to a $(m,n)$-shifted tableau under the Coxeter-Knuth algorithm.   
  \end{theorem}
}

\begin{example}\label{ex:shiftableTabs}
Consider the standard Young tableaux in the $5$-staircase partition in
Figure~\ref{fig:shiftableTabs}. The first on the left corresponds to
the reduced word $1241321432$ and a longest length maximal chain in $\tp{\lolli 5 0}$, the second to $342312343\mathbf{1}$ and a longest length maximal chain in  $\tp{\lolli 4 1}$, the third to $234123\mathbf{12}4\mathbf{1}$ and a longest length maximal chain in $\tp{\lolli 3 2}$, and the last to $\mathbf{1231423121}$ and a longest length maximal chain in $\tp{\lolli 1 4}$. In each reduced word, the subword which corresponds to $\tau$ has been highlighted.
\end{example}

\begin{figure} 
\begin{center}
\begin{tikzpicture}
  \def\h{4.5}
\node (1) at (0*\h,0){    \begin{ytableau}
    1&2&3&8\\4&5&9\\6&10\\7\\
    \end{ytableau}};
\node (3) at (2*\h,0){
  \begin{ytableau}
1&2&3&9\\4&5&6\\7&8\\10\\    
    \end{ytableau}};
\node (2) at (1*\h,0){
      \begin{ytableau}
1&2&7&8\\3&4&9\\5&6\\10\\      
    \end{ytableau}};
\node (4) at (3*\h,0){
          \begin{ytableau}
1&2&3&5\\4&6&7\\8&9\\10\\           
    \end{ytableau}};
\end{tikzpicture}
\end{center}

  \caption{The first tableau from the left is not $n$-row-shiftable except for
  $n=0$. The second tableau is $1$-row-shiftable and not $2$-row-shiftable, the third is
  $2$-, and therefore also $1$-, row-shiftable, and the final
  tableau is shiftable or
  $4$-row-shiftable, which is the same as $3$-row-shiftable.}  \label{fig:shiftableTabs}

  \end{figure}

\omitt{
\susanna{We should talk to make sure I have the shiftable stuff right before I change the caption for this fig. 3/14/24}
}

Suppose $\rho\in\shuf m n$ is a commuting shuffle of $\sigma\in
\redw{v_{m,n}}$ and $\tau\in\shredw{\longestk{n+1}}$. We continue to denote the set of positions
in $\rho$ held by the letters from $\sigma$ by $I$ and those
held by letters from $\tau$ by $J$. We are interested in the
positions of the elements from
$I=\{i_1<i_2<\ldots<i_{(m-1)(n+1)+\binom{m-1}{2}}\}$ and
$J=\{j_1<j_2<\cdots<j_{\binom{n+1}{2}}\}$ in the tableau
$\EG(\rho)$. Let $\EG(\sigma)|_I$ be the tableau created by replacing
$k$ with $i_k$ in $\EG(\sigma)$. Similarly, $\EG(\tau)|_J$ is
created from $\EG(\tau)$ by replacing $k$ with $j_k$. Finally, if
$T_1$ and $T_2$ are tableaux with $r_1$ and $r_2$ rows respectively
and such that row $r_1$ of $T_1$ is longer than row $1$ of $T_2$, let
$\stacktab{T_1}{T_2}$ denote the tableau with rows $1$ through $r_1$ given by
the rows of $T_1$ and rows $r_1+1$ through $r_1+r_2$ given by
$T_2$. Informally, $\stacktab{T_1}{T_2}$ is $T_1$ stacked on top of $T_2$. 

\begin{lemma}
  \label{lem:stackedTabs}
Fix integers $m$ and $n$, where $m\ge1$, and let $\rho=\sigma\tau$, where
$\sigma\in\redw{\vmn}$ and $\tau\in\shredw{\longest^{n+1}}$. Let
$(P_1,Q_1)=\EG(\sigma)$, $(P_2,Q_2)=\EG(\tau)$, and
$J=\{(m-1)(n+1)+\binom{m-1}{2}+1,\ldots,\binom{m+n}{2}\}$. Then
the $Q$ tableau of $\EG(\rho)$ is $\stacktab{Q_1}{Q_2|_{J}}$.
  \end{lemma}
\begin{proof}
Claim~\ref{claim:vmn}, part~\ref{claim:vex} lets us conclude that
$P_1$ and $Q_1$ both have shape $(m+n-1,m+n-2,\ldots,n+1)$. Row $i$ of
$P_1$ has entries $i,i+1,\ldots,m+n-1$. Since
$\tau\in\shredw{\longestk{n+1}}$, both $P_2$ and $Q_2$ with have
$(n+1)$-staircase shape and $\stacktab{Q_1}{Q_2|_J}$ is well-defined. All entries in $\tau$ are less than or equal to $n$. Each $i\in\tau$ will be straight bumped through $P_1$ (see the proof of Lemma~\ref{lem:commutingInsert}) and into row $m$ as $i+m-1$ because row $j$ of $P_1$ contains $i+j-1$ and $i+j$. Thus the bottom portion of $Q$ will be given by $Q_2|_J$
\end{proof}

\begin{proposition}
  \label{prop:stackedTabs}
Suppose $\rho\in\shuf m n$ is a commuting shuffle of
$\tau\in\redw{v_{m,n}}$ and $\sigma=\shredw{\longestk{n+1}}$. Let $I$ be indices of $\sigma$ in $\rho$, and $J$ those of $\tau$. We have
that
$$\EGQ(\rho)=\stacktab{\EGQ(\tau)|_I}{\EGQ(\sigma)|_J}.$$
\end{proposition}

\begin{proof}
This is a straightforward proof by induction on the number of
commutation moves needed to transform  $\sigma\tau$ into $\rho$. The base case is
Lemma~\ref{lem:stackedTabs} and the induction step is done using
Lemma~\ref{lem:commutingInsert}.
\end{proof}

When the reduced expression $\rho$ is equal to the concatenation
$\sigma\tau$, Proposition~\ref{prop:stackedTabs} proves that
$\EGQ(\rho)=T$ is $n$-row shiftable. Our only concern would be at the
``boundary'' between row $m-1$ and row $m$, but all entries of $J$ are
greater than all entries of $I$, so that the entry in cell $(m-1,j+1)$
is less than the entry in $(m,j)$. More work is required in
general. Luckily, the description of $\Gamma$, which is the inverse to
the Edelman-Greene map and is given in Definition~\ref{def:promo}, is
particularly simple when the tableau $T$ is shiftable.

\begin{lemma}
  \label{lem:EGForShifted}
  Let $\tau\in\shredw{\longestk{n+1}}$ and let $T=\EGQ(\tau)$. Denote the entry in row $i$, column $j$ of $T$ by $T_{i,j}$. We have that $\tau_{T_{i,j}}=j$. 
\end{lemma}

For example, in Figure~\ref{fig:shiftableTabs}, the fourth tableau has $1,4,8,10$ in its first column and these are the positions of the 1s in its reduced word. 
\begin{proof}The key to the proof is that the evacuation path for cell $(i,j)$,
where $i+j=n$, is
$\{(i,j),(i-1,j),\ldots,(1,j),(1,j-1),\ldots,(1,1)\}$ and that no
positive entry ever changes column during evacuation. By Proposition~\ref{prop:FN},
$\Gamma(T)$ is both a lattice word and a reverse lattice word, so the
last letter of $\Gamma(T)$ is a $1$, meaning the whole first column of
$T$ shifted down and $T_{1,1}$ is set to $0$. In order to be a reverse
lattice word, the second to last letter of $\Gamma(T)$ must be a
$2$. Since the first column of was shifted down,
$\delta(T)_{i,2}>\delta(T)_{i+1,1}=T_{i1}$, so the whole second column
of $\delta(T)$ shifts down, $0$ moves to column $2$, and
$\delta(T)_{1,1}=-1$. In general, if the largest entry of $\delta^k(T)$
is in column $j+1$ and column $j+1$ has been shifted down $x$ times,
then column $j$ has been shifted down $x+1$ times and column $j+1$
will be simply shifted down by applying $\delta$ to $\delta^k(T)$ and
the entries in row $1$ in columns $j-1,j-1,\ldots,1$ are not positive.
\end{proof}

Consider $\sigma\in\redw{v_{m,n}}$, $\tau=\shredw{\longestk{n+1}}$,
and $\rho\in\cshuf \sigma\tau$. We claim that $\rho$ inserts to an
$n$-row shiftable tableau. Again, consider this inductively, based on
the number of commutations from $\sigma\tau$ to $\rho$. By the
discussion following the proof of Lemma~\ref{prop:stackedTabs}, the base
case is true. By Lemma~\ref{prop:stackedTabs}, we need only be concerned
with cells on the in rows $m-1$ and $m$. Suppose that after some
commutation moves, in cell $(m-1,j)$ we have $i$ and in cell $(m,j-1)$
we have $k$, where $k>i$, where $i\in I$ and $k\in J$. We must show
that no commutation between a letter from $\sigma$ and a letter from
$\tau$ will create a violation in these cells. If $k>i+1$, then by
Lemma~\ref{lem:commutingInsert}, no violation can be created.

We are left with the case $k=i+1$. We claim that $\rho_i=j$ and $\rho_{i+1}=j-1$, so these elements may not be commuted. The reasoning is as in the proof of Lemma~\ref{lem:EGForShifted}, because $\tau$ is a lattice and reverse lattice word. Although elements from $\sigma$ may be commuted to the right of elements of $\tau$ in $\rho$, either the elements from $\sigma$ are strictly greater than $n$ or they are strictly less than the elements they commute past.

We have shown the following corollary. 

\begin{corollary}
\label{cor:shufImpliesShif}
Suppose $\rho\in\shuf m n$ is a commuting shuffle of
$\tau\in\redw{v_{m,n}}$ and $\sigma=\shredw{\longestk{n+1}}$. The tableau $\EGQ(\rho)$ is $n$-row shiftable.
\end{corollary}

We've almost arrived at the main result of this section, Theorem~\ref{thm:3equiv}. We'll need two lemmas to finish its proof. Lemma~\ref{lem:shiftTabForm} is an extension of Lemma~\ref{lem:EGForShifted} and Lemma~\ref{lem:biggerhigher} is a converse to Lemma~\ref{lem:commutingInsert}.

\begin{lemma}
\label{lem:shiftTabForm}
    Let $T$ be an $n$-row shiftable tableau, let
$J=\{j_1<j_2<\cdots<j_{\binom{n+1}{2}}\}$ be the entries from rows $m,
m+1,\ldots,m+n-1$ of $T$, and let $I$ be the entries from rows
$1,2,\ldots,m-1$, so that $I\cup J=[\binom{m+n}{2}]$. Let
$\rho=\Gamma(T)$ be the reduced expression for $\longestk{m+n}$ which
corresponds to $T$. Then $\rho_{T_{i,j}}=j$ for $i=m,m+1,\ldots,m+n-1$.
\end{lemma}
\begin{proof}
Recall that $\rho=\Gamma(T)$. The algorithm for $\Gamma$ proceeds in steps, where $\delta^kT$ is the image of $T$ after $k$ evacuations. Let $\mult{t}{T}{j}$ be the number of removals from column $j$ of $T$ in steps $1,\ldots,t$.

Because of the shiftability condition, we can describe in detail how the elements of $T$ in rows $m,\ldots,m+n-1$ are removed. As in the proof of
Lemma~\ref{lem:EGForShifted}, the fact that $\delta$ moves the entries of $T$ downward is key, but we must take more care because the shiftability condition holds only on rows $m-1,m-2,\ldots,m+n-1$, which we refer to as the lower rows of $T$. Let $t_k=\binom{m+n}{2}-j_k+1$, so that $t_k$ is the step where the $k^{\text{th}}$ element is removed from the lower rows. For an $n$-row shiftable tableau $T$ and $t=t_k,1\le k\le\binom{n+1}{2}$, we claim
\begin{enumerate}
\item \label{part:multbounds} $\mult{t}{T}{j+1}\le\mult{t}{T}{j}\le\mult{t}{T}{j+1}+1$, for $j+1\le n$ and $\mult{t}{T}{j+1}\le n-j$.
\item \label{part:eqrems}If $\mult{t}{T}{j+1}=\mult{t}{T}{j}$ and $\mult{t}{T}{j+1}\le n-j$, then $(\delta^tT)_{i+1,j}>(\delta^tT)_{i,j+1}$.
\item \label{part:uneqrems}If $\mult{t}{T}{j+1}=\mult{t}{T}{j}-1$ and $\mult{t}{T}{j+1}\le n-j$, then $(\delta^tT)_{i+1,j}<(\delta^tT)_{i,j+1}$.
\item \label{part:form} $(\delta^tT)_{i+\mult{t_k}{T}{j},j}=T_{i,j}$, where $t_k\le t<t_{k+1}$ and $\mult{t_k}{T}{j}\le n-j+1$.
\end{enumerate}

The reasoning behind the claim is straightforward, although there is a lot of notation. The column where the first element is removed from the lower rows of $T$ will have to be $t_1$ because the largest element in the lower rows of $T$ is $T_{m+n-1,1}$. The claims all hold at the first step, since the removal of $T_{m+n-1,1}$
simply shifts column 1 down. Since $T_{m+n-2,1}<T_{m+n-2,2}$, after
step $t_1$, the entry at the bottom of column 2 will be larger than
the entry at the bottom of column 1 and so the next removal from the
lower rows will be from column 2. By \eqref{part:uneqrems}, after step
$t_2$, all entries in column 2 will shift downward and there will be
no rightward shift of cells from column 1. After step $t_2$,
$\mult{t_2}{T}{j+1}=\mult{t_2}{T}{j}$ and $T$'s original shiftable
condition causes $(\delta^tT)_{i+1,j}>(\delta^tT)_{i,j+1}$ for $j=1$
and $m\le i\le m+n-2$. The process continues in this manner. Entries
from rows $1,\ldots, m-2$ may shift rightward, but that will not
affect the lower rows.

The result of the claim is that $\rho_{T_{ij}}=j$, for $i=m,m,+1,\ldots,m+n-1$.
\end{proof}

\begin{lemma}
\label{lem:biggerhigher}
Let $T$ be an $n$-row shiftable tableau of  $(m+n)$-staircase shape and $T=\EGQ(\rho)$ for $\rho\in\redw{\longestk{m+n}}$. Suppose $k$ is in an entry in row $i_k> m-1$ and $k+1$ is in row $i_{k+1}\le m-1$. Then $\rho_k$ and $\rho_{k+1}$ commute.  Further, let $T'$ be the tableau $T$ with entries $k$ and $k+1$ swapped. Then $\Gamma(T')$ is the reduced expression $\rho$ with the positions of $\rho_k$ and $\rho_{k+1}$ swapped.
\end{lemma}

\begin{proof}
    Let $T$ and $\rho$ be as  in the statement of the lemma, where $k$ is the entry in row $i_k$ and column $j_k$ and $k+1$ in row $i_{k+1}$ and column $j_{k+1}$. By Lemma~\ref{lem:shiftTabForm}, $\rho_k=j_k$. Since $T$ is $n$-row shiftable, $k+1$ cannot be in column $j_k$ or $j_{k+1}$. Therefore, $j_{k+1}>j_k+1$. In any tableau, an entry moves to the right (higher column index) or lower (higher row index) as the elementary promotion operator $\delta$ is applied. Therefore $$\rho_{k+1}\ge j_{k+1}>j_k+1=\rho_k+1,$$ showing that $\rho_k$ and $\rho_{k+1}$ commute.

    Define the \mydef{exit path} of a label $i$ in a tableau $R$ to be the sequence of cells $i$ occupies $R$, $\delta(R)$, $\delta^2(R),\ldots.$ For the second part of the lemma, let $T'$ be the tableau $T$ with entries $k$ and $k+1$ swapped. We claim that $\delta^tT$ is equal to $\delta^tT'$ with $k$ and $k+1$ swapped for $t$ such that $0\le t\le\binom{m+n}{2}-k$ and they are simply equal if $t>\binom{m+n}{2}-k$. This is because $k$ and $k+1$ will never be compared in step \eqref{part:promo_comp} of the definition of $\delta$ in Definition~\ref{def:promo}, since they are always separated by at least a column. Any comparison between $k$ and $j$, $j\neq k+1$, will have the same result as a comparison between $k+1$ and $j$. Therefore, the exit path of $k$ in $T$ is the same as the exit path of $k+1$ in $T'$ and exit path of $k+1$ in $T$ is the same as the exit path of $k$ in $T'$, forcing $\Gamma(T')$ to be $\rho$ with $\rho_k$ and $\rho_{k+1}$ swapped. 
\end{proof}
\begin{theorem}\label{thm:3equiv}
  The following are equivalent. Let $m$ and $n$ be integers, with $m\ge1$, $\rho\in\redw{\longestk{m+n}}$, and $C$ the chain of permutations which corresponds to $\rho$.
  \begin{enumerate}
  \item\label{thmpart:shuf} $\rho\in\shuf m n$
  \item \label{thmpart:shif}$T=\EGQ(\rho)$ is an $n$-row-shiftable tableau of shape $(m+n)$-staircase.
  \item \label{thmpart:tp}$C\in\lc {\lolli m n}$
    \end{enumerate}
  
  \end{theorem}


In Section~\ref{sec:balancedTabs}, we add a fourth item to this list.

\begin{proof}

 In Proposition~\ref{prop:MLCsAndCSs}, we showed that
 \eqref{thmpart:shuf} and \eqref{thmpart:tp} are equivalent. By
 Corollary~\ref{cor:shufImpliesShif}, \eqref{thmpart:shuf} implies
 \eqref{thmpart:shif}. To finish, we show \eqref{thmpart:shif} implies
 \eqref{thmpart:shuf}.

We actually prove a bit more. Let $T$ be an $n$-row shiftable tableau, let $\rho=\Gamma(T)$, let
$J=\{j_1<j_2<\cdots<j_{\binom{n+1}{2}}\}$ be the entries from rows $m,
m+1,\ldots,m+n-1$ of $T$, and let $I$ be the entries from rows
$1,2,\ldots,m-1$, so that $I\cup J=[\binom{m+n}{2}]$. For any subset $K$ of $[\binom{m+n}{2}]$, set $\rho_K$ to be the subsequence of $\rho$ with indices in $K$ We claim that $\rho\in\cshuf \sigma \tau$, where $\sigma$ is $\rho_I$, $\tau$ is $\rho_J$, $\sigma\in\redw{\vmn}$, and $\tau\in\shredw{\longestk{n+1}}$. We use induction on the number of \mydef{shuffle inversions}. The pair $(i,j)$ is a shuffle inverison if $i<j$, and if $i$ is in row $r_i$ and $j$ is in row $r_j$, then $r_i\ge m$, and  $r_j<m$. The set of shuffle inversions of $T$ is
denoted $\INV(T)$. 

Suppose $|\INV(T)|=0$. Then $I=[B]$, where $B=\ell(\vmn)$ and $J=\{B+1,\ldots,\binom{m+n}{2}\}$, and $\rho$ is the concatenation of $\rho_I$ and $\rho_J$. We see right away that $\EGQ(\rho_I)$ is $T|_I$ and therefore $\EGQ(\rho_J)$ must be the tableau given by rows $m,m+1,\ldots,m+n-1$ of $T$, with the entries replaced by $[\binom{n+1}{2}]$ in order. It is straightforward to see that $\sigma\in\redw{\vmn}$ and $\tau\in\shredw{\longestk{n+1}}$ Let $\sigma=\rho_I$ and $\tau=\rho_J$ and we are done with the base step.

Assume $\INV(T)$ is nonempty. Then there is at least one pair of the form $(k,k+1)\in\INV(T)$. Let $k$ be the least of all $k$ such that $(k,k+1)\in\INV(T)$. Let $T'$ be $T$ with the positions of $k$ and $k+1$ swapped. The tableau $T'$ is also $n$-row shiftable and $|\INV(T')|=|\INV(T)|-1$, so by induction, $\Gamma(T')=\rho'$ is a commuting shuffle of $\rho'_{I'}$ and $\rho'_{J'}$, where $I'$ is the set of entries from rows $1,\ldots,m-1$ of $T'$ and $J'$ is the remainder of the entries of $T'$. We now apply Lemma~\ref{lem:biggerhigher}: since $\rho=\rho'$ with $\rho_k$ and $\rho_{k+1}$ swapped, $I=I'\setminus\{k\}\cup\{k+1\}$ and $J=J'\setminus\{k+1\}\cup\{k\}$, and $\rho_k$ and $\rho_{k+1}$ commute, we are done. 

\omitt{
Let $T_J$ be the tableau formed from rows $m,m+1,\ldots,m+n-1$, with
the entries relabeled, preserving order, by
$1,\ldots,\binom{n+1}{2}$. Since $T_J$ is shifted,
$\Gamma(T_J)\in\shredw{\longestk{n+1}}$, and so by
Lemma~\ref{lem:EGForShifted} plus the result of the claim at the start
of this proof, $\sigma=\Gamma(T_J)$ and $\sigma\in\shredw{\longestk{n+1}}$.

Shiftability is crucial here. See Figure~\ref{fig:badtab}.

We still must show that $\sigma\in\redw{\vmn}$ and that
$\rho\in\cshuf{\sigma}{\tau}$. To see that $\rho\in\cshuf{\sigma}{\tau}$, it is enough to show that if $x\in J$ and $x+1\in I$, then $|\rho_x-\rho_{x+1}|>1$. Suppose $T_{i,j}=x$ and $T_{k,\ell}=x+1$, where $i\ge m$ and $k<m$. We have that $\rho_x=j$, since $x$ was removed from the shiftable part of $T$ and $\rho_{x+1}\ge\ell$, since entries cannot move to the left during promotion. Because $x+1>x>T_{i-1,j+1}$, $x+1$ cannot be in column $j+1$, so that $\ell>j+1$.

We may now commute elements of $\sigma$ and $\tau$ to obtain
$\rho'=\sigma\tau$ and $\EG(\rho')$ has $(m+n-1)$-staircase shape. By
Lemma~\ref{lem:stackedTabs}, the tableau $\EG(\sigma)$ is rows $1$
thorough $m-1$ of $\EG(\rho')$. We again use Corollary~4.2
of \cite{stan84}, this time to show that $\sigma\in\redw{\vmn}$.
}
\end{proof}
\begin{figure}
\centering
\begin{tikzpicture}
\node at (0,0) {$T$=\begin{ytableau} 1&2&3&4\\5&7&9\\6&10\\8\end{ytableau}};
\node at (4,0) {$T_J$=\begin{ytableau} 1&3\\2\end{ytableau}};

\end{tikzpicture}
\caption{Consider $T$ above. Let $m=3$ and $n=2$, so that $T$ is not $(m,n)$-shiftable and $J=\{6,8,10\}$. Then $\rho=\Gamma(T)=1234312132$, $\rho|_J=112$, but $\sigma=\Gamma(T_J)=212$.}
\label{fig:badtab}
\end{figure}

\omitt{
Suppose $\rho\not\in\lc{\lolli m n}$. Let $C=\{e=\wch 0\lessdot\wch
 1\lessdot\cdots\lessdot\wch{\binom{m+n}{2}}=\longestk{m+n}\}$ be the
 chain in $\S_{m+n}$ under weak order such that $\wch i=\wch
 {i-1}\rho_i$. We say $w\in\S_{m+n}$ has a \mydef{violation at $h$}
 if $\lolli m n\mid_{\{w_1,\ldots,w_{h-1}\}}$ is connected but $\lolli
 m n\mid_{\{w_1,\ldots,w_{h-1},w_{h}\}}$ is not. Since
 $\rho\not\in\lc{\lolli m n}$, there exists $i$ and $h$ such that
 $\wch{i}\in C$ has a violation at $h$. Let $k-1$ be the largest such
 $i$ and $h$ be the smallest for $\wch{k-1}$: $\wch{k-1}$ has a violation and $\wch{i}$ does not for $i\ge k$. 
} 

\omitt{
Let $T$ be the tableau associated to $\rho$ by the Edelman-Greene
bijection. We claim that $h\le n$, that $k$ is the last entry in row
$m+n-h$ of $T$, and that the last entry of row $m+n-h+1$ is less than
$k$. Together, these claims show that $T$ is not
$(m,n)$-shiftable. The rest of the proof explains why the claims are valid.
\susanna{Claim 2 and 3 not necessarily true}

Because $\lolli m n\mid_{\{\wch{k-1}_1,\ldots,\wch{k-1}_{h-1}\}}$ is
 connected but $\lolli m
 n\mid_{\{\wch{k-1}_1,\ldots,\wch{k-1}_{h-1},\wch{k-1}_{h}\}}$ is not
 and $\wch{k}$ has no violation, we may conclude that $\rho_k=h$
 and $\wch{k-1}_{h+1}>\wch{k-1}_h\geq m$. Furthermore, either
 $\wch{k-1}_1,\ldots,\wch{k-1}_{h-1}<\wch{k-1}_h$ or
 $w_1,\ldots,w_{h-1}>w_{h}$. The graph $\lolli m
 n\mid_{\{\wch{k-1}_1,\ldots,\wch{k-1}_{h-1},\wch{k-1}_{h+1},\wch{k-1}_{h}\}}=\lolli
 m
 n\mid_{\{\wch{k}_1,\ldots,\wch{k}_{h-1},\wch{k}_{h},\wch{k}_{h+1}\}}$
 is connected, which would not be possible if
 $\wch{k-1}_1,\ldots,\wch{k-1}_{h-1}$ were all less than
 $\wch{k-1}_h$. Therefore,
 $\wch{k-1}_1,\ldots,\wch{k-1}_{h-1}>\wch{k-1}_h$. We remark that
 there are thus at least $h-1$ values from $[m+n]$ which are greater
 than $\wch{k-1}_h$, which together with the fact that $\wch{k-1}_h\ge
 m$, shows that $h\le n$, the first of our claims.

We may also conclude that
$\wch{k-1}(1)=\wch{k-1}(2)+1, \wch{k-1}(2)=\wch{k-1}(3)+1, \ldots,\wch{k-1}(h-1)=\wch{k-1}(h)+1$,
since $h$ is the first violation in $\wch{k-1}$. 
} 
\omitt{
\susanna{Not done, following paragraph is not true}
What's more,
$\wch{k-1}_1=m+n,\wch{k-1}_2=m+n-1,\ldots,\wch{k-1}_{h-1}=m+n-(h-1)+1$,
since there are no violations in $\wch{k-1}$ which are less than
$h$. Indeed,
$\wch{k}_1=m+n,\wch{k}_2=m+n-1,\ldots,\wch{k}_{h-1}=m+n-(h-1)+1, \wch{k}_{h}=m+n-h+1$. Therefore,
$\rho_i>h$ for $i>k$.

Now we consider $\Gamma$, the inverse of Edelman-Greene's bijection,
to see why the second and third of our claims are valid and finish the
proof. At each step, the algorithm removes the largest entry in the
tableau, which is necessarily at the end of a row. If the row of the
entry removed at step $j$ is $i$, then $\rho_{\binom{m+n}{2}-j+1}$ is
set to $m+n-i$. Since $\rho_k=h\le n$ and $\rho_i>h$ for $i>k$, we
have that at step $\binom{m+n}{2}-k+1$, the entry was taken from row
$m+n-h\ge m$, and that all entries larger than this entry were removed
from rows $m+n-\rho_i<m+n-h$. In particular, the entry at the end of
row $m+n-h+1$ is less than the entry at the end of row $m+h-h$.
} 

\subsection{Balanced tableaux}
\label{sec:balancedTabs}
It is straightforward to characterize the balanced tableaux which correspond to maximal length chains in $\tp{\lolli m n}$, because of Corollary~\ref{cor:intraclassLmn}. Let $N=m+n$. To be maximal length, there must be no intra-class hyperplanes in the corresponding walk. According to Corollary~\ref{cor:intraclassLmn}, this means that for all $a<b$ and for all $j$ such that $a<N-j\le b-1$ and $N-j\ge m$ ($n\ge j$), the entry in $(a,N-b+1)$ is smaller than the entry in $(N-j,N-b+1)$.

\section{Maximum length chains in $\tp{G}$ and a Young quasisymmetric Schur function expansion}
\label{sec:LMF}


In this section we will show that Stanley's symmetric function for the longest permutation $\longestk{m+n}$ summing over only the reduced decompositions corresponding to longest maximal chains, rather than all chains,  of $\tp{L_{m,n}}$ is a positive sum of  Young quasisymmetric Schur functions.  To do this we will briefly introduce symmetric functions, quasisymmetric functions and Stanley's symmetric function. For more details see~\cite{LMv13,S99}.

A function $f(x)\in \mathbb{Q}[x_1,x_2,\ldots]$ is called {\it symmetric} if the coefficient of 
$x_1^{\alpha_1}x_2^{\alpha_2}\cdots x_k^{\alpha_k}$ is the same as the coefficient of 
$x_{w(1)}^{\alpha_1}x_{w(2)}^{\alpha_2}\cdots x_{w(k)}^{\alpha_k}$
for all permutations $w\in\fS_n$ for $n\geq 1$. A function $f(x)\in \mathbb{Q}[x_1,x_2,\ldots]$ is called {\it quasisymmetric} if the coefficient of $x_1^{\alpha_1}x_2^{\alpha_2}\cdots x_k^{\alpha_k}$ is the same as the coefficient of $x_{i_1}^{\alpha_1}x_{i_2}^{\alpha_2}\cdots x_{i_k}^{\alpha_k}$ for any sequence of positive integers $i_1<i_2<\cdots<i_k$.
\begin{example}
The function $x_1^2x_2+x_1^2x_3+x_2^2x_3+\cdots$ is quasisymmetric, but is not symmetric. Whereas $x_1^2x_2+x_1x_2^2+x_1^2x_3+x_1x_3^2+x_2^2x_3+x_2x_3^2+\cdots$ is both symmetric and quasisymmetric. All symmetric functions are quasisymmetric. 
\end{example}
All bases for the algebra of symmetric functions are indexed by integer partitions.  While there are many bases for the algebra of symmetric functions we define the Schur basis. 

A {\it semi-standard tableau (SSYT)} $T$ is a Young diagram of an integer partition where each cell is filled with a positive integer so that 
\begin{enumerate}
\item rows weakly increase, $T(i,j)\leq T(i,j+1)$ for all $i,j\geq 1$, 
\item and columns strictly increase, $T(i,j)<T(i+1,j)$ for all $i,j\geq 1$,
\end{enumerate}
where we consider cells $T(i,j)$ with $i,j\geq 1$ outside the Young diagram to be filled with $\infty$. The standard Young tableaux (SYT) from Section~\ref{sec:prelim tableaux} are SSYT but requires that the union of all cell entries is $[n]$ if $\lambda\vdash n$. Let $\text{SSYT}(\lambda)$ be the collection of all SSYT with underlying Young diagram $\lambda$.
The {\it Schur function} associated to integer partition $\lambda$ is 
$$s_{\lambda}=\sum_{T\in\text{SSYT}(\lambda)}x^{T}$$
where $x^T=\prod x_i^{m_i(T)}$ and  $m_i(T)$ is the number of times $i$ appears in $T$. We call a symmetric function  {\it Schur positive} if it is a nonnegative sum of Schur symmetric functions. 

The algebra of quasisymmetric functions also has many famous bases. The two we  introduce are the fundamental basis and the  Young quasisymmetric Schur functions. Any basis of the quasisymmetric functions is indexed by  integer compositions. An {\it integer composition} $\alpha = (\alpha_1,\alpha_2,\ldots ,\alpha_l)$ is an ordered list of positive integers where each $\alpha_i$ is called a {\it part}. If the parts sum to $n$ we write $\alpha\models n$ or $|\alpha|=n$. Integer compositions of $n$ are in bijection with ordered pairs $(S,n)$, where $S$ is a subset of the positive integers and $n>\max(S)$. Given an integer partitions $\alpha=(\alpha_1,\alpha_2,\cdots,\alpha_{l})$ of $n$, we pair this with $(S,n)$ where $S=\{\alpha_1,\alpha_1+\alpha_2,\ldots, \alpha_1+\alpha_2+\ldots\alpha_{{l}-1}\}$. For notation,  let $\set(\alpha)=S$. 
Let $S$ be a subset of the positive integers, $n>\max(S)$ and $\alpha$ be the corresponding integer composition of $n$. The {\it fundamental basis} is 
$$Q_{S,n}(x)=Q_{\alpha}(x)=\underset{i_s\neq i_{s+1}\forall s\in S}{\sum_{i_1\leq i_2\leq \cdots\leq i_{n}}}x_{i_1}x_{i_2}\cdots x_{i_{n}}$$

To define Young quasisymmetric Schur functions we  define a new kind of tableaux, semi-standard Young column tableaux (SSYCT) based on Young diagrams of compositions. Let $\alpha=(\alpha_1,\alpha_2,\ldots,\alpha_{l})$. Row 1 of the Young diagram is at the top with $\alpha_1$ boxes and row $l$ is at the bottom with $\alpha_{l}$ boxes. All rows are left aligned. The Young diagram of the composition $\alpha$ is filled with positive integers to form the tableau $\tau$ with the following rules:
\begin{enumerate}
\item Rows  weakly increase, so $\tau(i,j)\leq \tau(i,j+1)$ for all $i,j\geq 1$. 
\item The first column strictly increases, so $\tau(i,1)<\tau(i+1,1)$ for all $i\geq 1$
\item If $(j,k+1)\in \alpha$ and $\tau(i,k)\leq \tau(j,k+1)$ for pairs $i$ and $j$ with $i>j$, then $\tau(i,k+1)<\tau(j,k+1)$. 
\end{enumerate}
If $(i,j)$ is not a box in the Young diagram, i.e. $(i,j)\notin \alpha$, and $i,j\geq 1$ then we consider $\tau(i,j)=\infty$. We call a SSYCT a standard Young column tableau (SYCT) if the set of the entries of $\tau$ equals $[|\alpha|]$. See Figure~\ref{fig:SSYCT} for an example.

\begin{figure}
\begin{center}
\begin{tikzpicture}
\draw (0,0) rectangle (.5,.5);
\draw (.25,.25) node {$a$};
\draw[dashed] (-.5,-1) rectangle (0,-.5);
\draw (-.25,-.75) node {$b$};
\draw[dashed] (0,-1) rectangle (.5,-.5);
\draw (.25,-.75) node {$c$};

\begin{scope}[shift={(6,0)}]
\node at (5,0){\begin{ytableau}
1&2&3&11&17&22&24\\
4&5&13&15&18&23\\
6&9&14&16&27\\
7&10&19&28\\
8&12&20\\
21&25\\
26
\end{ytableau}};
\node at (0,0){\begin{ytableau}
1&2&3&11&27\\
4&5&20\\
6&12&19&28\\
7&10&14&16&18&23\\
8&9&13&15&17&22&24\\
21&25\\
26
\end{ytableau}};
\end{scope}

\end{tikzpicture}
\end{center}
\caption{On the left we illustrate the layout of condition 3 for SSYCT: if $a\geq  b$, then $a>c$. Note that only the cell with $a$ has to be in the composition. In the middle we have a SYCT $Y$ of shape $(5,3,4,6,7,2,1)$. On the right we have the SYT of shape $\stair_8$ such that $\hat{\rho}_{\emptyset}(Y)=T$, where $\hat{\rho}_{\emptyset}$ is defined later.}
\label{fig:SSYCT}
\end{figure}
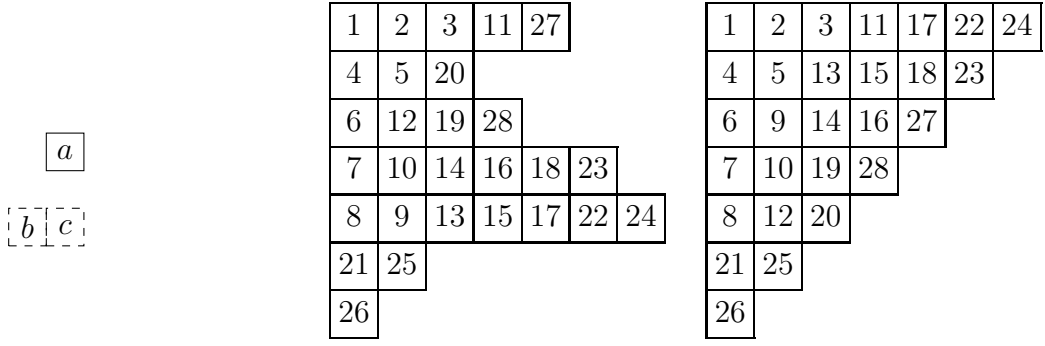

The {\it Young quasisymmetric Schur function} for composition $\alpha$ is 
$$\hat{\mathscr{S}}_{\alpha}=\sum_{\beta}\hat{d}_{\alpha,\beta}Q_{\beta}$$
where 
$\beta$ is a composition of $|\alpha|$ and $\hat{d}_{\alpha,\beta}$ counts the number of SYCT of shape $\alpha$ and $\Des(\tau)=\set(\beta)$. Let $\hat{D}_{\alpha}$ be the collection of all SYCT of shape $\alpha$. 

\begin{example}
For example $Q_{(2,2)}=x_1x_1x_2x_2+x_1x_2x_3x_3+x_1x_1x_2x_3+x_1x_2x_3x_4+\cdots$. The term $x_1x_2x_2x_3$ does not appear in  $Q_{(2,2)}$ (i.e. has a coefficient of 0). The Young quasisymmetric Schur function for $\alpha=(2,2)$ is $\hat{\mathscr{S}}_{\alpha}=Q_{(1,2,1)}+Q_{(2,2)}$.
\end{example}

The quasisymmetric function we define and investigate in this section comes from summing over a subset of reduced expressions that are in the definition of Stanley's symmetric function.  Given a permutations $w\in\fS_N$ the {\it Stanley symmetric function} is  
$$\cF_w(x)=\sum_{\sigma\in \cR(w)}Q_{\Des(\sigma),|w|}$$
where $\Des(\sigma)=\{i:\sigma(i)>\sigma(i+1)\}$. The Stanley symmetric function for each permutation is indeed a symmetric function even through it is defined as a sum of quasisymmetric functions. This function helped prove the first enumeration formula for the number of reduced words for a permutation and proved that the number of reduced words for $\longestk{N}$ is the number of SYT of the staircase shape $\stair_N$~\cite{stan84}. A bijection was found later by Edelman and Greene~\cite{EG}, which we describe in Definition~\ref{def:EG}.

\begin{theorem}[Stanley~\cite{stan84}] 
The Stanley symmetric function $\cF_w$ is Schur-positive for all permutations $w\in\fS_N$ and for the longest permutation $\longestk{N}$ Stanley's symmetric function is the Schur function of the staircase shape,
$$\cF_{\longestk{N}}=s_{\stair_N}.$$
\label{thm:staley symmetric function}
\end{theorem}
The quasisymmetric function that this section studies, $\LMF{G}$ where $G$ is a connected  filled graph, is only a certain portion of  Stanley's symmetric function for the longest word $\longestk{N}$. This is the portion  associated to longest maximal chains in $\tp{G}$,  
$$\LMF{G}=\sum_{\sigma\in \LMR{G}}Q_{\Des(\sigma),|\sigma|}.$$

\begin{example}The two reduced words associated to the two longest maximal chains in $\tp{P_4}$ are $123121$ and $121321$. This can be seen in Figure~\ref{fig:PathEquiv}. 
Then, $\LMF{P_4}=Q_{(3,2,1)}+Q_{(2,2,1,1)}=\hat{\mathscr{S}}_{(3,2,1)}$.
\end{example}

\begin{theorem} 
\label{thm:Longest chain quasi Schur}
For $m\geq 1$ and $n\geq 0$,  $\LMF{L_{m,n}}$ is a positive sum of Young quasisymmetric Schur functions. In particular
$$\LMF{L_{m,n}}=\sum_{\alpha\in \text{comp}(m,n)}\hat{\mathscr{S}}_{\alpha}$$
where $\text{comp}(m,n)$ is the set of compositions $(\alpha_1,\alpha_2,\ldots ,\alpha_{m-1},n,\cdots,2,1)$ where $[\alpha_1,\alpha_2,\ldots, \alpha_{m-1}]$ is a permutation of $[n+1,n+m-1]$.

\label{thm:LCF}
\end{theorem}
In order to prove Theorem~\ref{thm:LCF} we will need to use the bijection $\hat{\rho}_{\emptyset}$, originally appearing in~\cite{LMv13}, between SSYCT and SSYT. The bijection $\hat{\rho}_{\emptyset}$ also restricts to a bijection between SYCT and SYT. Both  $\hat{\rho}_{\emptyset}$ and  $\hat{\rho}^{-1}_{\emptyset}$ are easy to define. To define $\hat{\rho}_{\emptyset}$ we start with a SSYCT $Y$.
\omitt{
\tvi{(Sam (Dec 18): Using $\tau$ breaks our convention of reserving Greek letters for reduced words. What might be better?) Try $Y$. }
\tvi{(3-19-24) It should be adjusted now. }}
We form $\hat{\rho}_{\emptyset}(Y)$ with two steps:
\begin{enumerate}
\item Rearrange all rows in $Y$ so that the underlying diagram is a Young diagram of an integer partition.
\item Write all  columns of $Y$ in increasing order, so that the smallest element is in row one.   
\end{enumerate}
See Example~\ref{fig:SSYCT}. The inverse $\hat{\rho}_{\emptyset}^{-1}$ is slightly more complicated. We start with a SSYT $T$ and we form $\hat{\rho}_{\emptyset}^{-1}(T)$ inductively column by column. Column one is fixed. 
\begin{enumerate}
\item Start with column 2 and the smallest element in column 2. Place this number and cell in the row with the highest index such that
\begin{enumerate}[(i)]
\item there is a cell in the previous column in the same row and 
\item we maintain the condition that rows weakly increase. 
\end{enumerate}
\item Repeat this for all other elements in column 2 of $T$. 
\item Repeat steps 1 and 2 for columns $3, 4, \ldots,t$ of $T$ in order where $T$ has $t$ columns.
\end{enumerate}
See Example~\ref{fig:SSYCT}. 

We want to show that the map $\hat{\rho}_{\emptyset}$ restricts to bijection between $n$-row-shiftable tableaux of the staircase shape $\stair_{m+n}$ and SYCT of shapes in $\text{comp}(m,n)$. Since the map $\hat{\rho}^{-1}_{\emptyset}$ is defined inductively, we will set up our proof inductively. Let $\stair(n,c)$ be the partial staircase shape $(c^{n-c},c-1,c-2,\ldots,1)$ with $n-1$ rows and $c$ columns. 
Essentially $\stair(n,c)$ is $\stair_n$, but we are only keeping the first $c$ columns, i.e. columns $c+1,c+2,\ldots,n-1$ are removed. 
See Figure~\ref{fig:rho proof aid} for an example and visual aid for the proof in Lemma~\ref{lem:rho details}.

\begin{figure}
\begin{center}
\begin{tikzpicture}

\begin{scope}[shift={(0,.25)}]
\draw (0,0) rectangle (3,2);
\draw (.5,0) --(.5,2);
\draw (2,0) --(2,2);
\draw (2.5,0) --(2.5,2);
\draw (1.25,1) node {$\cdots$};
\draw (1.25,-.5) node {$\cdots$};
\draw (2,-.5) rectangle (3,0);
\draw (2,-1) rectangle (2.5,-.5);
\draw (2.5,0)--(2.5,-.5);
\draw (0,0) rectangle (.5,-2.5);
\draw(0,-.5) -- (.5,-.5);
\draw(0,-1) -- (.5,-1);
\draw(0,-2) -- (.5,-2);
\draw (0.25,2.25) node {\tiny 1};
\draw (2.25,2.25) node {\tiny$ c-1$};
\draw (2.75,2.25) node {\tiny$ c$};
\draw (.25,-1.5) node {$\vdots$};
\draw (-.6,-.25) node {\tiny$ n-c$};
\draw (-.6,-.75) node {\tiny$ n-c+1$};
\draw (-.6,-2.25) node {\tiny$ n-1$};
\end{scope}
\begin{scope}[shift={(6,0)}]
\node at (0,0){\begin{ytableau}
1&2&3\\
4&5&11\\
6&12&20\\
7&10&17\\
8&9&14\\
21&25\\
26
\end{ytableau}};
\end{scope}
\begin{scope}[shift={(10,0)}]
\node at (0,0){\begin{ytableau}
1&2&3\\
4&5&11\\
6&9&14\\
7&10&17\\
8&12&20\\
21&25\\
26
\end{ytableau}};
\end{scope}

\end{tikzpicture}
\end{center}
\caption{On the left we have a general picture of $\stair(n,c)=(c^{n-c}, c-1,\ldots, 2, 1)$. In the middle we have a  SYCT $Y$ with shape $\stair(8,3)$ and its associated SYT $T$ on the right where $\hat{\rho}_{\emptyset}(Y)=T$ and  $T$ is $2$-row shiftable.}
\label{fig:rho proof aid}
\end{figure}
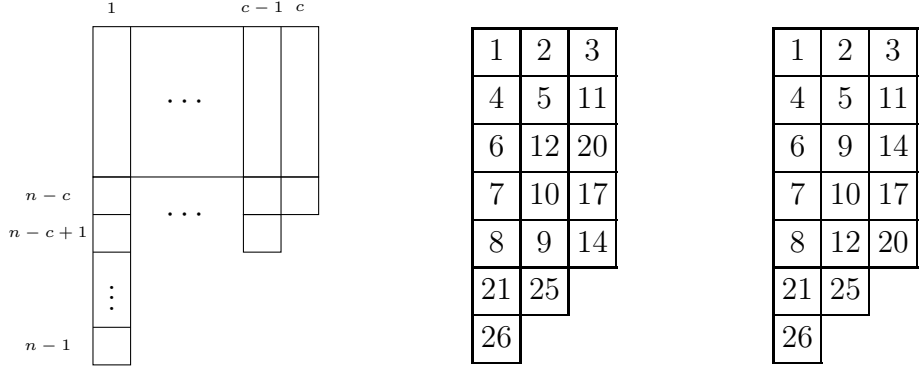

\begin{lemma} 
\label{lem:rho details}
Suppose we have a SYT $T$ and a SYCT $Y$ with $\hat{\rho}_{\emptyset}(Y)=T$. Then,
\begin{enumerate}[(i)]
\item If $T$ has shape $\stair(n,c)$, then $Y$ has shape $\alpha$ where the parts of $\alpha$ are a rearrangement of the parts of $\stair(n,c)$. In particular if $T$ in addition is $(c-1)$-row shiftable, then the shapes of $T$ and $Y$ are equal.
\item For a fixed $r$ and $c$ if $Y(r,c)\geq Y(i,c)$ for all $i<r$, then $Y(r,j)>Y(i,j+1)$ and $Y(r,j)>Y(i,j)$ for all $i<r$ and $j< c$. 
\item If $Y$ has shape $\stair(n,c)$, then  $Y$ and $T$ are equal in rows $i\geq n-c+1$ and $T$ is $(c-1)$-row shiftable.
\end{enumerate}
\end{lemma}

\begin{proof}
{\it (i)} Suppose we have a SYT $T$ of shape $\stair(n,c)$ and $\hat{\rho}^{-1}_{\emptyset}(T)=Y$. If $c=1$, then both $Y$ and $T$ have shape $(1^{n-1})$, and we are done. Now suppose $c>1$. Let $T'$ be $T$ with column $c$ removed and $Y'$ be $Y$ with column $c$ removed. We know $T'$ has shape $\stair(n,c-1)$. By the definition of the map $\hat{\rho}^{-1}_{\emptyset}$ we have that $\hat{\rho}^{-1}_{\emptyset}(T')=Y'$.  
By induction we can  conclude that $Y'$ has a shape that is a rearrangement of the shape $\stair(n,c-1)$, even though $T'$ and $Y'$ aren't exactly SYT and SYCT respectively  since  they are not standardized to numbers $1$ through $N$ for some $N$.
We can get $Y$ from $Y'$ if we repeat step 1 of $\hat{\rho}^{-1}_{\emptyset}$ for column $c$ in $T$. 
We know $Y'$ has $n-c+1$ rows of length $c-1$ and column $c$ of $T$ has $n-c$ entries. By rule 1{\it(i)} of $\hat{\rho}^{-1}_{\emptyset}$ each of the $n-c$ entries of column $c$ of $T$ must be placed in one of the $n-c+1$ rows of length $c-1$. As a result, $Y$ will have $n-c$ rows of length $c$, one row of length $c-1$ and by induction one row of length $i$ for $1\leq i<c-1$. This is a shape that is a rearrangement of $\stair(n,c)$. Now additionally suppose that $T$ is $(c-1)$-row shiftable for $c>1$. This means that $T'$ is $(c-2)$-row shiftable and by induction $Y'$ and $T'$ have the same shape $\stair(n,c-1)$. 
Because columns of $T$ strictly increase and because $T(n-c+1,c-1)>T(n-c,c)$, when forming $Y$ from $Y'$ using column $c$ in $T$, we get that no entry in column $c$ can go into row $n-c+1$ and keep that row weakly increasing. This means that the shapes of $T$ and $Y$ are equal. 

{\it (ii)} Suppose that our SYCT $Y$ has $Y(r,c)\geq Y(i,c)$ for all $i<r$. By condition 3 of SSYCT we must have that $Y(r,c-1)>Y(i,c)$ for all $i<r$.  Since rows weakly increase, condition 1 of SSYCT, we have $Y(r,c-1)>Y(i,c-1)$ for all $i<r$. 
This is identical to the assumption for part {\it(ii)} of this lemma, but for column $c-1$. Continuing this argument, but using column $c-1$ and the other columns of smaller indices, gives us all aspects of our conclusion. In simple terms, if we have one entry $Y(r,c)$ greater than or equal to all elements above it in the same column, then any entry $Y(r,j)$ to the left of $Y(r,c)$ is strictly greater than everything in above it in the same column, $j$, and everything above it in the one greater column, $j+1$. 

{\it (iii)} 
Now instead suppose that we have a SYCT $Y$ of shape $\stair(n,c)$ and that $\hat{\rho}_{\emptyset}(Y)=T$. If $c=1$, then $T=Y$ are trivially $(c-1)$-row shiftable. Suppose that $c>1$ and that $T'$ and $Y'$ are $T$ and $Y$ with column $c$ removed. By the definition of the map $\hat{\rho}_{\emptyset}$ we have that $\hat{\rho}_{\emptyset}(Y')=T'$.  
Because $Y'$ has shape $\stair(n,c-1)$ by induction we know that the entries of $Y'$ and $T'$ are equal on rows $i\geq n-c+2$ and $T'$ is $(c-2)$-row shiftable.  
Since $T$ and $T'$, and also the pair $Y$ and $Y'$, are equal in the first $c-1$ columns, 
this means that $Y$ and $T$ are equal on rows $i\geq n-c+2$ and are both at least $(c-2)$-row shiftable. This is almost everything we need short the shiftability condition between rows $n-c$ and $n-c+1$. 
Note that $Y(n-c+1,c)=\infty$, so is greater than all entries in column $c$. By part {\it(ii)} of this lemma, we have that $Y(n-c+1,j)>Y(i,j)$ for all $i<n-c+1$ for $j<c$. It follows from the inductive assumption and the mapping of $\hat\rho_{\emptyset}(Y')=T'$ that any entry of $Y$ in row $n-c+1$ must be less than the entry in the  row with larger index $n-c+2$ below it. We can conclude that $Y$ and $T$ must be equal on row $n-c+1$ as well. Again since $Y(n-c+1,c)=\infty$, so is greater than all entries in column $c$ by part {\it(ii)} of this lemma, we have that $Y(n-c+1,j)>Y(i,j+1)$ for all $i<n-c+1$. This assures us that the entry in $Y$ that will be moved to position $(n-c,j)$ when forming $T$ from $Y$ will have $T(n-c+1,j-1)>T(n-c,j)$ for any $2\leq j\leq c$. Hence, $T$ is $(c-1)$-row shiftable. 
\end{proof}

\begin{lemma}
\label{lem:rho restricts}
The map $\hat{\rho}_{\emptyset}$ restricts to bijection between $n$-row-shiftable tableaux and SYCT of shapes in $\text{comp}(m,n)$ for $m\geq 1$ and $n\geq 0$.
\end{lemma}
\begin{proof}This proof essentially comes down to showing that $\hat{\rho}_{\emptyset}$ is well defined under this restriction. 

Suppose that $T$ is a $n$-row-shiftable tableau of the staircase shape $\stair_{m+n}=(m+n-1, m+n-2,\ldots, 2,1)$ and  $T(i+1,j-1)>T(i,j)$  for all $i\geq m-1$ and $2\leq j\leq m+n-i$.  We want to show that $Y=\hat{\rho}^{-1}_{\emptyset}(T)$ is a SYCT with a shape in $\text{comp}(m,n)$. 
By Lemma~\ref{lem:rho details}{\it(i)} we know that the shape of $Y$ is a rearrangement of $\stair_{m+n}$.
Since $T$ is $n$-row shiftable, by Lemma~\ref{lem:rho details}{\it(i)} we must have the first $n$-columns of $Y$ be shape $\stair(m+n,n)$, which  means that the shape of $Y$ is in $\text{comp}(m,n)$.

Now suppose that $Y$ is a SYCT of shape $\alpha\in\text{comp}(m,n)$. We can easily see that the shape of $T=\hat{\rho}^{-1}_{\emptyset}(Y)$ is the staircase shape $(m+n-1, m+n-2,\ldots, 2,1)$. What is left to show is that $T$ is $n$-row-shiftable. Let $T'$ and $Y'$ be $T$ and $Y$ respectively restricted to the first $n+1$ columns. By Lemma~\ref{lem:rho details}{\it(iii)} since $Y'$ has shape $\stair(m+n,n+1)$ we know that $T'$ is $n$-row shiftable. Since $T$ and $T'$ are equal in the first $n+1$ columns, we have that $T$ is also $n$-row shiftable.
\end{proof}
\omitt{
\tvi{Sam (Dec 12): Note that our LMR notations might be difference since Sam's is using parentheses and Susanna's is using subscripts. A decision will need to be made, but it is easy since we are using macros.}
\tvi{(3-19-24) I believe this was addressed. And in any case this is an easy fix later as we do a read through. No worries about these comments anymore.}
}
\begin{proof}[Proof of~\ref{thm:Longest chain quasi Schur}]
We will need to show that there is a bijection between reduced words of the longest maximal chains in $\tp{L_{m,n}}$, $\LMR{L_{m,n}}$, and SYCT of shapes in $\text{comp}(m,n)$, $\varphi:\LMR{L_{m,n}}\rightarrow \cup_{\alpha\in\text{comp}(m,n)}\hat{D}_{\alpha}$. 
We additionally want the descent set to be preserved, meaning for $\sigma\in\LMR{L_{m,n}}$ we want $\Des(\sigma)=\Des(\varphi(\sigma))$. 
We actually already have all pieces of this bijection, $\varphi=\hat{\rho}_{\emptyset}\circ \EGQ$, the composition of the Edelman-Greene bijection and the $\hat{\rho}_{\emptyset}$ map defined just before this proof. 
By properties in \cite{LMv13} (Section 4.3) and in \cite{EG} (Theorem 6.27) the descent set is preserved. All that is left to show is that this map is well defined, which was done in Lemma~\ref{lem:rho restricts}. 
\end{proof}

\section{Shortest length maximal Chains}
\label{sec:SMC}

In this section instead of looking at longest maximal  chains we will look at the other extreme, shortest maximal  chains. Interestingly, these chains in $\tp{L_{m,n}}$ are in bijection with reduced words of the permutation $w_{m,n}=[m+n,m-1,m-2,\ldots, 1, m, m+1, \ldots, m+n-1]$. 


First let us recall that maximal  chains of $\tp{G}$ for a filled graph $G$ 
are in bijection with several objects described in Theorem~\ref{thm:general chains} and in particular for lollipop graphs, Theorem~\ref{thm:lollipop chains}. We will be primarily be focused on describing the objects associated to balanced tableaux.  
We will be particularly focusing on equivalence classes of balanced tableaux. Given a connected  filled $N$-vertex graph $G$. Two balanced tableaux of the staircase shape $\stair_N$ will be in the same equivalence class if the tableaux are  order isomorphic after removing all cells and entries associated to intraclass hyperplanes. Let $\MB{G}$ be the collection of balanced tableaux created after deleting all cells and the entries within them associated to intraclass hyperplanes and standardizing, i.e. replacing the $i$th smallest number with $i$.  Let $\SMB{G}$ be the tableaux in $\MB{G}$ with the smallest number of remaining cells. See Figure~\ref{fig:smc example} for some examples.

We  describe the exact conditions for elements in $\SMB{L_{m,n}}$ next and will use $\SMB{L_{m,n}}$ to prove that shortest maximal chains of $\tp{L_{m,n}}$ are in bijection with reduced words of the permutation $w_{m,n}$ defined in Theorem~\ref{thm:varsigma bijection}. Later in Section~\ref{sec:further} we  describe in more detail what the shortest maximal chains of $\tp{L_{m,n}}$ look like and  define a variation on Stanley symmetric function associated to $\tp{L_{m,n}}$. We will also show that the portion of the function defined in Section~\ref{sec:further} associated to shortest maximal chains is Schur positive.


The objects in  $\SMB{L_{m,n}}$ have shapes of skew tableaux. We  now define skew Young  diagrams. Given two integer partitions $\lambda$ and $\mu$ such that $\lambda_i\geq \mu_i$ for all $i$ supposing that $\mu_i=0$ for $i>\ell(\mu)$ the  {\it skew Young diagram} $\lambda//\mu$ will be the Young  diagram of $\lambda$ with the Young diagram of $\mu$ removed. Let $|\lambda//\mu|$ be the number of cells in the skew Young diagram. 
Suppose each cell of the skew Young diagram of  $\lambda//\mu$ is filled with a positive integer. 
We will call this a {\it balanced skew tableaux} if for any cell $(i,j)$ filled with $x$, the number of entries in the hook (cells right of $x$, below $x$ and including $x$) at most $x$ is equal to the number of cells below $x$ and including $x$. 
If the outer shape $\lambda$ is a staircase shape, then this condition is equivalently that the number of entries to the right of $x$ less than $x$ equals the number of entries below $x$ greater than $x$. The following two technical lemmas  prove and describe exactly what objects are in $\SMB{L_{m,n}}$, which we summarize in Corollary~\ref{cor:SMB conditions}. See Figure~\ref{fig:smc technical2}   for illustrations.




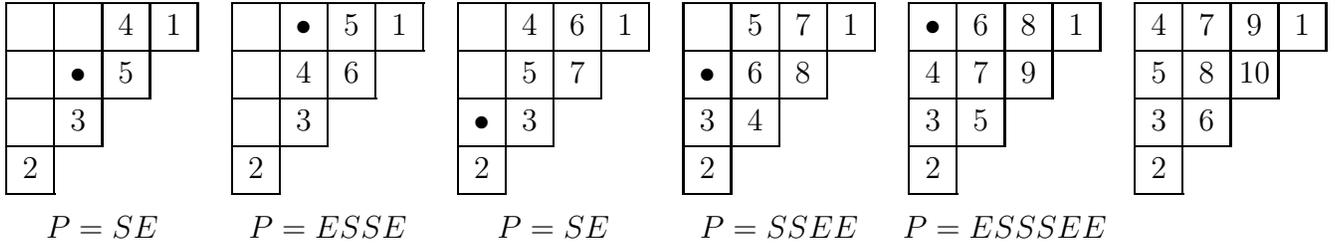
\begin{figure}
\begin{center}
\begin{tikzpicture}

\begin{scope}[shift={(0,0)}]
\node at (0,0){\begin{ytableau}
{ }&{ }&4&1\\
{ }&{\bullet}&5\\
{ }&3\\
2
\end{ytableau}};
\node at (0,-1.7) {$P=SE$};
\node at (3,0){\begin{ytableau}
{ }&{\bullet}&5&1\\
{ }&{4}&6\\
{ }&3\\
2
\end{ytableau}};
\node at (3,-1.7) {$P=ESSE$};
\node at (6,0){\begin{ytableau}
{ }&{4}&6&1\\
{ }&{5}&7\\
{\bullet}&3\\
2
\end{ytableau}};
\node at (6,-1.7) {$P=SE$};

\node at (9,0){\begin{ytableau}
{ }&{5}&7&1\\
{\bullet}&{6}&8\\
{3}&4\\
2
\end{ytableau}};
\node at (9,-1.7) {$P=SSEE$};
\node at (12,0){\begin{ytableau}
{\bullet}&{6}&8&1\\
{4}&{7}&9\\
{3}&5\\
2
\end{ytableau}};
\node at (12,-1.7) {$P=ESSSEE$};
\node at (15,0){\begin{ytableau}
{4}&{7}&9&1\\
{5}&{8}&10\\
{3}&6\\
2
\end{ytableau}};
\end{scope}

\end{tikzpicture}
\end{center}
\caption{On the left we have $B\in\mathcal{SMB}(L_{3,2})$ that satisfies all conditions stated in Lemma~\ref{lem:shortest balanced}. The sequence of $B_i'$ illustrates the argument in  Lemma~\ref{lem:shortest balanced inverse}. On the right $B'$ has all cells in $\mu=(2,2,1)$ associated to intraclass hyperplanes as per the conditions in Theorem~\ref{thm:lollipop chains}, whose removal and standardization gives us $B$.
}
\label{fig:smc technical2}
\end{figure}

\begin{lemma}
\label{lem:shortest balanced}
For $m\geq 1$ and $n\geq 0$, if $T$ is a balanced tableau of shape $\stair_{m+n}$ associated to a shortest maximal chain in $\tp{L_{m,n}}$, then all cells in $\mu=(n^{m-1},n-1,n-2,\ldots, 1)$ are associated to intra-class hyperplanes, and 
\begin{enumerate}[(i)]
\item the smallest entry in columns $i\in[n]$ is in the row with the highest index, i.e. at the bottom of the column, 
\item $T(m+n-1,1)<T(m+n-2,2)<\cdots<T(m,n)$, and 
\item $T(m,n)<T(i,n+1)$ for all $i\in[m-1]$.
\end{enumerate}
\end{lemma}
\begin{proof}
This will follow from Theorem~\ref{thm:lollipop chains}(3), which we will use continuously in this proof. Let $T$ be a balanced tableau of shape $\stair_{m+n}$ associated to some shortest maximal chain in $\tp{L_{m,n}}$. Recall that the only cells that can be associated to an intra-class hyperplane are the non-edge cells, which is defined right before Theorem~\ref{thm:intra_class_conditions}. For $G=L_{m,n}$, the non-edge cells form the integer partition  $\mu=(n^{m-1},n-1,n-2,\ldots, 1)$. For $T$ to be  associated to some shortest maximal chain, all cells in $\mu$ should be associated to an intra-class hyperplanes. 

Let us first justify condition {\it(i)}. To have a balanced tableau $T$ to associated to a shortest maximal chain we need to have all cells in $\mu=(n^{m-1},n-1,n-2,\ldots, 1)$ disappear. Let's fix the  column $j\in [n]$. To have cell $(i,j)$ in $\mu$ disappear for any $i\in [m,m+n-j]$ we need 
$T(i,j)>\min\{T(a,j):a\in [i+1,m+n-j]\}$. This implies that $T(m+n-j,j)=\min\{T(a,j):a\in [m,m+n-j]\}$.
To have cell $(i,j)$ in $\mu$ disappear for $i\in [1,m-1]$ we need 
$T(i,j)>\min\{T(a,j):a\in [m,m+n-j]\}$, and we know this minimum is $T(m+n-j,j)$. This implies that $T(m+n-j,j)=\min\{T(a,j):a\in [1,m+n-j]\}$. Hence in the first $n$ columns the cell in the row with the highest index has the minimum entry in the column.

Now let us justify condition {\it(ii)}. By part {\it(i)} we know for $j\in [n-1]$ that 
$T(m+n-j-1,j)>T(m+n-j,j)$. Since $T$ is a balanced tableau, we must have the entry right of $T(m+n-j-1,j)$ be larger so $T(m+n-j-1,j)<T(m+n-j-1,j+1)$. Putting these two inequalities together we have $T(m+n-j,j)<T(m+n-j-1,j+1)$ for $j\in [m-1]$, which proves condition {\it(ii)}. 

Finally, let us prove condition {\it(iii)} by contradiction. Let us suppose that $T(m,n)=x$ is not smaller than all entries in column $n+1$. Specifically suppose that $x<T(i,n+1)$ and $i$ is the largest index where $x<T(i,n+1)$, so $x>T(a,n+1)$ for $a\in [i+1,m-1]$. Let us check the balanced condition for $y=T(i,n+1)$. Notice that all entries below $y$ are larger than $x$ and since $x>y$ we know all entries below $y$ are larger than $y$ too. Since $T$ is balanced, all entries to the right of $y$ are smaller than $y$. 
Now let us consider the balanced condition for $z=T(i,n)$, the entry just to the left of $y$. Since the cell $(i,n)$ and its entry will disappear, we know that $x<z$ which implies $y<z$. 
Since everything to the right of $y$ is less than $y$ and $z>y$, we know that everything to the right of $z$ is less than $z$. Since $T$ is balanced we should have everything below $z$ be greater than $z$. 
However, we require $x<z$ when $x$ is below $z$. This is a contradiction, so condition {\it(iii)} must hold. 
\end{proof}

\begin{lemma}
\label{lem:shortest balanced inverse}
For $m\geq 1$ and $n\geq 0$, suppose $B$ is a balanced skew tableaux of shape $\stair_{m+n}//\mu$ where $\mu = n^{m-1}(n-1)(n-2)\ldots 21$. Assume further that
\begin{enumerate}[(i)]
\item $B(m+n-1,1)<B(m+n-2,2)<\cdots<B(m,n)$, and 
\item $B(m,n)<B(i,n+1)$ for all $i\in[m-1]$.
\end{enumerate}
Then, there exists a balanced tableaux  $T$ of shape $\stair_{m+n}$ where all cells in $\mu = n^{m-1}(n-1)(n-2)\ldots 21$ are associated to intra-class hyperplanes of $\tp{L_{m,n}}$. 
Specifically we will recursively construct a balanced tableau $T$ of shape $\stair_{m+n}$ so that the entries in any column $j\in[n]$ in union form a consecutive range of integers. 
\end{lemma}

\begin{proof}
Suppose that  $B$ is a balanced skew tableaux of the stated shape that satisfies {\it(i)} and {\it(ii)}. We will form the balanced tableaux $T$ of shape $\stair_{m+n}$ inductively column by column and cell by cell. We will fill each columns of $B$ one by one starting with column $n$ and proceeding left to columns of smaller index. We will fill each column $k$ cell by cell starting with the bottom most cell in cell $(m+n-k-1,k)$ and proceeding upwards in column $k$ to rows with lower index. 
Suppose so far we have formed $B'$ where we have filled in columns $c>k$, all cells below cell $(i,k)$ in column $k$ and are  now going to fill the cell $(i,k)$. We will form $T'$ from $B'$ by picking an $x$ to place in cell $(i,k)$ and increase all entries $z\geq x$ by one. 
We will have our conclusion if $T'$ is a balanced skew tableaux, the set of entries in column $k$  form a consecutive list of integers and there exists an entry in a row $r\geq m$ below $x$ in $T'$ that is less than $x$. 
Take the collection of numbers to the right and below the cell $(i,k)$ in $B'$. Write them in increasing order to form a word $h$. We will next form another word $P$ from $h$ by letting $P_j=E$ if $h_j$ is a number  right (east) of cell $(i,k)$ and by letting $P_j=S$ if $h_j$ is a number below (south) of cell $(i,k)$. See Figure~\ref{fig:smc technical2} for
some examples. 
Then $P$ is a sequence of $S$'s and $E$'s of length $2(m+n-k-i)$ with $m+n-k-i$ $E$'s and $S$'s. We also know that all the $S$'s will be consecutive since the integers below cell $(i,k)$ in union form a consecutive list of integers. 
We can also conclude that $P$ doesn't end in a $S$ because either  $B'(m+n-k,k)<B'(i,n+1)$ if $i\leq m-1$ or $B'(m+n-k,k)>B'(i,m+n-i)$ by our assumptions on $B$ and how $B'$ and $B$ are order isomorphic, entries have the same relative order, on the cells that they both share (i.e. the cells that $B$ has). This means that the $(m+n-k-i)$th entry of $P$ is a $S$ associated to some number $y$ below cell $(i,k)$. 
Place $y+1$ in cell $(i,k)$ and increase all $z\geq y+1$ by one to form $T'$. 
Note that this process assures us that the smallest entry in column $k$ is the bottom most one. As a result we are assured  that there is some entry below cell $(i,k)$ in a row $r\geq m$, specifically the bottom most cell in column $k$. 
Thus, cell $(i,k)$ will be associated to an intra-class hyperplane by Theorem~\ref{thm:lollipop chains}(3). The set of entries of $T'$  in column $k$ are also guaranteed  to form a consecutive list of integers by this choice. 
All that is left is for us to show that $T'$ is a balanced skew tableau. It will be sufficient to show the balanced condition for cell $(i,k)$ since all other cells satisfy this condition since $B'$ was a balanced skew tableau. Let $P=P^{(1)}P^{(2)}$ where $P^{(1)}$ is the first half of $P$ and $P^{(2)}$ is the second half. The number of entries in $T'$ to the right of $y+1$ in cell $(i,k)$ greater than $y+1$ equals the number of $E$'s in $P^{(2)}$. 
The number of entries in $T'$ below $y+1$ in cell $(i,k)$ less than $y+1$ equals the number of $S$'s in $P^{(1)}$. If there are $e_2$ $E$'s in $P^{(2)}$, then there are $e_1=m+n-k-i-e_2$ $E's$ in $P^{(1)}$, leaving the remaining $m+n-k-i-e_1=e_2$ entries in $P^{(1)}$ to be $S$'s. Thus $T'$ is a balanced skew tableaux.
\end{proof}

The last two lemmas allow us to explicitly describe the elements of  $\SMB{L_{m,n}}$.

\begin{corollary}
\label{cor:SMB conditions}
Let $m\geq 1$ and $n\geq 0$. The set $\SMB{L_{m,n}}$ contains all 
skew balanced tableaux $B$ of the the shape $\stair_{m+n}//\mu$ where $\mu = n^{m-1}(n-1)(n-2)\ldots 21$ such that 
\begin{enumerate}[(i)]
\item $B(m+n-1,1)<B(m+n-2,2)<\cdots <B(m,n)$ and
\item  $B(m,n)<B(i,m)$ for any $i\in [m-1]$. 
\end{enumerate}
\end{corollary}
\begin{proof}
Let $T$ be a balanced tableau of shape $\stair_{m+n}$ associated to a shortest maximal chain in $\tp{L_{m,n}}$. By Lemma~\ref{lem:shortest balanced} we know that all the cells in $\mu = n^{m-1}(n-1)(n-2)\ldots 21$ will be associated to intra-class hyperplanes, so will essentially disappear. 
If we standardize the remaining cells of $T$ to form a skew tableau $B$ by replacing the $i$th smallest number with $i$, then  because $T$ satisfies Lemma~\ref{lem:shortest balanced} {\it(ii)} and {\it(iii)} we know that $B$ satisfies conditions {\it(i)} and {\it(ii)} of this lemma. This proves if $B\in \SMB{L_{m,n}}$, then $B$ satisfies all conditions in this corollary.

Now let us suppose $B$ is a balanced skew tableau as described in this corollary. We want to show that there is balanced tableau $T$ of shape $\stair_{m+n}$ associated to a shortest maximal chain in $\tp{L_{m,n}}$. This means we require first that  $T$ have all cells in $\mu = n^{m-1}(n-1)(n-2)\ldots 21$ be associated to intra-class hyperplanes. We secondly require that if we delete all cell in $\mu$ and standardize, then we will get $B$, i.e. the cells in $B$ and $T$ restricted to $\stair//\mu$ have the same relative order. 

We construct $T$ as in Lemma~\ref{lem:shortest balanced inverse}. By construction the cells in $B$ and $T$ restricted to $\stair//\mu$ have the same relative order. Also by Lemma~\ref{lem:shortest balanced inverse} all cells in $\mu$ are associated to intra-class hyperplanes. 
\end{proof}

We now will focus on showing that $\SMB{L_{m,n}}$ is in bijection with reduced words of 
$$w_{m,n}=[m+n,m-1,m-2,\ldots, 1, m, m+1, \ldots, m+n-1].$$ 
To do this we will be relying on  an underlying  structure of $\cR(w_{m,n})$ created from braid moves and commutation moves. 
Tits~\cite{T69} proved that  we can transform a reduced word of $w\in\fS_N$ into any other reduced word of $w$ by a sequence of small mutation:
\begin{enumerate}
\item Exchange concurrent $i(i+1)i$ with $(i+1)i(i+1)$, or vice versa, which we will call a {\it braid move}.
\item Exchange adjacent $ij$  with $ji$, or vice versa,  if $|i=j|>1$, which we will call a  {\it commutation move}.
\end{enumerate}
\begin{remark}
If we have a set $S$ such that $\sigma\in S$ for some reduced word $\sigma\in\cR(w)$, you can get to any element in $S$ from $\sigma$ with braid or commutation moves and  $S$ is closed under all braid and commutation moves, then $S=\cR(w)$. 
\label{rmk:R_closed}
\end{remark}

Since the reduced words  of the longest permutation $\cR(\longestk{N})$ are in bijection with hyperplane walks and balanced tableaux of shape $\stair_{N}$, as stated in  Theorem~\ref{thm:longestwordSets}, there are associated braid and commutation moves on hyperplane walks and balanced tableaux. The braid and commutation moves on hyperplane walks $H_{a_1,b_1}H_{a_2,b_2}\ldots H_{a_l,b_l}$ are as follows. 
\begin{enumerate}
\item Braid moves apply to contingent triples of the form $H_{a,b}H_{a,c}H_{b,c}$ with $a<b<c$, which are switched with $H_{b,c}H_{a,c}H_{a,b}$ or vice versa. 
\item Commutation moves apply to adjacent $H_{a_i,b_i}H_{a_{i+1},b_{i+1}}$ if $a_i,b_i,a_{i+1},b_{i+1}$ are all distinct, which are switched with $H_{a_{i+1},b_{i+1}}H_{a_i,b_i}$. 
\end{enumerate}
We can translate the commutation and braid moves to the balanced tableaux in  $\text{bal}(\stair_{N})$. 
\begin{enumerate}
\item A braid applies to the cells with numbers $i,i+1,i+2$ if they appear in cells $(a,N-c+1),(a,N-b+1),(b,N-c+1)$ for some $a<b<c$ with $i+1$ appearing in cell $(a,N-c+1)$. The braid move exchanges the $i$ and $i+2$. 
\item A commutation applies to the cells with numbers $i,i+1$ if they do not appear in the same row, the same column or in cells of the form $(a,N-b+1),(c,N-a+1)$ for some $a,b,c$. The commutation move exchanges the $i$ and $i+1$. 
\end{enumerate}

We will now describe the bijection $$\varsigma:\SMB{L_{m,n}}\rightarrow \cR(w_{m,n}).$$ 
Given $B\in \SMB{L_{m,n}}$ let $I_1=\{B(m+n-1,1),B(m+n-2,2),\ldots, B(m,n)\}$ and $I_2$ be the remaining numbers in $B$ in the remaining   columns $j>n$. We know that the portion of the balanced skew tableaux in columns $j>n$ is balanced. Let $R$ be a standardized version of this balanced tableaux. We know by Theorem~\ref{thm:longestwordSets} that there is an associated reduced word $\rho$ of $\longestk{m}$. We map 
$B\mapsto \sigma$
where $\sigma$ is $\rho$ on the indices in $I_2$ and $\sigma$ is the decreasing sequence $m+n-1, m+n-2, \ldots, m$ on indices $I_1$. See Figure~\ref{fig:smc example} for an example. 


Before we can prove that $\varsigma$ is a bijection we will first need to prove a technical lemma about reduced words of the longest permutations. 

\begin{lemma}
\label{lem:omega_column1}
Given a reduced word $\sigma\in \cR(\longestk{N})$ and associated balanced tableau $B$. Let $I$ be the collection of numbers in column one of $B$. Then $\sigma$ restricted to indices $I$ is $(N-1)(N-2) \ldots 1$. Additionally, this $N-1$ is the first $N-1$ in $\sigma$. 
\end{lemma}
\begin{proof}
Consider the balanced tableau $B$. Fill column one bottom to top with $1,2,\ldots, N-1$. Fill column 2 bottom to top with $N,N+1,\ldots, 2N-3$ and so on until we place all the numbers 1 through $\binom{N}{2}$. The associated reduced word is $$((N-1)(N-2)\cdots 1)((N-1)(N-2)\cdots 2)((N-1)(N-2)\cdots 3)\cdots ((N-1)(N-2))(N-1).$$ We can see on indices $I=[n]$, the first column of $B$, we have  $(N-1)(N-2) \ldots 1$, and this $N-1$ is the first $N-1$ in $\sigma$. 

Suppose $\sigma\in \cR(\longestk{N})$ and its associated balanced tableau $B$ satisfy the conditions of this lemma. If we can show that the lemma is still true after a braid move or commutation move on $B$, then we have proven the lemma. Suppose we get $\sigma'$ from a commutation move on positions $i$ and $i+1$, which means we switch $i$ and $i+1$ in $B$ to get $B'$ associated to $\sigma'\in \cR(\longestk{N})$. 
If  $i$ and $i+1$ both appear in columns $j>1$ then we are done since column one is unchanged. 
We can not have both $i$ and $i+1$ in column one since then  $\sigma_i$ and $\sigma_{i+1}$ would be consecutive by our assumptions and we couldn't perform a commutation move. Only one case is left,  that is when exactly one of $i$ or $i+1$ appears in column one. Let $I$ and $I'$ be the set of entries in column one of $B$ and $B'$ respectively. If $i\in I$, then $I'=I-\{i\}\cup \{i+1\}$. 
By assumption $\sigma_i=a$ is part of the decreasing subsequence $(N-1)(N-2)\ldots 1$ of $\sigma$ restricted to indies $I$. Also by assumption $\sigma'_{i+1}=a$, so the subsequence of $\sigma'$ restricted to $I'=I-\{i\}\cup \{i+1\}$ is also $(N-1)(N-2)\ldots 1$. 
It is also easy to see that if the  $N-1$ in $\sigma$ was the the first $N-1$ in $\sigma$, then the $N-1$ in $\sigma'$ was the the first $N-1$ in $\sigma'$.
The case where $i+1\in I$ is similar. 

Suppose $\sigma\in \cR(\longestk{N})$ and its associated balanced tableau $B$ satisfy the conditions of this lemma. Suppose we get $\sigma'$ from a braid move on positions $i$, $i+1$ and $i+2$, which means we switch around $i$, $i+1$ and $i+2$ in $B$ to get $B'$ associated to $\sigma'\in \cR(\longestk{N})$. 
If  $i$, $i+1$ and $i+2$ all appear in columns $j>1$ then we are done since column one is unchanged. 
Suppose that at least one of $i$, $i+1$, $i+2$ is in column one. 
By the description of braid moves on balanced tableaux stated before this lemma, we know that $i+1$ must be in column one and that exactly one of $i$ or $i+2$ is also in column one. Let $I$ and $I'$ be the set of entries in column one of $B$ and $B'$ respectively. If $i\in I$, then $I'=I-\{i\}\cup \{i+2\}$. By assumption $\sigma_i\sigma_{i+1}\sigma_{i+2}=(a+1)a(a+1)$ where $\sigma_i\sigma_{i+1}=(a+1)a$ is part of the decreasing subsequence $(N-1)(N-2)\ldots 1$ of $\sigma$ restricted to indies $I$. Also by assumption $\sigma'_{i}\sigma'_{i+1}\sigma'_{i+2}=a(a+1)a$ so $\sigma'_{i+1}\sigma'_{i+2}=(a+1)a$. Thus, the subsequence of $\sigma'$ restricted to $I'=I-\{i\}\cup \{i+2\}$ is also $(N-1)(N-2)\ldots 1$. It is also easy to see that if the  $N-1$ in $\sigma$ was the the first $N-1$ in $\sigma$, then the $N-1$ in $\sigma'$ was the the first $N-1$ in $\sigma'$. The case where $i+2\in I$ is similar. 
\end{proof}

We now have all the details needed to  prove that $\varsigma$ is a bijection. See Figure~\ref{fig:smc example} for an example. 

\begin{theorem}
\label{thm:varsigma bijection}
For $m\geq 1$ and $n\geq 0$, the map $\varsigma$ is a bijection $\SMB{L_{m,n}}$ to $\cR(w_{m,n})$ where $w_{m,n}=[m+n,m-1,m-2,\ldots, 1, m, m+1, \ldots, m+n-1]$.
\end{theorem}

\begin{proof}
Let $S$ be the set of all $\varsigma(B)$ for $B\in \SMB{L_{m,n}}$. We will first show that $\varsigma$ is a bijection  from $\SMB{L_{m,n}}$ to $S$. It won't be obvious at this point that the elements of $S$ are reduced words. Next we will show that $S=\cR(w_{m,n})$ by using the strategy described in Remark~\ref{rmk:R_closed}. This will imply our result. 

First let us show that $\varsigma$ is a bijection  $\SMB{L_{m,n}}$ to $S$. We know  by the construction of $S$ that $\varsigma$ is surjective. We can show that $\varsigma$ is  injective by showing how we can recover $B$ from $\varsigma(B)=\sigma$. Let $B\in\SMB{L_{m,n}}$ and $I_1$, $I_2$ and $\rho$ be defined as in the definition of $\varsigma$. Because $\rho\in \cR(\longestk{m})$, the maximum letter in $\rho$ is $m-1$. Since $\rho$ is placed on indices $I_2$ with letters at most $m-1$ and we only place letters at least $m$ on the remaining indices of $\sigma$, we easily can recover $I_1$ and $I_2$ from $\varsigma(B)=\sigma$. This means we can recover $\rho$ as well, which implies we can reconstruct $B$ from $\sigma$. Hence $\varsigma$ is injective and thus a bijection from $\SMB{L_{m,n}}$ to $S$. 

We will now show that $S$ contains a reduced word of $w_{m,n}$. After that we will show that all other elements of $S$ are connected to this initial element via braid and commutation moves. Consider the  word $\sigma=(m+n-1)(m+n-2)\cdots m \rho$ for some fixed $\rho\in\cR(\longestk{m})$. We can see that the length of this word is $n+\binom{m}{2}$ which equals the length of $\sigma$, $\inv(w_{m,n})$. Now we will argue $\sigma$ is a reduced word of $w_{m,n}$. 
Let us apply $\sigma$ to the identity permutation $[1,2,\ldots ,m+n]$. If we apply the first portion $(m+n-1)(m+n-2)\cdots m$ of $\sigma$ to the identity we get $[1,2,\ldots m-1,m+n,m,m+1,\ldots, m+n-1]$. If we further apply $\rho$ we reverse the order of the first $m$ numbers and will get $w_{m,n}$. Hence $\sigma\in\cR(w_{m,n})$. 
Now consider the skew tableau $B$ of shape $\stair_{m+n}//\mu$ where $\mu = n^{m-1}(n-1)(n-2)\ldots 21$. Fill cells $(m+n-1,1), (m+n-2, 2),\ldots, (m,n)$ with $1,2,\ldots, n$ respectfully. There is some balanced tableau $R$ associated to $\rho$ of shape $\stair_m$. Increase all fillings of $R$ by $n$ and place a copy of this new $R$ in the right-most $m-1$ columns of $B$, i.e. in columns $j>n$. Note that $B$ satisfies all the conditions in Corollary~\ref{cor:SMB conditions} and so $B\in \SMB{L_{m,n}}$ proving that $\sigma\in S$. 

Suppose that we have a general $\sigma'\in S$ where $\varsigma(B')=\sigma'$ and $I_1'$, $I_2'$, and $\rho'$ are as in the definition of $\varsigma$. We will show that we can transform $\sigma'$ into $\sigma$ via commutation and braid moves. Later we will show that $S$ is closed under braid and commutation moves, so we will not prove right now whether each step of the transformation from $\sigma'$ to $\sigma$ is in $S$ since that will follow from $S$ being closed under commutation and braid moves. To transform  $\sigma'$ to $\sigma$ we will first perform commutation moves to move the subsequence $m+n-1, m+n-2,\ldots m$ on indices $I_1'$ in $\sigma'$ to indices $[n]$ so that we have $\sigma''=(m+n-1)(m+n-2)\ldots m \rho'$. This is possible if the maximum of $\sigma'$ on indices $[\max(I_1')]-I_2'$ is $m-2$.  By Corollary~\ref{cor:SMB conditions}{\it(i)} we have that $\max(I_1')=B'(m,n)$ and by Corollary~\ref{cor:SMB conditions}{\it(ii)} we have that $B'(m,n)<B'(i,m)$ for all $i\in[m-1]$. This means that all portions of $\rho'$ coming from column $n+1$ of $B'$ are not part of $\sigma'$ restricted to $[\max(I_1')]-I_2'$. By Lemma~\ref{lem:omega_column1} this means that the first $m-1$ in $\rho'$ is not part of $\sigma'$ restricted to $[\max(I_1')]-I_2'$. Hence the maximum of $\sigma'$ restricted to $\sigma'$ restricted to $[\max(I_1')]-I_2'$ is at most $m-2$. We can transform $\sigma''$ into $\sigma$ with braid and commutation moves by transforming $\rho'$ into $\rho$, which is possible since $\rho',\rho\in \cR(\longestk{m})$. All we have left to show is that $S$ is closed under braid and commutation moves. 

Suppose that we have a $\sigma\in S$ and $\sigma'$ is $\sigma$ after a commutation move performed on indices $i$ and $i+1$. We will show that $\sigma'\in S$. There are three cases. The first case is when  $i$ and $i+1$  are both in $I_2$. This means that the commutation move is purely happening to the associated reduced word  $\rho\in\cR(\longestk{m})$, which results in another reduced word $\rho'\in \cR(\longestk{m})$. This means $B'$ associated to $\sigma'$ may only differ from $B$ on columns $j>n$, which are changed to be associated to $\rho'$. Then $B'$ will also satisfy all conditions in  Corollary~\ref{cor:SMB conditions} proving that  $B'\in\SMB{L_{m,n}} $. Thus, $\sigma'\in S$. 
The second case is when one of $i$ and $i+1$  is in $I_1$ and the other is in  $I_2$.  This means in $B'$ we switch the numbers $i$ and $i+1$ where one is in columns $j\leq n$ and the other is in columns $j>n$. 
The only condition $B'$ may not satisfy is condition {\it(ii)} in Corollary~\ref{cor:SMB conditions}. This means that $B'(m,n)$ is no longer smaller than everything in column $n+1$. Then in $B$, we must have had $i$ in cell $(m,n)$ and $i+1$ in column $n+1$. Thus, $\sigma(i)=m$. By Lemma~\ref{lem:omega_column1} since $i+1$ is the smallest entry in column $n+1$ of $B$ the reduced word $\rho$ has $m-1$ at index $i+1$. Hence, we couldn't perform a commutation move and we have a contradiction. The third and last case is if $i$ and $i+1$ are both in $I_1$. However, this would means that $\sigma_i$ and $\sigma_{i+1}$ are consecutive by Lemma~\ref{lem:omega_column1} and again we couldn't perform a commutation move. Thus, $S$ is closed under commutation moves.

Suppose that we have a $\sigma\in S$ and $\sigma'$ is $\sigma$ after a braid move performed on indices $i$, $i+1$ and $i+2$. We will show that $\sigma'\in S$. By the construction of $\sigma$ from its associated $B\in\SMB{L_{m,n}}$ using  $\varsigma$, the braid move can not involve any $\sigma_j>m$ since there is exactly one copy of $a\in [m,m+n-1]$ in $\sigma$ by construction. 
Suppose that the braid move involves the unique $m$ in $\sigma$. This means that $\sigma_{i}\sigma_{i+1}\sigma_{i+2}=(m-1)m(m-1)$ where $i+1\in I_1$ and $i,i+2\in I_2$. 
As a consequence the subword $\sigma_{i}\sigma_{i+2}=(m-1)(m-1)$ is part of a reduced word of $\longestk{m}$, which is a contradiction. 
The only case we have to consider is when the braid move happens on indices $i,i+1,i+2$, which are all in  $I_2$. 
This means in $B$, the braid move is happening on columns $j>n$ to form $B'$. This also means that $B'$ is $B$ except for a mixing around  the numbers $i$, $i+1$ and $i+2$ that appear on columns $j>n$. This will not affect any  condition on $B$ stated in Corollary~\ref{cor:SMB conditions}. Thus, $B'\in\SMB{L_{m,n}} $, which implies that $\sigma'\in S$ and $S$ is closed under braid moves. 
\end{proof}

\begin{figure}
\begin{center}
\begin{tikzpicture}
\begin{scope}[shift={(-7,3)}]
\node at (0,0) {$\SMB{L_{m,n}}$};
\draw[->] (1.5,0)--(2.5,0) node[midway,above] {$\varsigma$};
\node at (4,0) {$\cR(w_{m,n})$};
\draw[->] (0,-1)--(0,-2); 
\draw[->] (4,-1)--(4,-2)node[midway,right] {$\psi$}; 
\node at (0,-3) {$\balanced{\stair_{m+n}}$};
\node at (4,-3) {$\SMR{L_{m,n}}$};
\draw[->] (1.5,-3)--(2.5,-3);
\end{scope}

\begin{scope}[shift={(0,3)}]
\node at (0,0){\begin{ytableau}
{ }&{ }&4&1\\
{ }&{ }&5\\
{ }&3\\
2
\end{ytableau}};
\node at (1.5,0) {$=B$};
\draw[->] (1.75,-.5)--(1.75,-2.5);
\end{scope}
\begin{scope}[shift={(0,0)}]
\node at (0,0){\begin{ytableau}
{4}&{7}&9&1\\
{5}&{8}&10\\
{3}&6\\
2
\end{ytableau}};
\node at (1.5,0) {$=T$};
\node at (1.5,-1) {$w=1432143243$};
\end{scope}
\begin{scope}[shift={(7,3)}]
\draw[->] (-4.5,0)--(-1.5,0) node[midway,above] {$\varsigma$};
\node at (0,0) {$\sigma = 14321$};
\end{scope}
\begin{scope}[shift={(7,0)}]
\draw[->] (-5,0)--(-4,0);
\draw[->] (0,2.5)--(0,.5)node[midway,right] {$\psi$};
\node at (0,0) {$\gamma=(1,2)(1,5,4,3,2)(2,5,4,3)(4,5)(3,4)$};
\end{scope}
\begin{scope}[shift={(0,-2)}]
\node at (0,0) {$[1,2,3,4,5],[2,1,3,4,5],[5,2,1,3,4],[5,4,2,1,3],[5,4,2,3,1],[5,4,3,2,1]\in \mathcal{SMC}(L_{3,2})$};
\end{scope}

\end{tikzpicture}
\end{center}
\caption{On the right we have an example of various objects related to shortest maximal chains of $\tp{L_{3,2}}$ related by bijections. 
We have the sequence of $3$-132 avoiding permutations at the bottom which is associated to $w\in \cR(w_{3,2})$, $w\in \mathcal{SMR}(L_{3,2})$, $B\in \mathcal{SMB}(L_{3,2})$, and $T\in \balanced{\stair_{m,n}}$ associated to $B$ as described in Lemma~\ref{lem:shortest balanced inverse} and pictured in Figure~\ref{fig:smc technical2}.  On the left we have a diagram that we prove to be commuting in Theorem~\ref{thm:R to SMR bijection}. }
\label{fig:smc example}
\end{figure}

In the following theorem we summarize all objects we have proven to be in bijection with shortest maximal chains of $\tp{L_{m,n}}$. We include one additional object here $\SMC{L_{m,n}}$, which is defined in detail in Section~\ref{sec:further}. In essence, $\SMC{L_{m,n}}$ is the collection of sequence of  $m$-132 avoiding permutations as described in Theorem~\ref{thm:lollipop chains}. Each permutation is a maximal element of its associated equivalence class in a shortest maximal chain of $\tp{L_{m,n}}$.  
An example can be found in Figure~\ref{fig:smc example}.

\begin{theorem} 
\label{thm:SMC bijection}
Let $m\geq 1$ and $n\geq 0$. The following are in bijection:
\begin{enumerate}[(i)]
\item Shortest maximal chains of $\tp{L_{m,n}}$, which are $\SMC{L_{m,n}}$.
\item Skew balanced tableaux $B\in \SMB{L_{m,n}}$ of the the shape $\stair_{m+n}//\mu$ where $\mu = n^{m-1}(n-1)(n-2)\ldots 21$ such that $B(m+n-1,1)<B(m+n-2,2)<\cdots B(m,n)$ and $B(m,n)<B(i,m)$ for any $i\in [m-1]$. 
\item Reduced words of $w_{m,n}=[m+n,m-1,m-2,\ldots, 1, m, m+1, \ldots, m+n-1]$.
\end{enumerate}
\end{theorem}
\begin{proof} In  {\it(i)} we have the bijection between shortest maximal chains of $\tp{L_{m,n}}$ and $\SMC{L_{m,n}}$ by definition. We get from {\it(i)} to  {\it(ii)} from  Corollary~\ref{cor:SMB conditions}. Finally, we get from  {\it(ii)} to {\it(iii)} using  Theorem~\ref{thm:varsigma bijection}.
\end{proof}
\section{Expanding on Stanley's Symmetric Function}
\label{sec:further}

In Section~\ref{sec:LMF} we introduced Stanley's symmetric function $\cF_{w}$ for a permutation $w$. We proved in Theorem~\ref{thm:Longest chain quasi Schur} that $\cF_{\longestk{m+n}}$  restricted to the longest chains in $\tp{L_{m,n}}$ is a positive sum of Young quasisymmetric Schur Functions. 
We also proved in Theorem~\ref{thm:SMC bijection} that the shortest maximal chains of $\tp{L_{m,n}}$ are in bijection with reduced words of a specific permutation $w_{m,n}=[m+n,m-1,m-2,\ldots, 1, m, m+1, \ldots, m+n-1]$. Since $\cF_{w}$ is known to be {\it Schur positive}, a positive sum of Schur functions,  for all permutations $w$ this implies that there might be a generalization of $\cF_{\longestk{m+n}}$ that connect these results. We present one possible generalization of Stanley's symmetric function $\cF_{\longestk{N}}$  that encapsulates Theorem~\ref{thm:Longest chain quasi Schur}  and Theorem~\ref{thm:SMC bijection}. 

Recall that Stanley's symmetric function $\cF_{\longestk{N}}$ is defined by summing $Q_{\Des(\sigma),N}$ over all reduced words $\sigma\in\cR(\longestk{N})$, i.e. chains in the weak order $\fS_n$, Theorem~\ref{thm:longestwordSets}. While we have a notion of chains in $\tp{G}$ for a filled graph $G$, we do not have a notion of a descent set for these chains. We note that given  two chains of the Bruhat order $\fS_N$ within the same equivalence class in $\tp{G}$, we are not guaranteed the reduced words associated to the two chains have the same descent set.
Our approach will not view maximal chains of $\tp{G}$ as equivalence classes of maximal chains of $\fS_N$, i.e. $\cR(\longestk{N})$ or equivalently sequences of permutations related by adjacent transpositions. Instead we will view maximal chains of $\tp{G}$ as sequences of permutation, where each permutation is a maximal element of its equivalence class in $\tp{G}$. We will find that these sequences of permutations are related by cycles rather than just adjacent transpositions. 
Specifically, these cycles are what we will call {\it adjacent cycles}, $c\in\fS_N$. Meaning that that there is some $a<b$ where $c(i)=i$ for all $i\notin [a,b]$, $c(i)=i-1$ for $i\in (a,b]$ and $c(a)=b$. We write $c=(a,b,b-1,\ldots, a+1)$. 
See Figure~\ref{fig:smc example} for an example.  We will first prove that when one equivalence class is covers another in $\tp{G}$, for a connected filled graph $G$, that the respective maximal elements are related by an adjacent cycle.

\begin{proposition}
Let $G$ be a connected filled graph and suppose that $u,v\in\fS_n$ with $[\![ u]\!]\lessdot [\![ v]\!]$ in $\tp{G}$ are both maximal elements in their equivalence classes. Then $u\circ c = v$ for some cycle $c=(a,b,b-1,\ldots, a+1)$ with $a<b$. 
\label{prop:chains to cycle seq}
\end{proposition}

\begin{proof}
Suppose that $u$ is the maximal element of $[\![ u]\!]$ and $v$ is the maximal element of $[\![ v]\!]$ for $[\![ u]\!]\lessdot [\![ v]\!]$ in $\tp{G}$ for a connected filled graph $G$. 
This means that $u\lessdot u_1$ for some $u_1\in [\![ v]\!]$ where $u\circ s_{t}=u_1$. Consider $u_k=u\circ s_t\circ s_{t-1}\circ\cdots\circ s_{t-k+1}$ where $u=u_0$. 
Note that $s_t\circ s_{t-1}\circ\cdots\circ s_{t-k+1}$ is an adjacent cycle for all $k\geq 1$.  There will be some smallest $k\geq 1$ where either $u_{k-1}\lessdot u_{k}$ but $u_{k}\not< u_{k+1}$
or $u_k\in [\![ v]\!]$ but $u_{k+1}\notin [\![ v]\!]$. We claim that $v=u_k$. This will prove the lemma since equivalence classes have unique maximal and minimal elements~\cite{BM} (Theorem 4.14). 

It will suffice to prove that $u_k$ is a maximal element of its equivalence class. By Lemma~\ref{lem:ab-cut} a permutation $w\in\fS_N$ is a maximal element of its equivalence class if whenever $w(i)<w(i+1)$ we have that $\{w(i+2),\ldots, w(N)\}$ is not a $w(i)w(i+1)$-cut set in $G$. By our assumptions we know that $u$ satisfies this condition and we will prove that $u_k$ satisfies this condition. 

First let us set up some notation to make the rest of the proof easier. Let $u=XYabZ$ be the breakdown of $u$ into  subwords where $a=u(t)$, $b=u(t+1)$, $X=u(1)u(2)\ldots u(t-k)$, $Y=u(t-k+1)\ldots u(t-1)$, and $Z=u(t+2)u(t+3)\ldots u(N)$. This way $u_k=XbYaZ$. Note that by how we chose $k$ that $b>Y(i)$ for all $i$ and because  $u\lessdot u_1$ we know that $b>a$.  

We will how show that $u_k$ is a maximal element by considering every  case where  $u_k(i)<u_k(i+1)$ for  $i\in[N-1]$. 
We will conclude in all cases that $\{u_k(i+2),\ldots, u_k(N)\}$ is not a $u_k(i)u_k(i+1)$-cut set. 
First consider the case where $u_k(i)u_k(i+1)$ is part of $X$ or $Z$. 
Since $u$ is a maximal element we know that $\{u(i+2),\ldots, u(N)\}=\{u_k(i+2),\ldots, u_k(N)\}$ is not a $u(i)u(i+1)$-cut set, equivalently $u_k(i)u_k(i+1)$-cut set since $u(i)=u_k(i)$ and $u(i+1)=u_k(i+1)$ in this case. 
Next consider when $u_k(i)u_k(i+1)$ is part of $Y$. Suppose that $u_k(i)=Y(j)$. Since $u$ is a maximal element we know that $S=\{Y(j+2),\ldots, Y(k-1),a,b,Z(1),\ldots,Z(N-t-1)\}$ is not a $Y(j)Y(j+1)$-cut set. We have our conclusion if $S-\{b\}$ is also not a $Y(j)Y(j+1)$-cut set. This is true because $G$ is filled and $b>Y(j),Y(j+1)$. 

The next case is if $u_k(i+1)=b$. We have our conclusion by how we chose $k$. Now consider if $u_k(i)=b$ so $u_k(i+1)=a$ if $k=1$ or $u_k(i+1)=Y(1)$ if $k>1$. In either case $b>u_k(i+1)$, so there is nothing to check. 

Next consider if $u_k(i+1)=a$ so $u_k(i)=b$ if $k=1$ and $u_k(i)=Y(k-1)$ if $k>1$. There is nothing to check in this case because $b>a$ and $b>Y(k-1)$ if $k>1$. 

Finally consider if $u_k(i)=a$, which forces $u_k(i+1)=Z(1)$. We  will suppose that the length of $Z$ is more than one  and $a<Z(1)$ since otherwise there is nothing to show.  First consider the case where $b>Z(1)$. By our assumption on $u$ we know that $\{Z(1),Z(2),\ldots, Z(N-k-1)\}$ is not an $ab$-cut set. Because $a<Z(1)<b$ as vertices of $G$ and $\{Z(1),Z(2),\ldots, Z(N-k-1)\}$ is not an $ab$-cut set we can conclude that $\{Z(2),\ldots, Z(N-t-1)\}$ is not an $aZ(1)$-cut set, which completes this case. Next instead consider when $b<Z(1)$. Because $u$ is a maximal element we know that $\{Z(2),Z(2),\ldots,Z(N-t-1)\}$ is not a $bZ(1)$-cut set, so there is a path $P$ from $b$ to $Z(1)$ not involving vertices in $\{Z(2),Z(3),\ldots,Z(N-t-1)\}$. Because $u$ is a maximal element we know that $\{Z(1),Z(2),\ldots,Z(N-t-1)\}$ is not a $ab$-cut set, so there is a path $Q$ from $a$ to $b$ not involving vertices in $\{Z(1),Z(2),\ldots,Z(N-t-1)\}$. If we concatenate the paths $P$ and $Q$ we get a path from $a$ to $Z(1)$ not involving vertices in $\{Z(2),Z(3),\ldots,Z(N-t-1)\}$, which is exactly what we want. Hence, $u_k$ must be a maximal element and since maximal elements are unique in their equivalence classes in $\tp{G}$ we are done.
\end{proof}


Using Proposition~\ref{prop:chains to cycle seq} we can define the objects that will take the place of reduced words in our generalization of Stanley's symmetric function. 
Let $\MR{G}$ be the collection of sequences of cycles associated to the maximal chains of $\tp{G}$. 
This is a generalization of $\cR(\longestk{N})$ since $\MR{K_N}$ and $\cR(\longestk{N})$ are essentially equivalent. Given a $\gamma\in \MR{G}$ we define the {\it length} of $\gamma=(a_1,b_1,b_1-1,\ldots,a_1+1)(a_2,b_2,b_2-1,\ldots,a_2+1)\ldots(a_\ell,b_\ell,b_\ell-1,\ldots,a_\ell+1)$ to be $|\gamma|=\ell$ and the {\it descent set} of $\gamma$ to be $$\Des(\gamma)=\{i:b_i\geq b_{i+1}\}.$$ 
The generalization of Stanley's symmetric function for a connected filled graph $G$ is 
$$\cF_G=\sum_{\gamma\in \MR{G}}Q_{\Des(\gamma),|\gamma|}.$$

\begin{example}
For the path $P_4=L_{1,3}$ we have $\MR{P_4}$ containing nine elements for the the nine maximal chains that you can see in Figure~\ref{fig:PathEquiv}. These nine elements of $\MR{P_4}$ are:
\begin{enumerate}
\item $\gamma=(1,2)(2,3)(3,4)(1,2)(2,3)(1,2)$ and $\Des(\gamma)=\{3,5\}$
\item $\gamma=(1,2)(2,3)(1,2)(3,4)(2,3)(1,2)$ and $\Des(\gamma)=\{2,4,5\}$
\item $\gamma=(1,2)(2,3)(3,4)(1,3,2)(2,3)$ and $\Des(\gamma)=\{3,4\}$
\item $\gamma=(1,3,2)(2,3)(3,4)(2,3)(1,2)$ and $\Des(\gamma)=\{1,3,4\}$
\item $\gamma=(1,2)(1,4,3,2)(3,4)(2,3)$  and $\Des(\gamma)=\{2,3\}$
\item $\gamma=(1,3,2)(2,4,3)(3,4)(1,2)$ and $\Des(\gamma)=\{2,3\}$
\item $\gamma=(1,3,2)(2,4,3)(1,2)(3,4)$ and $\Des(\gamma)=\{2\}$
\item $\gamma=(1,4,3,2)(2,3)(3,4)(2,3)$ and $\Des(\gamma)=\{1,3\}$
\item $\gamma=(1,4,3,2)(2,4,3)(3,4)$  and $\Des(\gamma)=\{1,2\}$
\end{enumerate}
Certain homogeneous portions of $\cF_{P_4}$ are positive sums of Young quasisymmetric Schur functions or positive sums of Schur symmetric functions, but the whole function is not homogeneous, not a positive sum of Schur symmetric functions and not a positive sum of Young quasisymmetric Schur functions. Specifically the portion that has homogeneous degree five is not positive sum of  Young quasisymmetric Schur functions. 
\begin{align*}
\cF_{P_4}&=Q_{(3,2,1)}
+Q_{(2,2,1,1)}
+Q_{(3,1,1)}
+Q_{(1,2,1,1)}
+2Q_{(2,1,1)}
+Q_{(2,2)}
+Q_{(1,2,1)}
+Q_{(1,1,1)}\\
&=\hat{\mathscr{S}}_{(3,2,1)}
+Q_{(3,1,1)}
+Q_{(1,2,1,1)}
+2Q_{(2,1,1)}
+Q_{(2,2)}
+Q_{(1,2,1)}
+s_{(1,1,1)}
\end{align*}
\end{example}

We  already know by Theorem~\ref{thm:SMC bijection} that the  elements of  $\MR{L_{m,n}}$ associated to shortest maximal chains are in bijection with reduced words of $w_{m,n}=[m+n,m-1,m-2,\ldots, 1, m, m+1, \ldots, m+n-1]$. 
Let us call this set $\SMR{L_{m,n}}$  for easy reference. We now will explicitly describe this bijection $\varphi:\cR(w_{m,n})\rightarrow \SMR{L_{m,n}}$. Our goal is to show that $\Des(\sigma)=\Des(\varphi(\sigma))$ for any $\sigma\in\cR(w_{m,n})$. As a result, the portion of $\cF_{L_{m,n}}$ associated to shortest maximal chains will be Schur positive. 

The explicit description of the bijection
$$\psi:\cR(w_{m,n})\rightarrow \SMR{L_{m,n}}$$
is as follows. In Section~\ref{sec:SMC}, we define the bijection $\varsigma:\SMB{L_{m,n}}\rightarrow \cR(w_{m,n})$. In the proof of Theorem~\ref{thm:varsigma bijection} we argue from $\sigma\in\cR(w_{m,n})$ we can find $I_1$, 
the collection of indices on which the unique $m+n-1, m+n-2, \ldots, m$ appear and 
on the remaining indices $I_2$ appears a reduced word $\rho\in\cR(\longestk{m})$. 
Map $\sigma_i$ for $i\in I_1$ to $(m+n-\sigma_i, m+n-\sigma_i+1,\ldots, m+n)$. For $i\notin I_1$ there will be $l$ numbers in $I_1$ less than $i$. In this case map $\sigma_i$ to $(\sigma_i+l,\sigma_i+l+1)$. See Figure~\ref{fig:smc example} for an example. 

\begin{theorem}
Let $m\geq 1$ and $n\geq 0$. The map $\psi:\cR(w_{m,n})\rightarrow \SMR{L_{m,n}}$ is a bijection and $\Des(\sigma)=\Des(\psi(\sigma))$ for all $\sigma\in\cR(w_{m,n})$. 
\label{thm:R to SMR bijection}
\end{theorem}

\begin{proof}
Consider the diagram on the left in Figure~\ref{fig:smc example}. Let $\sigma\in\cR(w_{m,n})$, $\gamma=\psi(\sigma)$, $B=\varsigma^{-1}(\sigma)$, $T$ be the balanced tableau of shape $\stair_{m+n}$ described as in Lemma~\ref{lem:shortest balanced inverse} associated to $B$ (we will call this the {\it lift} of $B$)   and $\tau$ be the reduced word of $\longestk{m+n}$ associated to $T$. 
We will show that the diagram commutes by considering what happens each time we add in a transposition in the reduced word $\sigma\in\cR(w_{m,n})$. Let $B^{(k)}$ be $B$ restricted to the numbers in $[k]$, $\sigma^{(k)}=\sigma_1\sigma_2\cdots \sigma_k$, $T^{(k)}$ be the associated lift of $B^{(k)}$ and $\gamma^{(k)}=\gamma_1\gamma_2\cdots \gamma_k$. Let $\tau^{(k)}$ be the portion of the reduced word of $\tau$ associated to $T^{(k)}$. We will show that the diagram commutes, thus showing that $\psi$ is well defined by showing two things. First that $\tau^{(k)}$ and $\gamma^{(k)}$ are essentially the same reduced word when we replace each cycle of $\tau^{(k)}$ with its associated unique reduced word. Second, that the permutation $w^{(k)}$ equal to $\gamma^{(k)}$ is a $m$-132 avoiding permutation, thus showing that $\gamma^{(k)}$ is indeed an element of $\MR{L_{m,n}}$ by Proposition~\ref{prop:max's and min's of L(G)}. 
We will do this by inducting on $k$. 

Let $\sigma\in\cR(w_{m,n})$. In this proof  we will need to consider the structure of $\sigma$, so we will first make some observations. Let $I_1$, $I_2$ and $\rho$ be defined from $\sigma$ as they are defined in $\varsigma$ in Section~\ref{sec:SMC}. Consider how $\sigma$  turns the identity into $w_{m,n}$. 
The portion of $\sigma$ associated to $I_1$ turns $m,m+1,\ldots, m+n$ in the identity into $m+n, m,m+1,\ldots, m+n-1$. 
The portion of $\sigma$ associated to $I_2$ turns $1,2,\ldots, m-1, m+n$ in the identity into $m+n, m-1,m-2,\ldots, 1$. 
Let $s^{(k)}$ be the permutation associated to 
$\sigma^{(k)}$, $\rho^{(k)}$ be the portion of $\rho$ in  $\sigma^{(k)}$ and $p^{(k)}$ be the associated permutation of $[m]$. 
By our observation, $p^{(k)}$ is equal to  $s^{(k)}$ if we restrict $s^{(k)}$ to $\{1,2,\ldots, m-1, m+n\}$ and replace the $m+n$ with $m$. 
Also note that in the indices $[m]$ of $s^{(k)}$ we either have the numbers  $\{1,2,\ldots, m-1, m\}$ with $m$ at index $m$ or the numbers $\{1,2,\ldots, m-1, m+n\}$. 
This makes $s^{(k)}$ equal to $p^{(k)}$ in the first $m$ indices when $s^{(k)}(m)=m$ or $s^{(k)}$ equal to $p^{(k)}$ in the first $m$ indices if we replace $m+n$ with $m$ when $s^{(k)}(m)\neq m$. 
This means we can essentially work with the first $m$ numbers of $s^{(k)}$ and $p^{(k)}$  interchangeably. 

Finally, we include two additional inductive assumptions. Let $l^{(k)}$ be the number of indices in $I_1$ at most $k$. We will show that $w^{(k)}$ starts with $m+n,m+n-1,\ldots, m+n-l^{(k)}+1$, ends with $m+1,m+2,\ldots, m+n-l^{(k)}$ and in the middle there is $p^{(k)}$. From this structure we automatically have that $w^{(k)}$ is $m$-132 avoiding.  

The base case is when $k=0$. The tableaux $B^{(0)}$ and $T^{(0)}$ are empty and are naturally associated to the empty words $\sigma^{(0)}$ and $\gamma^{(0)}$. The diagram commutes for $k=0$. Also $w^{(0)}=12\cdots (m+n)$ is the identity permutation and is $m$-132 avoiding. Also $l^{(0)}=0$, $p^{(0)}=[1,2,\ldots, m-1]$ and certainly $w^{(0)}$ is starting with the empty word, ending in $m+1,m+2, \ldots, m+n$ and in the middle we have $p^{(0)}$. 

Now suppose that $k> 0$. Suppose the diagram commutes for $k-1$ and  $w^{(k-1)}$ starts with $m+n,m+n-1,\ldots, m+n-l^{(k-1)}+1$, ends with $m+1,m+2,\ldots m+n-l^{(k-1)}$ and in the middle has $p^{(k-1)}$. 

The first case is when $k\in I_1$.
According to our definition of $\psi$, we have that $\gamma^{(k)}=
\gamma^{(k-1)}(m+n-\sigma_k,m+n, m+n-1,\ldots,m+n-\sigma_k+1)$
If we can show that that the reduced word $\tau^{(k)}=\tau^{(k-1)}(m+n-1)(m+n-2)\cdots (m+n-\sigma_k)$, then the diagram commutes. 
We know by assumption that $l^{(k)}=l^{(k-1)}+1$ and $p^{(k)}=p^{(k-1)}$, so  $w^{(k-1)}$
starts with $m+n,m+n-1,\ldots, m+n-l^{(k)}+2$, ends with $m+1,m+2,\ldots, m+n-l^{(k)}+1$ and has $p^{(k)}$ in the middle. Suppose that $k$ in $B$ is in cell $(r,c)$.  
Let us lift $B^{(k)}$ to make  $T^{(k)}$ as implied by how we would like $B$ to make $T$ as described in Lemma~\ref{lem:shortest balanced inverse}. We know that column $c$ of $T^{(k)}$ is filled with the $r$ highest numbers in $T^{(k)}$. 
This means when going from $w^{(k-1)}$ to $w^{(k)}$ we don't have an inversion $(i,r+1)$ in $w^{(k-1)}$ but we will  have an inversion  $(i,r+1)$ in $w^{(k)}$ for all $i\leq r$. This means we are bringing $r+1$, which is right of all $i\in [r]$ in $w^{(k-1)}$, to the left of all $i\in [r]$ to form $w^{(k)}$. 
Note that $c=l^{(k)}$ and $c=m+n-r$. We know that $w^{(k-1)}$
starts with $m+n,m+n-1,\ldots, r+2$, ends with $m+1,m+2,\ldots, r+1$ and has $p^{(k)}$ in the middle. To form $w^{(k)}$ we move the $r+1$ at index $m+n$ left until it moves past all $i\in [r]$ (i.e. $p^{(k)}$) , which are in indices $[a^{(k)},m+a^{(k)}]$ in $w^{(k-1)}$. 
This is describing  concatenating the reduced word $(m+n-1)(m+n-2)\cdots r$ to the end of $\tau^{(k)}$, 
which is the cycle 
$(r,m+n,m+n-1,\ldots, r+1)=(m+n-\sigma_k,m+n,m+n-1,\ldots, m+n-\sigma_k+1)$ since $r=m+n-\sigma_k$. It follows that $w^{(k)}$ will start with $m+n,m+n-1,\ldots, m+n-l^{(k)}+1$, end with $m+1,m+2,\ldots, m+n-l^{(k)}$ and will still have $p^{(k)}$ in the middle. Thus this case is complete. 

The second case is if $k\notin I_1$.
According to our definition of $\psi$, we have that $\gamma^{(k)}=\gamma^{(k-1)}(\sigma_k+l^{(k)},\sigma_k+l^{(k)}+1)$. If we can show that that the reduced word $\tau^{(k)}=\tau^{(k-1)}(\sigma_k+l^{(k)})$, then the diagram commutes. 
We know by our assumptions that $l^{(k)}=l^{(k-1)}$ and  $p^{(k)}=p^{(k-1)}\sigma_k$.
We also know that $w^{(k-1)}$
starts with $m+n,m+n-1,\ldots, m+n-l^{(k)}+1$, ends with $m+1,m+2,\ldots, m+n-l^{(k)}$ and has $p^{(k-1)}$ in the middle. 
Suppose that $k$ in $B$ is in cell $(r,c)$. We know that $r\leq m-1$ and $c\geq n+1$.
Suppose that the lift of $B^{(k)}$ is $T^{(k)}$ and the lift of $B^{(k-1)}$ is $T^{(k-1)}$ just like we lift $B$ to form $T$. The only difference between $B^{(k-1)}$ and $B^{(k)}$ is a $k$ in some cell $(r,c)$. 
This means the only difference between $T^{(k-1)}$ and $T^{(k)}$ is some $j$ in  cell $(r,c)$. 
This means in $p^{(k-1)}$ we don't have the inversion $(r,m+n-c+1)$ but in $p^{(k-1)}$ we have the inversion $(r,m+n-c+1)$. 
Further, this inversion is in the indices  $\sigma_k,\sigma_{k}+1$. To form  $w^{(k)}$ from $w^{(k-1)}$ we will also be forming the inversion $(r,m+n-c+1)$. 
According to our inductive assumptions since $r,m+n-c+1$ occur at indices $\sigma_k,\sigma_{k}+1$ in $p^{(k-1)}$ we know that $r,m+n-c+1$ occur at indices $\sigma_k+l^{(k)},\sigma_{k}l^{(k)}+1$ in $w^{(k-1)}$. 
This means that $\tau^{(k)}=\tau^{(k-1)}(\sigma_k+l^{(k)})$ as reduced words. This matches $\gamma^{(k)}=\gamma^{(k-1)}(\sigma_k+l^{(k)},\sigma_k+l^{(k)}+1)$ as the map $\psi$ described. It is easy to also see that $w^{(k)}$ starts with $m+n,m+n-1,\ldots, m+n-l^{(k)}+1$, ends with $m+1,m+2,\ldots, m+n-l^{(k)}$ and has $p^{(k)}$ in the middle, which completes out inductive argument. 

Finally, we will show that $\Des(\sigma)=\Des(\psi(\sigma))$ for all $\sigma\in\cR(w_{m,n})$. Suppose that $\sigma\in\cR(w_{m,n})$ and $\gamma = \psi(\sigma)$ with $\gamma=(a_1,b_1,b_1-1,\ldots,a_1+1)(a_2,b_2,b_2-1,\ldots,a_2+1)\ldots(a_\ell,b_\ell,b_{\ell}-1,\ldots,a_\ell+1)$. We will show that $i\in\Des(\sigma)$ if and only if  $i\in\Des(\gamma)$ for all $i$.

Let $I_1$, $I_2$ and $\rho\in\cR(\longestk{m})$ be defined from $\sigma$ as is done in the proof of Theorem~\ref{thm:varsigma bijection}. Let $l^{(i)}$ equal the number of elements in $I_1$ less than $i$. 

Our first case is when $i,i+1\in I_1$, then $\sigma_i\sigma_{i+1}$ is part of the decreasing subword $m+n-1, m+n-2, \ldots, m$, so $i\in\Des(\sigma)$. By the definition of $\psi$ we have that $\gamma_i\gamma_{i+1}=
(m+n-\sigma_i,m+n,m+n-1,\ldots, m+n-\sigma_i+1)
(m+n-\sigma_{i+1}, m+n,m+n-1,\ldots, m+n-\sigma_{i+1}+1)$. This means $b_i=b_{i+1}$ and $i\in \Des(\gamma)$. 

Our second case is when $i,i+1\in I_2$. This means that $l^{(i)}=l^{(i+1)}=l$ and $\sigma_i\sigma_{i+1}$ maps to $(\sigma_i+l,\sigma_i+l+1)(\sigma_{i+1}+l,\sigma_{i+1}+l+1)$ so $b_ib_{i+1}=(\sigma_i+l+1)(\sigma_{i+1}+l+1)$. Clearly $b_i>b_{i+1}$ if and only if $\sigma_i>\sigma_{i+1}$, so $i\in \Des(\sigma)$ if and only if $i\in \Des(\gamma)$. 

Our third case is when $i\in I_1$ and $i+1\in I_2$. By our study of the map $\varsigma$ we know we will have $\sigma_i\geq m$ and $\sigma_{i+1}<m$. Hence $i\in \Des(\sigma)$. 
By our map $\psi$ we have that $\gamma_i\gamma_{i+1}=
(m+n-\sigma_i, m+n,m+n-1,\ldots, m+n-\sigma_i+1)
(\sigma_{i+1}+l^{(i+1)},\sigma_{i+1}+l^{(i+1)}+1)$ so $b_ib_{i+1}=(m+n)(\sigma_{i+1}+l^{(i+1)}+1)$. Clearly we also have $i\in\Des(\gamma)$. 

Our last case is when $i\in I_2$ and $i+1\in I_1$. By our study of the map $\varsigma$ we know we will have $\sigma_i< m$ and $\sigma_{i+1}\geq m$. Hence $i\notin \Des(\sigma)$. By our map $\psi$ we have that $\gamma_i\gamma_{i+1}=
(\sigma_{i}+l^{(i)},\sigma_{i}+l^{(i)}+1)(m+n-\sigma_{i+1}, m+n,m+n-1,\ldots, m+n-\sigma_{i+1}+1)$ so $b_ib_{i+1}=(\sigma_{i}+l^{(i)}+1)(m+n)$. Clearly we also have $i\notin\Des(\gamma)$. 
\end{proof}

We summarize what we have shown in this section. 

\begin{corollary}
For $m\geq 1$ and $n\geq 0$   the quasisymmetric function $\cF_{L_{m,n}}$ 
\begin{enumerate}[(i)]
\item Has degrees ranging from $\binom{m}{2}+n$ to $\binom{m+n}{2}$. 
\item The homogeneous degree $\binom{m+n}{2}$ portion is a positive sum of Young quasisymmetric Schur functions.
\item The homogeneous degree $\binom{m}{2}+n$ portion is a positive sum of Schur symmetric functions. 
\end{enumerate}
\end{corollary}

\begin{proof}
Parts {\it(ii)} and {\it(iii)} follow from Theorem~\ref{thm:staley symmetric function}, Theorem~\ref{thm:Longest chain quasi Schur} and Theorem~\ref{thm:SMC bijection}. Part {\it(i)} is an observation on range of degree resulting from the degrees we see from shortest maximal chains in Theorem~\ref{thm:SMC bijection}, which is the length of $w_{m,n}$, and longest maximal chains in Theorem~\ref{thm:Longest chain quasi Schur}, which is the length of $\longestk{m+n}$.
\end{proof}

\printbibliography

\end{document}